\documentclass[a4paper,reqno,11pt]{amsart}
\usepackage{amsfonts,amsmath,amsthm,amssymb,stmaryrd}
\usepackage{color}

\allowdisplaybreaks[4]

\newtheorem{Theorem}{Theorem}[section]

\newtheorem{Lemma}[Theorem]{Lemma}
\newtheorem{Proposition}{Proposition}[section]
\newtheorem{Remark}{Remark}[section]

\newcommand{\beq}{\begin{equation}}
\newcommand{\eeq}{\end{equation}}
\newcommand{\ben}{\begin{eqnarray}}
\newcommand{\een}{\end{eqnarray}}
\newcommand{\beno}{\begin{eqnarray*}}
\newcommand{\eeno}{\end{eqnarray*}}

\newcommand{\pt}{\partial_{t}}
\newcommand{\px}{\partial_{x}}
\newcommand{\py}{\partial_{y}}
\newcommand{\pz}{\partial_{z}}

\newcommand{\ty}{\infty}

\newcommand{\R}{\mathbb{R}}
\newcommand{\N}{\mathbb{N}}

\newcommand{\T}{\mathbb{T}}

\newcommand{\de}{\delta}

\newcommand{\la}{\lambda}
\newcommand{\ga}{\gamma}
\newcommand{\al}{\alpha}
\newcommand{\si}{\sigma}

\newcommand{\ep}{\epsilon}

\newcommand{\Om}{\Omega}
\newcommand{\om}{\omega}
\newcommand{\ue}{u_{\epsilon}}
\newcommand{\ve}{v_{\epsilon}}

\newcommand{\pe}{p_{\epsilon}}

\newcommand{\fs}{\frac{1}{2}}
\newcommand{\ome}{\omega_{\epsilon}}
\newcommand{\Hg}{H^{s, \gamma}_g}
\newcommand{\aep}{a_{\epsilon}}
\newcommand{\gse}{g^s_{\epsilon}}

\numberwithin{equation}{section} \allowdisplaybreaks
\usepackage[left=1 in, right=1 in,top=1 in, bottom=1 in]{geometry}
\begin{document}
\title[Unsteady boundary layer for dilatant power-law fluids]{\bf Local Well-Posedness of the Unsteady Boundary Layer System for Dilatant Power-Law Fluids}
\author{Mingxue Zhang$^{1}$}
\author{Zhonger Wu$^{2,*}$}
\thanks{$^{1}$ Institute for Math and AI, Wuhan University, Wuhan, 430072, China }
\thanks{$^{2}$ Department of Mathematics, Shantou University, Shantou 515063, China }
\thanks{$^{*}$Corresponding author: Zhonger Wu, wze622520@163.com}
\thanks{Mingxue Zhang, zhangmingxuex@163.com}

\begin{abstract}

We establish local-in-time existence and uniqueness of monotone classical solutions to the two-dimensional nonstationary boundary
layer equations for dilatant power-law fluids in a periodic half-space, for every fixed exponent $n>1$.
The result holds under Oleinik's monotonicity condition and suitable weighted Sobolev, algebraic decay, and compatibility assumptions.
The main difficulties are the tangential derivative loss caused by the normal velocity and the degeneracy at infinity of the
principal coefficient $n\omega^{n-1}$ in the nonlinear vorticity diffusion $\partial_y^2(\omega^n)$, where $\omega=\partial_y u$.
We adapt the Masmoudi-Wong energy method and use a good unknown to control the tangential derivative loss.
For the nonlinear diffusion, we separate the estimates near the boundary from the treatment of degeneracy in the far field.
Nonlinear boundary reduction and localized trace estimates control the boundary terms, while weighted estimates exploit the compensating decay of vorticity derivatives to control the diffusion commutators.
Together with tangential regularization and weighted maximum and minimum principles, these estimates yield bounds uniform in the regularization parameter and preserve weighted monotonicity.
Under these assumptions, the result provides an affirmative local-in-time answer to the existence and uniqueness question in Open Problem~11 of Oleinik and Samokhin.


\end{abstract}
\maketitle

\section{Introduction}
We study the well-posedness of the boundary layer equations for a dilatant power-law fluid in the periodic domain
$\{(t,x,y)\mid t\in[0,T],\ x\in\T,\ y\in\R_+\}$. The system reads 
\begin{equation}\label{feiniu}
\left\{\begin{array}{l}
\pt u +u\px u +v\py u-\py(|\py u|^{n-1}\py u)+\px p=0,\\
\px u+\py v=0,\\
u|_{t=0}=u_0,\\
u|_{y=0}=v|_{y=0}=0,\quad
\lim\limits_{y\rightarrow\ty}u=U.
\end{array}\right.
\end{equation}
Here $(u,v)$ denotes the velocity field. Moreover, $U$ and $p$ are the traces on the boundary of the tangential velocity and the pressure associated with the outer flow, respectively, and satisfy Bernoulli's law:
\begin{equation}\label{Bernoulli}
\pt U +U\px U +\px p=0.
\end{equation}
The rheological behavior of the fluid depends on the power-law index $n$: the fluid is pseudo-plastic for $0<n<1$, dilatant for $n>1$, and Newtonian when $n=1$. In the Newtonian case, \eqref{feiniu} reduces to the classical Prandtl system. Throughout this paper, we investigate the boundary layer problem for dilatant fluids.

The Prandtl equations, introduced by Prandtl in 1904 to describe the discrepancy between ideal fluid motion and low viscosity flow near a solid boundary, have since been studied extensively. A central analytical obstacle is the loss of one tangential derivative. In two space dimensions, Oleinik and Samokhin \cite{OAO} established the local existence and uniqueness of classical solutions under a monotonicity assumption by means of the Crocco transform. Under the additional favourable-pressure condition, Xin and Zhang \cite{ZPX} subsequently obtained global weak solutions. In 2015, \cite{RA} and  \cite{NM} independently proved local well-posedness in weighted Sobolev spaces through energy methods. 
These results, together with the Sobolev ill-posedness mechanism identified by G\'erard-Varet and Dormy \cite{DG2010}, demonstrate the importance of monotonicity in the Sobolev theory.
Without monotonicity, well-posedness has instead been developed in analytic spaces \cite{MI,MCL,MP,MS,PZ} and in Gevrey spaces \cite{DC,HD,DG2015,WXLY2,CW}.

The three-dimensional problem is more involved because of the secondary flow, and monotonicity by itself no longer guarantees Sobolev well-posedness. The results in \cite{CJLARMA,CJLADV} show that the structural condition $\pz(\dfrac{u}{v})\equiv0$ and $\pz u>0$ is almost necessary and sufficient for well-posedness in Sobolev spaces. In the absence of such a structural assumption, a well-posedness result in the Gevrey class of index $2$ was obtained in \cite{WXLMY}.

Related developments also clarify how additional physical effects alter boundary layer dynamics. For results concerning magnetic effects, we refer to \cite{WXLY,CJLCPAM,ZT,ZEW, ZW}; for thermal effects, see \cite{ZT2,YGW1,YGW2}.

In contrast with the extensive Newtonian theory, rigorous results for non-Newtonian boundary layer equations remain limited. For pseudoplastic fluids, Oleinik and Samokhin \cite[Chapter~8]{OAO} used the Crocco transformation to establish the local existence and uniqueness of classical solutions to the two-dimensional nonstationary boundary-layer system. Using a related transformation, Zhang \cite{JWZ} obtained a corresponding result for a special class of outer flows. More recently, Wu and Tan \cite{WZE} proved well-posedness in the pseudoplastic regime $\frac{1}{3}<n<1$ by energy methods under Oleinik's monotonicity condition. 
For dilatant fluids, the results in Chapter~8 of \cite{OAO} concern only stationary boundary layer problems.
In the comments at the end of that chapter, Oleinik and Samokhin identify the solvability of the corresponding nonstationary system as an unresolved question.
This question is explicitly stated as Open Problem~11 in \cite[p.~502]{OAO}:
\begin{quote}
``The existence and the uniqueness of a solution of the nonstationary system of boundary layer in dilatable fluids.''
\end{quote}
Subsequently, Said \cite{Said} established local well-posedness of monotone solutions to the nonstationary boundary layer system in the particular case $n=2$.

In this paper, we establish local-in-time existence and uniqueness of monotone classical solutions to the two-dimensional nonstationary boundary layer system for dilatant power-law fluids, for every fixed exponent $n>1$.
As stated in Theorem~\ref{zhuding}, the result holds in a periodic half-space under Oleinik's monotonicity condition and suitable weighted Sobolev, algebraic decay, and compatibility assumptions on the initial data and outerflow. 
No positive global lower bound for the unweighted vorticity is assumed, and the normal diffusion is allowed to degenerate at infinity. Thus, within the stated periodic and weighted monotone solution class, our result provides an affirmative local-in-time answer to the existence and uniqueness question in Open Problem~11.

Under the monotonicity assumption $\py u>0$, system \eqref{feiniu} can be rewritten as
\begin{equation}\label{feiniu2}
\left\{\begin{array}{l}
\pt u +u\px u +v\py u-n(\py u)^{n-1}\py^2 u+\px p=0,\\
\px u+\py v=0,\\
u|_{t=0}=u_0,\\
u|_{y=0}=v|_{y=0}=0,\quad
\lim\limits_{y\to+\infty}u =U,
\end{array}\right.
\end{equation}
where the initial data $u_0=u_0(x,y)$ and the outer flow
$U=U(t,x)$ are given and satisfy the zeroth-order
compatibility conditions
\begin{equation}\label{initial-matching}
 u_0(x,0)=0,\qquad
 \lim_{y\to+\infty}u_0(x,y)=U(0,x).
\end{equation}

To construct approximate solutions and establish the required a priori estimates, we introduce a regularized system by adding a small tangential viscosity term. This regularization provides additional parabolic
regularity in the $x$ direction while preserving the essential structure of the original boundary layer equations.
As usual, the smooth compatible regularization used in the approximation argument is still denoted by $u_0$, with its vorticity denoted by $\omega_0$. More precisely, for $\epsilon\in(0,1]$, we consider the following regularized problem
\begin{equation}\label{feiniu2-1}
\left\{\begin{array}{l}
\pt \ue +\ue\px \ue +\ve\py \ue-\epsilon^2\px^2\ue-n(\py \ue)^{n-1}\py^2 \ue+\px \pe=0,\\
\px \ue+\py \ve=0,\\
\ue|_{t=0}=u_0^\ep,\\
\ue|_{y=0}=\ve|_{y=0}=0,\quad
\lim\limits_{y\to+\infty}\ue =U,
\end{array}\right.
\end{equation}
where $U$ and $\pe$ satisfy the regularized Bernoulli's law
\begin{equation}\label{BL}
\pt U +U\px U -\epsilon^2\px^2U+\px \pe=0.
\end{equation}
Equivalently, the vorticity $\ome:=\py \ue>0$ satisfies

\begin{equation}\label{wodu}
\left\{\begin{array}{l}
\pt \ome +\ue\px \ome +\ve\py \ome-\epsilon^2 \px^2\ome-n\ome^{n-1}\py^2 \ome-n(n-1)
\ome^{n-2}|\py\ome|^2=0,\\
\ome|_{t=0}:=\om_{0}=\py u_0,\\
n\ome^{n-1}\py\ome|_{y=0}=\px \pe,\quad
\lim\limits_{y\rightarrow\ty}\ome=0,
\end{array}\right.
\end{equation}
where the velocity field $(\ue,\ve)$ is given
\begin{equation}\label{uvdy}
\begin{aligned}
\ue(t,x,y):=U-\int_y^{\ty}\ome(t,x,z)\ dz,\quad 
\ve(t,x,y):=-\int_0^y \px \ue(t,x,z)\ dz.
\end{aligned}
\end{equation}
Set $\Omega:=\mathbb T\times\mathbb R_+$. We use the following weighted monotonicity class:
$$
\begin{aligned}
H_{\sigma, \delta}^{s, \gamma}:= & \bigg\{\omega: \Om \rightarrow \mathbb{R}:\|\omega\|_{H^{s, \gamma}}<+\infty,(1+y)^\sigma \omega \geq \delta, \\
& \text { and } \sum_{|\alpha| \leq 2}\left|(1+y)^{\sigma+\alpha_2} D^\alpha \omega\right|^2 \leq \frac{1}{\delta^2}\bigg\},
\end{aligned}
$$
where $s\geq6,~\ga\geq1,~\si>\ga+\frac{1}{2},~\de\in(0,1),~D^\alpha=\px^{\al_1}\py^{\al_2}$. Apart from that, the weighted norm $\|\cdot\|_{H^{s, \gamma}}$ denotes
\begin{equation}\label{hsga}
\|\omega\|_{H^{s, \gamma}}^2:=\sum_{|\alpha| \leq s}\left\|(1+y)^{\gamma+\alpha_2} D^\alpha \omega\right\|_{L^2\left(\Om\right)}^2.
\end{equation}
Correspondingly, we also denote
$$
H^{s, \gamma}:=\left\{\omega: \Om \rightarrow \mathbb{R}:\|\omega\|_{H^{s, \gamma}}<+\infty\right\}.
$$
\begin{Remark}
The weights describe the decay as $y\to+\infty$. The lower bound
$\omega\geq\delta(1+y)^{-\sigma}$ is compatible with
$(1+y)^\gamma\omega\in L^2(\Omega)$ only if
$\sigma>\gamma+\frac12$. Hence, $H_{\sigma,\delta}^{s,\gamma}$
is empty when $\sigma\leq\gamma+\frac12$.
\end{Remark}

We can now state the main result.

\begin{Theorem}[Local existence and uniqueness for \eqref{feiniu2}]
\label{zhuding}
Let $n>1$, $s\geq 6$ be an even integer, and assume that $ \gamma\geq1$, $\gamma+\frac12<\sigma\leq\gamma+\frac43$,
$\delta\in\left(0,\frac12\right)$.
Suppose that the outer flow $(U,p)$ is periodic in $x$ and satisfies \eqref{Bernoulli}, with $U>0$ and
\begin{equation}\label{outer-flow-regularity}
\sup_{t\geq0}
\sum_{\ell=0}^{\lfloor\frac{s+9}{2}\rfloor}\|\partial_t^\ell U(t)\|_{W^{s-2\ell+9,\infty}(\mathbb T)}
 <+\infty.
\end{equation}
Let the initial data $u_0$ satisfy the compatibility conditions \eqref{initial-matching}, with
$u_0-U(0)\in H^{s,\gamma-1}$,
$ \omega_0:=\partial_yu_0\in H_{\sigma,2\delta}^{s,\gamma}.$
Assume furthermore that the initial data and the outer flow satisfy the compatibility conditions up to order $\frac{s}{2}-1$.

Then there exists a time
$ T=T\left(s,\gamma,\sigma,n,\delta, \|\omega_0\|_{H_g^{s,\gamma}},U \right)>0$
such that the boundary layer problem \eqref{feiniu2} admits a unique classical solution $(u,v)$ on $[0,T]$ satisfying
$$ u-U\in L^\infty([0,T];H^{s,\gamma-1}) \cap C([0,T];H^s\text{-}w).$$
Moreover, the associated vorticity $\omega:=\partial_yu$ satisfies
$$\omega\in L^\infty([0,T];H_{\sigma,\delta}^{s,\gamma}) \cap C([0,T];H^s\text{-}w).$$
\end{Theorem}

\begin{Remark}\label{rem:initial-boundary-compatibility}
At $y=0$, \eqref{feiniu2} yields
$$\left.\partial_y(\omega^n)\right|_{y=0}=\partial_xp.$$
Higher-order compatibility conditions follow by repeatedly differentiation in time, with the time derivatives of $\omega$
eliminated using the vorticity equation. For even $s\geq6$, these conditions are required through order $\frac s2-1$.
\end{Remark}

\begin{Remark}\label{rem:outer-flow}
The regularity assumption \eqref{outer-flow-regularity} is the same as in \cite[Theorem~2.2]{NM}. It is not expected to be
optimal: the weighted energy estimates require fewer derivatives of $U$, while the additional derivatives simplify the construction
of the regularized solutions. Moreover, \eqref{Bernoulli} determines $\partial_xp$ from $U$. The positivity of the outer flow is consistent with the monotonicity assumption, since
$U(t,x)=\int_0^\infty\omega(t,x,y)\,dy>0.$
\end{Remark}

\begin{Remark}\label{rem:newtonian-limit}
When $n=1$, the nonlinear diffusion reduces to $\partial_y(\omega^n)=\partial_y^2u$, and \eqref{feiniu2} becomes the classical Prandtl system. The good unknown and the weighted solution spaces coincide with those used in \cite{NM}.
No uniform estimate or convergence result as $n\to1^+$ is claimed.
\end{Remark}

The remainder of the paper is organized as follows. Section~\ref{zhunbei} introduces the weighted spaces and auxiliary inequalities. Section~\ref{zhumingti} establishes estimates for the tangentially regularized problem that are uniform in $\epsilon$. Section~\ref{jixian} constructs a solution by compactness and proves uniqueness. 
\section{Preliminaries}\label{zhunbei}
We first introduce the notation and estimates used below.
Use $\py^{-1}$ to denote the inverse of the derivative $\py$, that is $\py^{-1}f(y)=\int_0^y f(\tilde{y})d\tilde{y}.$
Use $C$ to denote a non-negative constant that may vary from line to line. Use $[a]$ to denote the greatest integer not exceeding $a$.

In order to solve the dilatant power-law fluid boundary layer equations \eqref{feiniu} in Sobolev spaces, we need to overcome the following two main difficulties:
\begin{itemize}
  \item \emph{Tangential derivative loss.} From the divergence-free condition and $u|_{y=0}=v|_{y=0}=0$, we know that $v=-\py^{-1}\px u$, which creates a loss of the $x$-derivative, hence we can't use the standard energy estimates.
  \item \emph{Degenerate nonlinear diffusion.} Specifically, $n\omega^{n-1}$ vanishes as $y\to+\infty$, while differentiating it may produce negative powers of $\omega$; the normal diffusion is therefore degenerate and its commutators require weighted estimates.
\end{itemize}
To address these two difficulties, we use the weighted monotonicity bounds defining $H_{\sigma,\delta}^{s,\gamma}$.

For the first difficulty, we refer to the cancellation mechanism proposed in \cite{NM} and define the following norm:
$$
\|\om\|_{H^{s,\gamma}_g}^2=
\sum_{\substack{|\al|\leq s \\ \al_1\leq s-1}}
\|(1+y)^{\gamma+\alpha_2}D^{\al}\om\|^2_{L^2}
+\|(1+y)^{\gamma}g^{s}\|^2_{L^2},
$$
where $g^{s}:=\px^{s}\om-a\px^{s}(u-U)$ and $a:=\frac{\py\om}{\om}$.

On the one hand, $g^{s}$ can avoid the loss of $x$-derivative; see Section \ref{zuoqiexiangguji} for details.
On the other hand, we know from \cite[Appendix A]{NM} that there exist positive constants $c_\de$ and $C_{s,\ga,\si,\de}$ such that
\begin{equation}\label{chabuduo}
\begin{split}
c_\de\|\om\|_{H^{s, \gamma}_g}
\leq \|u-U\|_{H^{s, \ga-1}}
+\|\omega\|_{H^{s, \gamma}}
\leq C_{s,\ga,\si,\de}(\|\om\|_{H^{s, \gamma}_g}+
\|\px^{s} U\|_{L^2(\T)}).
\end{split}
\end{equation}
Therefore, in order to overcome the first difficulty, we will estimate $\|\om\|_{H^{s, \gamma}_g}$ instead of $\|\omega\|_{H^{s, \gamma}}$.

For the second difficulty, we estimate the boundary and interior terms arising from the nonlinear diffusion separately.
On the fixed strip $\T\times(0,1)$, the weighted lower bound for $\omega$ gives $n\omega^{n-1}\ge n(\delta 2^{-\sigma})^{n-1}>0$.
We estimate the boundary terms using the localized trace inequality on this strip. Higher-order boundary terms are treated similarly after reducing the normal derivatives at the boundary. These estimates require no global bound on $(1+y)^{-\gamma}\omega^{\frac{1-n}{2}}$ and impose no restriction of the form $\sigma(n-1)\le2\gamma$.
To estimate the interior terms, we characterize the decay of $\omega^{n-1}$ and its higher-order derivatives as 
$y\rightarrow +\infty$. 
In particular, the definition of $H_{\sigma,\delta}^{s,\gamma}$ implies that 
$$
(1+y)^{\beta_2}\omega^{\frac{n-3}{2}}|D^\beta\omega|
\le C_{n,\sigma,\delta}(1+y)^{-\frac{\sigma(n-1)}{2}},\quad \text{for}~|\beta|=1.
$$
Thus, when $1<n<3$, the decay of $D^\beta\omega$ compensates for the growth of $\omega^{\frac{n-3}{2}}$ at infinity.
Higher-order products are controlled using the two-sided decay bounds for $\omega$ and weighted Sobolev estimates.
Under the stated weight conditions, the estimates in Section~\ref{zhumingti} hold for each fixed $n>1$. The constants and the existence time may depend on $n$.

To handle higher-order derivatives of $\om^{n-1}$, we need the following Fa\`a di Bruno formula; see \cite{JA}.
\begin{Lemma}[Fa\`a di Bruno formula]\label{faa}
For a multi-index $\al=(\al_1,\al_2)$ with $|\alpha|\geq1$ and a real number $m$, let $f(t,x,y)>0$ be sufficiently smooth. Then there holds
\begin{equation}\label{faa1}
D^\alpha (f^m) = \sum_{r=1}^{|\al|} m^{\underline{r}} \, f^{m-r} \sum_{\substack{\beta^1 + \cdots + \beta^r = \alpha \\ |\beta^i| \ge 1}} \frac{\alpha!}{r!} \prod_{i=1}^r \frac{D^{\beta^i} f}{\beta^i!},
\end{equation}
where:
\begin{itemize}
  \item $m^{\underline{r}}=m(m-1)\cdots(m-r+1)$ is the falling factorial (this term vanishes when $r>m$ if $m$ is an positive integer.)
  \item $\al!=\al_1!\al_2!$, and similarly for $\beta^i!$.
  \item The inner sum runs over all ordered $r$-tuples $(\beta^1,\cdots,\beta^r)$ of non-zero multi-indices whose sum equals $\al$.
\end{itemize}
\end{Lemma}

Moreover, for $\om\in H^{s,\gamma}_{\si,\de}$, we have the following decay property of $D^{\al}\om$ as $y$ goes to $+\ty$, which is similar to \cite[Remark C.4]{NM}.
\begin{Proposition}\label{shuaijian}
Let $s\geq5$ be an integer, $\ga\geq1,~\si>\ga+\frac{1}{2}$ and $\de\in(0,1)$. If $\om\in H^{s,\gamma}_{\si,\de}$, there exists a constant $C>0$, which depends on $s,~\ga,~\si$ and $\de$, such that for all $|\al|\leq s-2$,
\begin{equation}\label{shuaijian1}
|D^{\al}\om|\leq Cc_\alpha(1+y)^{-b_{\al}}\quad\text{in}~\T\times\R_+
\end{equation}
where
\begin{equation}\label{shuaijian2}
b_\alpha :=
\begin{cases}
\sigma + \alpha_2 & \text{if } |\alpha| \leq 2, \\
\frac{(s-2 - |\alpha|)\sigma + (|\alpha| - 2)\gamma}{s-4} + \alpha_2 & \text{if } 3 \leq |\alpha| \leq s-3, \\
\gamma + \alpha_2 & \text{if } |\alpha| = s-2,
\end{cases}
\end{equation}
and
\begin{equation}\label{shuaijian3}
c_\alpha :=
\begin{cases}
1 & \text{if } |\alpha| \leq 2, \\
1+\|\omega\|_{H^{s, \gamma}} & \text{if } 3 \leq |\alpha| \leq s-3, \\
\|\omega\|_{H^{s, \gamma}} & \text{if } |\alpha| = s-2.
\end{cases}
\end{equation}
\end{Proposition}
The derivation of the coefficient $c_\alpha$ utilizes \cite[Lemma C.3]{NM} and the fact $2\sqrt{C_0C_2}\leq C_0+C_2$. And we omit the proof for brevity.
\begin{Remark}\label{shuaijian4}
To reduce the number of cases, we can also directly write it in the following form.
\begin{equation}\label{shuaijian5}
|D^{\al}\om|\leq C(1+\|\omega\|_{H^{s, \gamma}})(1+y)^{-b_{\al}}\quad\text{in}~\T\times\R_+
\end{equation}
where
\begin{equation}\label{shuaijian6}
b_\alpha :=
\begin{cases}
\sigma + \alpha_2 & \text{if } |\alpha| \leq 1, \\
\frac{(s-2 - |\alpha|)\sigma + (|\alpha| - 2)\gamma}{s-4} + \alpha_2 & \text{if } 2 \leq |\alpha| \leq s-2.
\end{cases}
\end{equation}
\end{Remark}

With Lemma \ref{faa} and Proposition \ref{shuaijian}, we are able to overcome the second difficulty.

Next, for the estimation of the $L^{\ty}$ norm of $D^{\al}\om,~|\al|\leq2$, we need the following two extremum principles. Their proofs are almost identical to those of \cite[Lemma E.1]{NM} and \cite[Lemma E.2]{NM}, so we omit them here.
\begin{Lemma}[Maximum Principle for Parabolic Equations]
\label{jida}
Let $\epsilon \geq 0$. If $H \in C([0, T]; C^2(\mathbb{T} \times \mathbb{R}_+)) \cap C^1([0, T]; C^0(\mathbb{T} \times \mathbb{R}_+))$ is a bounded function that satisfies the following
\[
\left\{ \partial_t + b_1 \partial_x + b_2 \partial_y - \epsilon^2 \partial_x^2 - b_3\partial_y^2 \right\} H \leq fH \quad \text{in } [0, T] \times \mathbb{T} \times \mathbb{R}_+,
\]
where the coefficients $b_1, b_2$, $b_3$ and $f$ are continuous and satisfy
\begin{equation}\label{eq:E.1}
b_3>0,\quad
\left\| \big(\frac{b_2}{1 + y},\frac{b_3}{(1+y)^2}\big) \right\|_{L^\infty([0, T] \times \mathbb{T} \times \mathbb{R}_+)} < +\infty \quad \text{and} \quad \|f\|_{L^\infty([0, T] \times \mathbb{T} \times \mathbb{R}_+)} \leq \lambda,
\end{equation}
then for any $t \in [0, T]$,
\begin{equation}\label{eq:E.2}
\sup_{\mathbb{T} \times \mathbb{R}_+} H(t) \leq
\max \left\{ e^{\lambda t} \|H(0)\|_{L^\infty(\mathbb{T} \times \mathbb{R}_+)},
\max_{\tau \in [0, t]} \left\{ e^{\lambda(t-\tau)} \|H(\tau)|_{y=0}\|_{L^\infty(\mathbb{T})} \right\} \right\}.
\end{equation}
\end{Lemma}
\begin{Lemma}[Minimum Principle for Parabolic Equations]
\label{jixiao}
Let $\epsilon \geq 0$. If $H \in C([0, T]; C^2(\mathbb{T} \times \mathbb{R}_+)) \cap C^1([0, T]; C^0(\mathbb{T} \times \mathbb{R}_+))$ is a bounded function with
\[
\kappa(t) := \min \left\{ \min_{\mathbb{T} \times \mathbb{R}_+} H(0), \min_{[0,t] \times \mathbb{T}} H\big|_{y=0} \right\} \geq 0
\]
and satisfies the differential inequality
\[
\left\{ \partial_t + b_1 \partial_x + b_2 \partial_y - \epsilon^2 \partial_x^2 - b_3\partial_y^2 \right\} H \geq fH
\]
where the coefficients $b_1, b_2$, $b_3$ and $f$ are continuous and satisfy \eqref{eq:E.1}. Then for any $t \in [0, T]$,
\begin{equation}\label{eq:E.3}
\min_{\mathbb{T} \times \mathbb{R}_+} H(t) \geq (1 - \lambda t e^{\lambda t}) \kappa(t).
\end{equation}
\end{Lemma}

Now, we present some commonly used inequalities, which can be found in \cite{CJLCPAM}, see also \cite{NM,CJX}.
\begin{Lemma}\label{jingchangyong}
For proper functions $f_1,f_2$, there holds:
\begin{description}
  \item[(i)] If $\lim\limits_{y\rightarrow+\ty}(f_1f_2)(x,y)=0$, then we have
\begin{equation}\label{L1}
\left|\int_{\T}(f_1f_2)|_{y=0}dx\right|
\leq \|\py f_1\|_{L^2(\Om)}\|f_2\|_{L^2(\Om)}
+\|f_1\|_{L^2(\Om)}\|\py f_2\|_{L^2(\Om)}.
\end{equation}
In particular, if $\lim\limits_{y\rightarrow+\ty}f_1(x,y)=0$, then
\begin{equation}\label{L1.5}
\left\|f_1|_{y=0}\right\|_{L^2(\T)}
\leq \sqrt{2}\|f_1\|^{\frac{1}{2}}_{L^2(\Om)}\|\py f_1\|^{\frac{1}{2}}_{L^2(\Om)}.
\end{equation}
  \item[(ii)] Let $\ga\in\R$ and an integer $s\geq3$, then for any $\al=(\al_1,\al_2)\in\N^2$ and
  $\tilde{\al}=(\tilde{\al}_1,\tilde{\al}_2)\in\N^2$, which satisfy
  $|\al|+|\tilde{\al}|\leq s$, there holds
\begin{equation}\label{L2}
\|(D^{\al}f_1\cdot D^{\tilde{\al}}f_2)(t,\cdot)\|_{L^2_{\gamma+\al_2+\tilde{\al}_2}(\Om)}
\leq C\|f_1(t)\|_{H^{s, \gamma_1}}\|f_2(t)\|_{H^{s, \gamma_2}},
\end{equation}
where $\ga_1,\ga_2\in\R$ and $\ga_1+\ga_2=\ga$.
  \item[(iii)] For any $\la>\frac{1}{2}$, $\tilde{\la}>0$, we have
\begin{equation}\label{L3}
\begin{aligned}
& \|(1+y)^{-\lambda}(\partial_y^{-1} f_1)(y)\|_{L_y^2(\mathbb{R}_{+})} \leq \frac{2}{2 \lambda-1}\|(1+y)^{1-\lambda} f_1(y)\|_{L_y^2(\mathbb{R}_{+})}, \\
& \|(1+y)^{-\tilde{\lambda}}(\partial_y^{-1} f_1)(y)\|_{L_y^{\infty}(\mathbb{R}_{+})} \leq \frac{1}{\tilde{\lambda}}\|(1+y)^{1-\tilde{\lambda}} f_1(y)\|_{L_y^{\infty}(\mathbb{R}_{+})}.
\end{aligned}
\end{equation}
And let $\gamma\in\R$ and an integer $s\geq3$, then for any
$\al=(\al_1,\al_2)\in\N^2$ and
$\tilde{\beta}=(\tilde{\beta}_1,0)\in\N^2$, which satisfy
$|\al|+|\tilde{\beta}|\leq s$, there holds
\begin{equation}\label{L4}
\|(D^{\al}f_1\cdot D^{\tilde{\beta}}
\py^{-1}f_2)(t,\cdot)\|_{L^2_{\gamma+\al_2}(\Om)}\leq
C \|f_1(t)\|_{H^{s,\gamma+\la}}\|f_2(t)\|_{H^{s,1-\la}}.
\end{equation}
In particular, for $\la=1$
\begin{equation}\label{L4.5}
\begin{aligned}
 \|(1+y)^{-1}(\partial_y^{-1} f_1)(y)\|_{L_y^2(\mathbb{R}_{+})}
&\leq 2\|f_1(y)\|_{L_y^2(\mathbb{R}_{+})}, \\
\|(D^{\al}f_1\cdot D^{\tilde{\beta}}
\py^{-1}f_2)(t,\cdot)\|_{L^2_{\gamma+\al_2}(\Om)}
&\leq
C \|f_1(t)\|_{H^{s,1+\gamma}}\|f_2(t)\|_{H^{s,0}}.
\end{aligned}
\end{equation}
\item[(iv)]For any $\la>\frac{1}{2}$, we have
\begin{equation}\label{L5}
\|(\py^{-1}f_1)(y)\|_{L_y^{\ty}(\R_+)}
\leq C\|f_1\|_{L^2_{y,\la}(\R_+)}.
\end{equation}
And let $\gamma\in\R$ and an integer $s\geq2$, then for any
$\al=(\al_1,\al_2)\in\N^2$ and
$\tilde{\beta}=(\tilde{\beta}_1,0)\in\N^2$, which satisfy
$|\al|+|\tilde{\beta}|\leq s$, there holds
\begin{equation}\label{L6}
\|(D^{\al}f_1\cdot D^{\tilde{\beta}}
\py^{-1}f_2)(t,\cdot)\|_{L^2_{\gamma+\al_2}(\Om)}\leq
C \|f_1(t)\|_{H^{s,\gamma}}\|f_2(t)\|_{H^{s,\la}}.
\end{equation}
\end{description}
\end{Lemma}
Moreover, we have the following calculus inequalities, which is \cite[Appendix B]{NM}
\begin{Lemma}\label{hardy}
Let $f: \Om\rightarrow \R$. Then \\
(i) if $\lambda>-\fs$ and $\lim\limits_{y\rightarrow +\infty} f(x,y)=0$, then
\begin{equation}\label{hd2}
\|(1+y)^\lambda f\|_{L^{2}(\Om)}\leq \frac{2}{2\lambda+1}\|(1+y)^{\lambda+1}\py f\|_{L^{2}(\Om)};
\end{equation}
(ii) if $\lambda<-\fs$, then
\begin{equation}\label{hd1}
\begin{aligned}
\|(1+y)^\lambda f\|_{L^{2}(\Om)}\leq
&\sqrt{-\frac{1}{2\lambda+1}}\left\|f|_{y=0}\right\|_{L^{2}(\T)}
-\frac{2}{2\lambda+1}\|(1+y)^{\lambda+1}\py f\|_{L^{2}(\Om)}.
\end{aligned}
\end{equation}
\end{Lemma}
\begin{Lemma}\label{chazhi}
Let $\phi~:~\T\times\R_+\rightarrow \R$. Then there exists a universal constant $\tilde{C}>0$ such that
\begin{equation}\label{chazhi1}
\|\phi\|_{L^{\ty}}\leq \tilde{C} (\|\phi\|_{L^2}
+\|\px\phi\|_{L^2}+\|\py^2\phi\|_{L^2}).
\end{equation}
\end{Lemma}

\begin{Lemma}[Trace estimate]\label{trace}
If $f\in H^1(\T\times(0,1))$, then
\begin{equation}\label{jiguji}
\begin{split}
\int_{\T}\Big| f|_{y=0}\Big| dx\leq 2\left(\int_0^1\int_{\T} |f| dxdy+\int_0^1\int_{\T} |\py f| dxdy\right).
\end{split}
\end{equation}
\end{Lemma}

\section{A Priori Estimates}\label{zhumingti}
In this section, we establish the a priori estimates for system \eqref{wodu}, which are required for the proof of Theorem \ref{zhuding}.

To justify the a priori analysis that follows, we first record the local solvability of the regularized system.  Indeed, the estimates below must be performed along a sufficiently smooth solution.  Following the approximation scheme of \cite[Section~7]{NM},
we construct such a solution for each fixed $\epsilon>0$ by a weighted quasilinear parabolic iteration.
Here the normal coefficient tends to zero at infinity, so the argument is made in weighted variables: use the normal coordinate $\frac{(1+y)^{1+\sigma(n-1)/2}-1}{1+\sigma(n-1)/2}$ and the unknown $(1+y)^\sigma\ome$. The principal coefficients become $\epsilon^2$ and $n\big((1+y)^\sigma\ome\big)^{n-1}$, uniformly positive for fixed $\epsilon$ in the prescribed positive interval. The transformed normal drift has at most linear growth because $\frac{\ve}{1+y}$ is bounded. Truncation and weighted parabolic iteration then apply, with the fixed tangential viscosity absorbing the derivative in the nonlocal velocity. No positive global lower bound for $n\ome^{n-1}$ in the original coordinate is assumed.
We summarize the result in the following lemma:


\begin{Lemma}[Local existence of the regularized system]\label{lemmalocal}
Let $s\geq6$ be an even integer, $n>1$, $\gamma\geq1$, $\gamma+\frac12<\sigma\leq\gamma+\frac43,~\delta\in(0,\frac12)$, and $\ep\in(0,1]$.
Assume that $U$ and $p_\epsilon$ satisfy the regularized Bernoulli's law \eqref{BL} and the regularity assumption \eqref{outer-flow-regularity}, and that the initial vorticity $\omega_0\in H_{\sigma,2\delta}^{s+12,\gamma}$.
Then there exist a time
$$ T_\epsilon=T_\epsilon\!\left(s,\gamma,\sigma,n,\delta,\epsilon, \|\omega_0\|_{H_g^{s+4,\gamma}},U \right)>0,$$
and a classical solution $(u_\epsilon,v_\epsilon,\omega_\epsilon)$ of \eqref{wodu}-\eqref{uvdy} on $[0,T_\epsilon]$, with
$$\omega_\epsilon\in C\bigl([0,T_\epsilon];H_{\sigma,\delta}^{s+4,\gamma}\bigr)
 \cap C^1\bigl([0,T_\epsilon];H^{s+2,\gamma}\bigr).$$

Furthermore, the regularized problem \eqref{feiniu2-1} is satisfied by the velocity $(u_\epsilon,v_\epsilon)$ defined by \eqref{uvdy}.
\end{Lemma}

With this local solution in hand, we fix $\epsilon\in(0,1]$ and derive estimates on a time interval independent of $\epsilon$.  These estimates will then allow the regularized solutions to be continued up to a common positive
time.

The argument is divided into three parts. First, when $|\alpha|\leq s$ and $\alpha_1\leq s-1$, no loss of tangential derivatives occurs, and the corresponding weighted $L^2$ estimates for $D^\alpha\omega_\epsilon$ follow from the standard energy method; this is carried out in Subsection \ref{zuofaxiangguji}.  Second, the remaining pure tangential derivative $\partial_x^s\omega_\epsilon$ cannot be estimated directly because $v_\epsilon=-\partial_y^{-1}\partial_xu_\epsilon$ contains one more $x$-derivative.  In Subsection \ref{zuoqiexiangguji}, this loss is removed by estimating the good unknown
$$g_\epsilon^s
 =\partial_x^s\omega_\epsilon-\frac{\partial_y\omega_\epsilon}{\omega_\epsilon}\partial_x^s(u_\epsilon-U).
$$
Finally, the two weighted $L^2$ estimates are combined to control $\|\omega_\epsilon\|_{H_g^{s,\gamma}}$, after which the maximum principle is applied to the weighted derivatives of order at most two and to
$(1+y)^\sigma\omega_\epsilon$.  This yields the uniform weighted $L^\infty$ bound and propagates the monotonicity lower bound.  These estimates close the continuation argument for the regularized solutions and provide the
uniform bounds needed for the compactness procedure in Section \ref{jixian}.


\subsection{Estimates with Normal Derivatives}\label{zuofaxiangguji}
First, we will deal with the weighted estimates for $D^{\al}\om$ with $|\al|\leq s$ and $\al_1\leq s-1$. The result reads as the following:
\begin{Proposition}
\label{guji1}
Let $s\geq6$ be an even integer, $n>1$,
$\gamma\geq1$, $\gamma+\frac12<\sigma\leq \gamma+\frac{4}{3}$, $\epsilon\in (0,1]$ and $\de\in(0,1)$ be sufficiently small.
If $(u_\epsilon,v_\epsilon,\omega_\epsilon)$ is a classical solution of \eqref{wodu} in $[0,T]$ and satisfies
\begin{equation*}
\ome\in C([0,T]; H^{s+4,\gamma}_{\si,\de})\cap C^{1}([0,T]; H^{s+2,\gamma}),
\end{equation*}
then there exists a positive constant C, which depends on $s,~\ga,~\si,~n$ and $\de$ such that
\begin{equation}\label{normal-energy-estimate}
\begin{split}
&\fs\frac{d}{dt}\sum_{\substack{|\al|\leq s \\ \al_1\leq s-1}}
\|(1+y)^{\ga+\al_2}D^{\al}\ome\|^2_{L^2}
+\epsilon^2\sum_{\substack{|\al|\leq s \\ \al_1\leq s-1}}\|(1+y)^{\ga+\al_2}\px D^{\al}\ome\|^2_{L^2}
\\
&+\frac{n}{2}\sum_{\substack{|\al|\leq s \\ \al_1\leq s-1}}
\|(1+y)^{\ga+\al_2}\ome^{\frac{n-1}{2}}\py D^{\al}\ome\|_{L^2}^2\\
\leq &
\frac{n}{4}\|(1+y)^{\ga}\ome^{\frac{n-1}{2}}\py\gse\|_{L^2}^2
+C_{s,\gamma,\sigma,n,\delta}\Bigg(1+\|\px^s U\|_{L^2(\T)}+\|\ome\|_{\Hg}\\
&\qquad+\sum_{l=0}^{\frac{s}{2}}\|\partial_t^l\px^{s-2l+1}\pe\|_{L^{\ty}(\T)}\Bigg)^{4(s+1)}.
\end{split}
\end{equation}
\end{Proposition}

\textbf{Proof.}
Let $\al$ satisfy $|\al|\leq s$ and $\al_1\leq s-1$. Applying $D^{\al}$ to \eqref{wodu}, it yields
\begin{equation}\label{alwodu}
\begin{split}
&\pt D^{\al}\ome +\ue\px D^{\al}\ome +\ve\py D^{\al}\ome -\epsilon^2\px^2D^{\al}\ome -n\ome^{n-1}\py^2 D^{\al}\ome \\
=
&-\sum_{0<\beta\leq\al}\binom{\al}{\beta}
\left(D^{\beta}\ue\px D^{\al-\beta}\ome +D^{\beta}\ve\py D^{\al-\beta}\ome\right)\\
&+n\sum_{0<\beta\leq\al}\binom{\al}{\beta}
D^{\beta}\ome ^{n-1}\py^2 D^{\al-\beta}\ome
+n(n-1)\sum_{0\leq\beta\leq\al}\binom{\al}{\beta}
D^{\beta}\ome^{n-2}D^{\al-\beta}|\py\ome|^2.
\end{split}
\end{equation}
Taking $L^2$ inner product of \eqref{alwodu} with $(1+y)^{2\ga+2\al_2}D^{\al}\ome$ yields
\begin{align}
&\frac{1}{2}\frac{d}{dt}\|(1+y)^{\ga+\al_2} D^{\al}\ome\|^2_{L^2}
+\epsilon^2\|(1+y)^{\ga+\al_2}\px D^{\al}\ome\|^2_{L^2}\notag\\
=&n\int_{\Om}(1+y)^{2\ga+2\al_2}\ome^{n-1}D^{\al}\ome\py^2 D^{\al}\ome dxdy\notag\\
&-\int_{\Om}(1+y)^{2\ga+2\al_2}D^{\al}\ome
\Big(\ue\px D^{\al}\ome +\ve\py D^{\al}\ome \Big) dxdy\notag\\
&-\sum_{0<\beta\leq\al}\binom{\al}{\beta}
\int_{\Om}(1+y)^{2\ga+2\al_2}D^{\al}\ome \Big(D^{\beta}\ue\px D^{\al-\beta}\ome +D^{\beta}\ve\py D^{\al-\beta}\ome \Big)dxdy\notag\\
&+n\sum_{0<\beta\leq\al}\binom{\al}{\beta}
\int_{\Om}(1+y)^{2\ga+2\al_2}D^{\al}\ome
D^{\beta}\ome^{n-1}\py^2 D^{\al-\beta}\ome dxdy\notag\\
&+n(n-1)\sum_{0\leq\beta\leq\al}\binom{\al}{\beta}
\int_{\Om}(1+y)^{2\ga+2\al_2}D^{\al}\ome
D^{\beta}\ome^{n-2}D^{\al-\beta}|\py\ome|^2 dxdy\notag\\
:=&\sum\limits_{i=1}^5 I_i.
\label{neiji}
\end{align}
For $I_1$, by integration by parts, we have
\begin{equation}\label{K1jisuan}
\begin{split}
I_1=
&-n\|(1+y)^{\ga+\al_2}\ome^{\frac{n-1}{2}}\py D^{\al}\ome\|_{L^2}^{2} \\
&-n(n-1)\int_{\Omega}(1+y)^{2\ga+2\al_2}\ome^{n-2}\py\ome D^{\al}\ome
\py D^{\al}\ome dxdy\\
&-n(2\ga+2\al_2)\int_{\Omega}(1+y)^{2\ga+2\al_2-1}\ome^{n-1} D^{\al}\ome
\py D^{\al}\ome dxdy\\
&-n\int_{\T}(\ome^{n-1} D^{\al}\ome
\py D^{\al}\ome)|_{y=0}dx\\
:=&\sum\limits_{i=1}^4 I^i_{1}.
\end{split}
\end{equation}
For $I^2_{1}$, since $\ome\in H^{s,\gamma}_{\si,\de}$, we have
\begin{equation}\label{K111}
\de\leq(1+y)^{\si}\ome\leq\de^{-1},\quad \left|(1+y)^{\si+1}\py\ome \right|\leq\de^{-1},
\end{equation}
then, for  $n>1$,
\begin{equation}\label{K11-1}
\left|\ome^{\frac{n-3}{2}}\py\ome \right|\leq \de^{-\frac{|n-3|}{2}}(1+y)^{-\si\frac{n-3}{2}}
\de^{-1}(1+y)^{-(\si+1)}=\de^{-1-\frac{|n-3|}{2}}(1+y)^{-\si\frac{n-1}{2}-1}\leq C_\de.
\end{equation}
Applying \eqref{K11-1}, we obtain
\begin{equation}\label{K125}
\begin{split}
|I^2_1|&\leq n(n-1)\|(1+y)^{\ga+\al_2}\ome^{\frac{n-1}{2}}\py D^{\al}\ome\|_{L^2}
\|\ome^{\frac{n-3}{2}}\py\ome \|_{L^\infty}
\|(1+y)^{\ga+\al_2}  D^{\al}\ome\|_{L^2}\\
&\leq \frac{n}{10}\|(1+y)^{\ga+\al_2}\ome^{\frac{n-1}{2}}\py D^{\al}\ome\|_{L^2}^2
+C_{n,\delta}\|\ome\|_{\Hg}^2.
\end{split}
\end{equation}
For $I^3_{1}$, similarly, with $\ome\in H^{s,\gamma}_{\si,\de}$, we can obtain
\begin{equation}\label{K32}
|I^3_1|\leq \frac{n}{10}\|(1+y)^{\ga+\al_2}\ome^{\frac{n-1}{2}}\py D^{\al}\ome\|_{L^2}^2
+C_{s,\gamma,n,\delta}\|\ome\|_{\Hg}^2.
\end{equation}
Now, we turn to address the boundary $I^4_1$. We have the following estimate, which will be shown later.
\begin{equation}\label{K140}
\begin{split}
\sum_{\substack{|\al|\leq s \\ \al_1\leq s-1}}|I^4_1|
\leq &\frac{n}{10}\sum_{\substack{|\al|\leq s \\ \al_1\leq s-1}}\|(1+y)^{\ga+\al_2}\ome^{\frac{n-1}{2}}\py D^{\al}\ome\|_{L^2}^2
+\frac{n}{4}\|(1+y)^{\ga}\ome^{\frac{n-1}{2}}\py\gse\|_{L^2}^2\\
&+C_{s,\gamma,\sigma,n,\delta}\Bigg((1+\|\px^s U\|_{L^2(\T)}+\|\ome\|_{\Hg}+\sum_{l=0}^{\frac{s}{2}}\|\partial_t^l\px^{s-2l+1}\pe\|_{L^\ty(\T)}\Bigg)^{4(s+1)}.
\end{split}
\end{equation}
Substituting all the estimates for $I^i_1$ into $I_1$  and summing over $\al$, we obtain
\begin{equation}\label{K1deguji}
\begin{split}
\sum_{\substack{|\al|\leq s \\ \al_1\leq s-1}}  I_1  \leq&
-\frac{7n}{10}\sum_{\substack{|\al|\leq s \\ \al_1\leq s-1}}
\|(1+y)^{\ga+\al_2}\ome^{\frac{n-1}{2}}\py D^{\al}\ome\|_{L^2}^2
+\frac{n}{4}\|(1+y)^{\ga}\ome^{\frac{n-1}{2}}\py\gse\|_{L^2}^2\\
&+C_{s,\gamma,\sigma,n,\delta}\Bigg(1+\|\px^s U\|_{L^2(\T)}+\|\ome\|_{\Hg}+\sum_{l=0}^{\frac{s}{2}}\|\partial_t^l\px^{s-2l+1}\pe\|_{L^\ty(\T)}\Bigg)^{4(s+1)}.
\end{split}
\end{equation}

For $I_2$, by $\eqref{L3}_2$,  Sobolev embedding inequality and $\eqref{chabuduo}$, we obtain
\begin{equation}\label{vdeguji}
\begin{split}
\|(1+y)^{-1}\ve\|_{L^{\ty}}=
&\|(1+y)^{-1}\py^{-1}\px \ue\|_{L^{\ty}}
\leq\|\px \ue\|_{L^{\ty}}\\
\leq&\|\px(\ue-U)\|_{L^{\ty}}+\|\px U\|_{L^{\ty}(\T)}\\
\leq&C_{s,\ga,\si,\de}(\|\ome\|_{\Hg}+\|\px^s U\|_{L^{2}(\T)}).
\end{split}
\end{equation}
Hence, by integration by parts and \eqref{vdeguji}, we have
\begin{equation}\label{K2deguji}
\begin{split}
|I_2|&=(\ga+\al_2)\Big|\int_{\Om}(1+y)^{2\ga+2\al_2-1}\ve(D^{\al}\ome)^2 dxdy\Big|\\
&\leq (\ga+\al_2)\|(1+y)^{-1}\ve\|_{L^{\ty}}\|(1+y)^{\ga+\al_2}D^{\al}\ome\|_{L^2}^2\\
&\leq C_{s,\ga,\si,\de}(\|\ome\|_{\Hg}+\|\px^2 U\|_{L^{2}(\T)})\|\ome\|_{\Hg}^2.
\end{split}
\end{equation}
For $I_3$, $I_4$ and $I_5$, we have the following estimates, which will be shown later.
\begin{equation}\label{K3deguji}
\begin{split}
|I_3|\leq C_{s,\ga,\si,\de}(\|\ome\|_{\Hg}+\|\px^s U\|_{L^{2}(\T)})\|\ome\|_{\Hg}^2,
\end{split}
\end{equation}
\begin{equation}\label{K4deguji}
\begin{split}
\sum_{\substack{|\alpha|\leq s\\\alpha_1\leq s-1}}|I_4|
\leq \frac{n}{10}\sum_{\substack{|\alpha|\leq s\\\alpha_1\leq s-1}}\|(1+y)^{\ga+\al_2}\ome^{\frac{n-1}{2}}\py D^{\al}\ome\|_{L^2}^2
+C_{s,\ga,\si,n,\de}(1+\|\px^{s} U\|_{L^2(\T)}+\|\ome\|_{H^{s, \gamma}_g})^s,
\end{split}
\end{equation}
and
\begin{equation}\label{K5deguji}
\begin{split}
|I_5|\leq \frac{n}{10}\|(1+y)^{\ga+\al_2}\ome^{\frac{n-1}{2}}\py D^{\al}\ome\|_{L^2}^2
+C_{s,\ga,\si,\de,n}
(1+\|\px^{s} U\|_{L^2(\T)}+\|\ome\|_{H^{s, \gamma}_g})^{s+1}.
\end{split}
\end{equation}

Substituting all estimates of $I_i$ into \eqref{neiji} and summing over $\al$ , we can establish that \eqref{normal-energy-estimate} holds.

\hfill $\square$

\noindent\textit{Proof of \eqref{K3deguji}.}
Using definition $I_3$ together with the identities $\py v_\epsilon=-\px u_\epsilon$ and $\ome=\py u_\epsilon$, the terms in $I_3$ can be classified into the following three types.
We write $e_1:=(1,0)$ and $e_2:=(0,1)$, and let $\eta\in \mathbb{N}$ and $\kappa,\theta\in \mathbb{N}^2$:\\
Type 1:
$$I_{31}:=\int_{\Om}(1+y)^{2\ga+2\al_2}D^{\al}\ome \px^\eta \ve D^\kappa\ome  dxdy,$$
where $1\leq \eta\leq s-1$ and $\eta e_1+\kappa=\alpha+e_2.$\\
Type 2:
$$I_{32}:=\int_{\Om}(1+y)^{2\ga+2\al_2}D^{\al}\ome \px^\eta \ue D^\kappa\ome  dxdy,$$
where $1\leq \eta\leq s$ and $\eta e_1+\kappa=\alpha+e_1.$\\
Type 3:
$$I_{33}:=\int_{\Om}(1+y)^{2\ga+2\al_2}D^{\al}\ome D^{\theta}\ome D^\kappa\ome  dxdy,$$
where $|\theta|\leq s-2$ and $\theta+\kappa=\alpha+e_1-e_2.$

For $I_{31}$, if $1\leq \eta\leq s-2$, similarly to the estimate in \eqref{vdeguji}, one can derive that
\begin{equation}\label{vxdeguji}
\begin{split}
\|(1+y)^{-1}\px^\eta \ve\|_{L^{\ty}}
\leq C_{s,\ga,\si,\de}(\|\ome\|_{\Hg}+\|\px^s U\|_{L^{2}(\T)}).
\end{split}
\end{equation}
Combined with the H\"{o}lder's inequality, it yields that
\begin{equation*}
\begin{split}
|I_{31}|&\leq \|(1+y)^{\ga+\al_2}D^{\al}\ome\|_{L^{2}}\|(1+y)^{-1}\px^\eta \ve\|_{L^{\ty}}
\|(1+y)^{\ga+\kappa_2}D^{\kappa}\ome\|_{L^{2}}\\
&\leq C_{s,\ga,\si,\de}(\|\ome\|_{\Hg}+\|\px^s U\|_{L^{2}(\T)})\|\ome\|_{\Hg}^2,
\end{split}
\end{equation*}
where we have used the fact that $\kappa_2=\alpha_2+1.$\\
By Lemma \ref{hardy} and $\eqref{chabuduo}$, for any $k\in[0,s-1]$, we obtain
\begin{equation}\label{vudeguji}
\begin{split}
\|\frac{\px^{k} \ve+y\px^{k+1}U}{1+y}\|_{L^{2}}\leq 2\|\px^{k+1}(\ue-U)\|_{L^{2}}
\leq C_{s,\ga,\si,\de}(\|\ome\|_{\Hg}+\|\px^s U\|_{L^{2}(\T)}).
\end{split}
\end{equation}
If $\eta=s-1$,  combining \eqref{vudeguji} with the Sobolev inequality, we obtain
\begin{equation*}
\begin{split}
|I_{31}|\leq& \|(1+y)^{\ga+\al_2}D^{\al}\ome\|_{L^{2}}
\Big(\|\frac{\px^{s-1} \ve+y\px^{s}U}{1+y}\|_{L^{2}}
\|(1+y)^{\ga+\kappa_2}D^{\kappa}\ome\|_{L^{\ty}}\\
&\qquad +\|\px^{s}U\|_{L^{2}(\T)}\|(1+y)^{\ga+\kappa_2}D^{\kappa}\ome\|_{L_x^{\ty}L_y^2}\Big)\\
\leq& C_{s,\ga,\si,\de}\Big(\|\ome\|_{\Hg}+\|\px^s U\|_{L^{2}(\T)}\Big)\|\ome\|_{\Hg}^2,
\end{split}
\end{equation*}
where we have also used the fact that $\kappa_2=\alpha_2+1.$ Therefore,
\begin{equation*}
\begin{split}
|I_{31}|\leq C_{s,\ga,\si,\de}(\|\ome\|_{\Hg}+\|\px^s U\|_{L^{2}(\T)})\|\ome\|_{\Hg}^2.
\end{split}
\end{equation*}

For $I_{32}$, from $\eta e_1+\kappa=\alpha+e_1 $, it is easy to see that $\kappa_2=\alpha_2$. If $1\leq \eta\leq s-1$, applying Lemma \ref{chazhi}, one can derive that
\begin{equation*}
\begin{split}
|I_{32}|&\leq \|(1+y)^{\ga+\al_2}D^{\al}\ome\|_{L^{2}}\|\px^\eta \ue\|_{L^\ty}
\|(1+y)^{\ga+\kappa_2}D^{\kappa}\ome\|_{L^{2}}\\
&\leq C(\|\px^\eta (\ue-U)\|_{L^\ty}+\|\px^\eta U\|_{L^\ty(\T)})\|\ome\|_{\Hg}^2\\
&\leq C(\|\px^\eta (\ue-U)\|_{L^2}+\|\px^{\eta+1} (\ue-U)\|_{L^2}+\|\px^\eta \py \ome\|_{L^2}+\|\px^\eta U\|_{L^\ty(\T)})\|\ome\|_{\Hg}^2\\
&\leq C_{s,\ga,\si,\de}(\|\ome\|_{\Hg}+\|\px^s U\|_{L^{2}(\T)})\|\ome\|_{\Hg}^2.
\end{split}
\end{equation*}
If $\eta= s$, it is easy to see that $\kappa=(0,0)$ or $\kappa=(0,1)$ , and thus
\begin{equation*}
\begin{split}
|I_{32}|&\leq \|(1+y)^{\ga+\al_2}D^{\al}\ome\|_{L^{2}}
\Big(\|\px^s(\ue-U)\|_{L^2}\|(1+y)^{\gamma+\kappa_2}D^\kappa\ome\|_{L^\infty}\\
&\quad+\|\px^sU\|_{L^2(\T)}\|(1+y)^{\gamma+\kappa_2}D^\kappa\ome\|_{L_x^\infty L_y^2}\Big)\\
&\leq C_{s,\gamma,\sigma,\delta}(\|\ome\|_{\Hg}+\|\px^s U\|_{L^2(\T)})\|\ome\|_{\Hg}\\
&\leq C_{s,\ga,\si,\de}(\|\ome\|_{\Hg}+\|\px^s U\|_{L^{2}(\T)})\|\ome\|_{\Hg}^2.\\
\end{split}
\end{equation*}
Here we have used  $|D^\kappa\ome|\leq\delta^{-1}(1+y)^{-\sigma-\kappa_2}$ and  $\sigma>\gamma+\frac12$ to obtain
\begin{equation*}
\begin{split}
\|(1+y)^{\gamma+\kappa_2}D^\kappa\ome\|_{L_x^\infty L_y^2}
&\leq
\delta^{-1}
\left(\int_0^\infty(1+y)^{2\gamma-2\sigma}\,dy\right)^{1/2}
=\frac{\delta^{-1}}{\sqrt{2\sigma-2\gamma-1}}\leq C_\de.
\end{split}
\end{equation*}

For $I_{33}$, from $\theta+\kappa=\alpha+e_1-e_2 $, it is easy to see that $\theta_2+\kappa_2=\alpha_2-1$. The Sobolev inequality and \eqref{chabuduo} yield
\begin{equation*}
\begin{split}
|I_{33}|&\leq \|(1+y)^{\ga+\al_2}D^{\al}\ome\|_{L^2}
\|(1+y)^{1+\theta_2}D^\theta\ome\|_{L^\infty}
\|(1+y)^{\ga+\kappa_2}D^\kappa\ome\|_{L^2}\\
&\leq C_{s,\ga,\si,\de}
(\|\ome\|_{\Hg}+\|\px^sU\|_{L^2(\T)})\|\ome\|_{\Hg}^2.
\end{split}
\end{equation*}
Combining the estimates of $I_{31},I_{32}$ and $I_{33}$, we obtain \eqref{K3deguji}.

\hfill $\square$

\noindent\textit{Proof of \eqref{K4deguji}.}
Recall that
$$I_4=n\sum_{0<\beta\leq\al}\binom{\al}{\beta}
\int_{\Om}(1+y)^{2\ga+2\al_2}D^{\al}\ome
D^{\beta}\ome^{n-1}\py^2 D^{\al-\beta}\ome dxdy.$$

When $s=4$, the result can be obtained by direct computation. Therefore, we only provide the proof for $s\geq5$ here, which requires the use of Lemma \ref{faa} and Proposition \ref{shuaijian}.

For $|\beta|=1$, using H\"{o}lder's inequality yields that
\begin{equation}\label{b1}
\begin{split}
&\Big|\int_{\Om}(1+y)^{2\ga+2\al_2}D^{\al}\ome
D^{\beta}\ome^{n-1}\py^2 D^{\al-\beta}\ome dxdy\Big|\\
=&(n-1)\Big|\int_{\Om}(1+y)^{2\ga+2\al_2}D^{\al}\ome
\cdot\ome^{n-2}D^{\beta}\ome\py^2 D^{\al-\beta}\ome dxdy\Big|\\
\leq &C_n\|(1+y)^{-1+\beta_2}\ome^{\frac{n-3}{2}}D^{\beta}\ome\|_{L^{\ty}}
\|(1+y)^{\ga+\al_2}D^{\al}\ome\|_{L^2}
\|(1+y)^{\ga+\al_2+1-\beta_2}\ome^{\frac{n-1}{2}}\py D^{\al+e_2-\beta}\ome\|_{L^2}.
\end{split}
\end{equation}
Since $\ome\in H^{s,\gamma}_{\si,\de}$ and $n>1$, we have
\begin{equation*}
\|(1+y)^{-1+\beta_2}\ome^{\frac{n-3}{2}}D^{\beta}\ome\|_{L^{\ty}}
\leq C_\delta(1+y)^{-1+\beta_2-\sigma\frac{n-3}{2}-(\sigma+\beta_2)}
\leq C_\delta(1+y)^{-1-\sigma\frac{n-1}{2}}\leq C_\delta.
\end{equation*}
Hence
\begin{equation}\label{b2}
\begin{split}
&n\Big|\int_{\Om}(1+y)^{2\ga+2\al_2}D^{\al}\ome
D^{\beta}(\ome^{n-1})\py^2 D^{\al-\beta}\ome dxdy\Big|\\
\leq &
\frac{n}{10(s+1)}\|(1+y)^{\gamma+\alpha_2+1-\beta_2}\ome^{\frac{n-1}{2}}\py D^{\alpha+e_2-\beta}\ome\|_{L^2}^2
+C_{s,\gamma,\sigma,\delta,n}\|\ome\|_{\Hg}^2.
\end{split}
\end{equation}
The shifted indices $\alpha+e_2-\beta$ are admissible. After multiplication by $\binom\alpha\beta$ and summation over $\alpha$ and $|\beta|=1$, the total multiplicity of each dissipation is at most $s+1$, giving the coefficient $\frac{n}{10}$ in \eqref{K4deguji}.

If $1<|\beta|\leq s-2$, we can deduce
\begin{equation}\label{b3}
\begin{split}
&\Big|\int_{\Om}(1+y)^{2\ga+2\al_2}D^{\al}\ome
D^{\beta}\ome^{n-1}\py^2 D^{\al-\beta}\ome dxdy\Big|\\
\leq &C\|(1+y)^{-2+\beta_2}D^{\beta}\ome^{n-1}\|_{L^{\ty}}
\|(1+y)^{\ga+\al_2}D^{\al}\ome\|_{L^2}
\|(1+y)^{\ga+\al_2+2-\beta_2}\py^2 D^{\al-\beta}\ome\|_{L^2}.
\end{split}
\end{equation}
By using Lemma \ref{faa}, we have
\begin{equation}\label{K41}
D^{\beta}\ome^{n-1} = \sum_{r=1}^{|\beta|} (n-1)^{\underline{r}} \, \ome^{n-1-r} \sum_{\substack{\beta^1 + \cdots + \beta^r = \beta \\ |\beta^i| \ge 1}} \frac{\beta!}{r!} \prod_{i=1}^r \frac{D^{\beta^i} \ome}{\beta^i!}.
\end{equation}
Let $r_0$ denote the number of $\beta^i$ such that $|\beta^i|=1$, then by $\ome\in H^{s,\gamma}_{\si,\de}$, \eqref{shuaijian5} and $\beta^1 + \cdots + \beta^r = \beta$, we have
\begin{equation}\label{K43}
\begin{split}
\|(1+y)^{-2+\beta_2}\ome^{n-1-r}\prod_{i=1}^r D^{\beta^i} \ome\|_{L^{\ty}}
\leq C(1+\|\ome\|_{H^{s,\gamma}})^r
\|(1+y)^{c_r}\|_{L^{\ty}},
\end{split}
\end{equation}
where
\begin{equation}\label{K44}
\begin{split}
c_r=&-2-\si(n-1-r)-\si r_0-\frac{(s-2)\si(r-r_0)-2\ga(r-r_0)}{s-4}\\
&+\frac{(|\beta|-r_0)(\si-\ga)}{s-4}.
\end{split}
\end{equation}

Case 1: if $r=r_0$, then it must be that $r=r_0=|\beta|$. In this time, we have
\begin{equation}\label{K45}
\begin{split}
c_r=-2-\si(n-1)<0.
\end{split}
\end{equation}

Case 2: if $r-r_0\geq1$, using $\si\leq \ga+4/3$, then we have
\begin{equation}\label{K46}
\begin{split}
c_r&=-2-\si(n-1)+\frac{(|\beta|-r)(\si-\ga)}{s-4}
-\frac{(\si-\ga)(r-r_0)}{s-4}\\
&\leq-2-\si(n-1)+\frac{(|\beta|-r)(\si-\ga)}{s-4}
-\frac{\si-\ga}{s-4}\\
&\leq-2-\si(n-1)+\frac{(s-3)(\si-\ga)}{s-4}
-\frac{\si-\ga}{s-4}\\
&=\si-\ga-2-\si(n-1)\\
&\leq\si-\ga-2\\
&\leq0.
\end{split}
\end{equation}
By combining \eqref{b3}-\eqref{K46} and \eqref{chabuduo}, we get that for $1<|\beta|\leq s-2$, there holds
\begin{equation}\label{xinhaolei}
\begin{split}
&\Big|\int_{\Om}(1+y)^{2\ga+2\al_2}D^{\al}\ome
D^{\beta}\ome^{n-1}\py^2 D^{\al-\beta}\ome dxdy\Big|\\
\leq& C(1+\|\ome\|_{H^{s, \gamma}})^s\\
\leq&  C_{s,\ga,\si,\de}(1+
\|\px^{s} U\|_{L^2(\T)}+\|\ome\|_{H^{s, \gamma}_g})^s.
\end{split}
\end{equation}

Now for $|\beta|=s-1$, when all $\beta^i$ in \eqref{K41} satisfy $|\beta^i|\leq s-2$, we can see that \eqref{xinhaolei} still holds. When there exists $|\beta^1|=s-1$, we have $r=1$ and $\beta^1=\beta$. Thus it holds
\begin{equation*}
\begin{split}
&\|(1+y)^{\ga+\al_2}\ome^{n-1-r}\prod_{i=1}^r D^{\beta^i} \ome
\py^2 D^{\al-\beta}\ome
\|_{L^2}
=\|(1+y)^{\ga+\al_2}\ome^{n-2}D^{\beta} \ome
\py^2 D^{\al-\beta}\ome
\|_{L^2}\\
\leq&\|(1+y)^{-\ga-2}\ome^{n-2}\|_{L^{\ty}}
\|(1+y)^{\ga+\beta_2}D^{\beta}\ome\|_{L^{\ty}_x L^{2}_y}
\|(1+y)^{\ga+\al_2-\beta_2+2}D^{\al-\beta+2e_2}\ome\|_{L^2_x L^{\ty}_y}.
\end{split}
\end{equation*}
From $\ome\in H^{s,\gamma}_{\si,\de}$ and $\si\leq \ga+\frac{4}{3}$, we obtain
\begin{equation*}
\begin{split}
\|(1+y)^{-\ga-2}\ome^{n-2}\|_{L^{\ty}}
\leq C_\delta (1+y)^{-\ga-2-\sigma(n-2)}\leq C_{n,\delta}.
\end{split}
\end{equation*}
Together with the Sobolev inequality, this yields
\begin{equation*}
\begin{split}
&\|(1+y)^{\ga+\al_2}\ome^{n-1-r}\prod_{i=1}^r D^{\beta^i} \ome
\py^2 D^{\al-\beta}\ome
\|_{L^2}
\leq C_{s,\ga,\si,n,\de}(1+\|\px^sU\|_{L^2(\T)}+\|\ome\|_{\Hg})^2.
\end{split}
\end{equation*}
Next for $|\beta|=s$, we can estimate by using the same method. In summary, we deduce that \eqref{K4deguji} holds by applying Cauchy's inequality.

\hfill $\square$

\noindent\textit{Proof of \eqref{K5deguji}.}
For $\alpha=(0,0)$ the term is estimated directly:
$$
|I_5|=n(n-1)\int_\Omega(1+y)^{2\gamma}\omega_\epsilon^{n-1}(\partial_y\omega_\epsilon)^2 dxdy
\leq C_{n,\delta}\|\omega_\epsilon\|_{H_g^{s,\gamma}}^2,
$$
because $\omega_\epsilon^{n-3}(\partial_y\omega_\epsilon)^2$
is bounded. Assume henceforth $1\leq|\alpha|\leq s$.
A direct computation yields
\begin{equation}\label{K51a}
\begin{split}
I_5=&n(n-1)\sum_{0\leq\beta\leq\al}\binom{\al}{\beta}
\int_{\Om}(1+y)^{2\ga+2\al_2}D^{\al}\ome
D^{\beta}\ome^{n-2}D^{\al-\beta}|\py\ome|^2 dxdy\\
=&n(n-1)\int_{\Om}(1+y)^{2\ga+2\al_2}D^{\al}\ome \ome^{n-2}D^{\al}|\py\ome|^2 dxdy\\
&+n(n-1)\sum_{0<\beta\leq\al}\binom{\al}{\beta}
\int_{\Om}(1+y)^{2\ga+2\al_2}D^{\al}\ome
D^{\beta}\ome^{n-2}D^{\al-\beta}|\py\ome|^2 dxdy\\
=&2n(n-1)\int_{\Om}(1+y)^{2\ga+2\al_2}D^{\al}\ome \ome^{n-2}\py\ome D^{\al}\py\ome dxdy\\
&+n(n-1)\sum_{0<\beta<\al}\binom{\al}{\beta}
\int_{\Om}(1+y)^{2\ga+2\al_2}D^{\al}\ome \ome^{n-2}D^{\beta}\py\ome D^{\al-\beta}\py\ome  dxdy\\
&+n(n-1)\sum_{0<\beta\leq\al}\binom{\al}{\beta}
\int_{\Om}(1+y)^{2\ga+2\al_2}D^{\al}\ome
D^{\beta}\ome^{n-2}D^{\al-\beta}|\py\ome|^2 dxdy\\
:=& I_{51}+I_{52}+I_{53}.
\end{split}
\end{equation}

For $I_{51}$, analogously to \eqref{K125}, one can deduce
\begin{equation}\label{I51}
\begin{split}
|I_{51}|\leq \frac{n}{10}\|(1+y)^{\ga+\al_2}\ome^{\frac{n-1}{2}}\py D^{\al}\ome\|_{L^2}^2
+C_{n,\delta}\|\ome\|_{\Hg}^2.
\end{split}
\end{equation}

For $I_{52}$, if $1\leq|\beta|<s-1$, we know that at least one of $|\beta+e_2|$ and $|\al-\beta+e_2|$ does not exceed $[\frac{s}{2}]$, and we may assume $|\beta+e_2|\leq[\frac{s}{2}]$.
From \eqref{shuaijian5}, $\ome\in H^{s,\gamma}_{\si,\de}$ and $\si\leq \ga+4/3$, we obtain
\begin{equation}\label{yww}
\begin{split}
 |(1+y)^{\beta_2-1}\omega_\epsilon^{n-2}D^\beta\partial_y\omega_\epsilon|
 &\leq C_{s,\gamma,\sigma,n,\delta}(1+\|\omega_\epsilon\|_{H^{s,\gamma}})
 (1+y)^{\sigma-\gamma-2-\sigma(n-1)}\\
 &\leq C_{s,\gamma,\sigma,n,\delta}(1+\|\omega_\epsilon\|_{H^{s,\gamma}}).
\end{split}
\end{equation}
Then, by \eqref{yww}, it follows that
\begin{equation*}
\begin{split}
&\left|\int_{\Om}(1+y)^{2\ga+2\al_2}D^{\al}\ome \ome^{n-2}D^{\beta}\py\ome D^{\al-\beta}\py\ome  dxdy \right|\\
\leq
&C\|(1+y)^{\ga+\al_2}D^{\al}\ome\|_{L^2}
\|(1+y)^{\ga+\al_2-\beta_2+1}D^{\al-\beta}\py\ome\|_{L^2}
\|(1+y)^{\beta_2-1}\ome^{n-2}D^{\beta}\py\ome\|_{L^\infty}\\
\leq
&C_{s,\ga,\si,\delta}(1+\|\ome\|_{H^{s,\gamma}})\|\ome\|_{\Hg}^2\\
\leq&  C_{s,\ga,\si,\de}(1+
\|\px^{s} U\|_{L^2(\T)}+\|\ome\|_{H^{s, \gamma}_g})^3.
\end{split}
\end{equation*}
Moreover, $|\beta|=s-1$, it is easy to see that $|\alpha-\beta+e_2|$=2. Combined with $\ome\in H^{s,\gamma}_{\si,\de}$, one can deduce
\begin{equation*}
\begin{split}
\left|(1+y)^{\al_2-\beta_2-1}\ome^{n-2}D^{\al-\beta}\py\ome\right|
\leq& C_\de (1+y)^{\al_2-\beta_2-1-\sigma-\al_2+\beta_2-1-\si(n-2)}\\
=&C_\de (1+y)^{-2-\si(n-1)}\leq C_\de.
\end{split}
\end{equation*}
Then,
\begin{align*}
&\left|\int_{\Om}(1+y)^{2\ga+2\al_2}D^{\al}\ome \ome^{n-2}D^{\beta}\py\ome D^{\al-\beta}\py\ome  dxdy \right|\\
\leq
&C\|(1+y)^{\ga+\al_2}D^{\al}\ome\|_{L^2}
\|(1+y)^{\ga+\beta_2+1}D^{\beta}\py\ome\|_{L^2}
\|(1+y)^{\al_2-\beta_2-1}\ome^{n-2}D^{\al-\beta}\py\ome\|_{L^\infty}\\
\leq
&C_{\delta}\|\ome\|_{\Hg}^2.
\end{align*}
Therefore,
\begin{equation}\label{I52}
\begin{split}
|I_{52}|\leq C_{s,\ga,\si,\de}(1+
\|\px^{s} U\|_{L^2(\T)}+\|\ome\|_{H^{s, \gamma}_g})^3.
\end{split}
\end{equation}

For $I_{53}$, by the Leibniz formula, it is easy to obtain
\begin{equation}\label{K51b}
\begin{split}
D^{\al-\beta}|\py\ome|^2=
\sum_{0\leq\tilde{\beta}\leq\al-\beta}
\binom{\al-\beta}{\tilde{\beta}}
D^{\tilde{\beta}}\py\ome
D^{\al-\beta-\tilde{\beta}}\py\ome.
\end{split}
\end{equation}
When $|\beta|=1$,
 if $0\leq|\tilde{\beta}|\leq[\frac{s}{2}]$, by the Sobolev inequality we obtain
\begin{equation*}
\begin{split}
&\left|\int_{\Om}(1+y)^{2\ga+2\al_2}D^{\al}\ome
D^{\beta}\ome^{n-2} D^{\tilde{\beta}}\py\ome
D^{\al-\beta-\tilde{\beta}}\py\ome dxdy \right| \\
\leq & C\|(1+y)^{\ga+\al_2}D^{\al}\ome\|_{L^2}
\|(1+y)^{\ga+\tilde{\beta}_2+1}D^{\tilde{\beta}}\py\ome \|_{L^\infty}\\
&\cdot\|(1+y)^{\ga+\al_2-\beta_2-\tilde{\beta}_2+1}D^{\al-\beta-\tilde{\beta}}\py\ome\|_{L^2}
\|(1+y)^{-\ga+\beta_2-2}D^{\beta}\ome^{n-2}\|_{L^\infty}\\
\leq &C \|(1+y)^{-\ga+\beta_2-2}D^{\beta}\ome^{n-2}\|_{L^\infty}\|\ome\|_{\Hg}^3.
\end{split}
\end{equation*}
Moreover, from $\ome\in H^{s,\gamma}_{\si,\de}$ and  $\si\leq \ga+4/3$, we have
\begin{equation*}
\begin{split}
|(1+y)^{-\ga+\beta_2-2}D^{\beta}\ome^{n-2}|
\leq & C_\de(1+y)^{-\ga+\beta_2-2}(1+y)^{-\si(n-2)-\beta_2}\\
=&C_\de(1+y)^{-\si(n-1)+\si-\ga-2}
\leq  C_\de,
\end{split}
\end{equation*}
and
\begin{equation*}
\begin{split}
\left|\int_{\Om}(1+y)^{2\ga+2\al_2}D^{\al}\ome
D^{\beta}\ome^{n-2} D^{\tilde{\beta}}\py\ome
D^{\al-\beta-\tilde{\beta}}\py\ome dxdy \right|
\leq C_\de\|\ome\|_{\Hg}^3.
\end{split}
\end{equation*}
If $[\frac{s}{2}]<|\tilde{\beta}|\leq|\al-\beta|$, the estimate is analogous. Therefore, when $|\beta|=1$, one can obtain
\begin{equation*}
\begin{split}
|I_{53}|
\leq C_{s,\gamma,\sigma,n,\delta}\|\ome\|_{\Hg}^3.
\end{split}
\end{equation*}
When $1<|\beta|\leq s-2$, we know that at least one of $|\tilde{\beta}+e_2|$ and $|\al-\beta-\tilde{\beta}+e_2|$ does not exceed $[\frac{s}{2}]$, and we may assume $|\tilde{\beta}+e_2|\leq[\frac{s}{2}]$. In this time, we can obtain by \eqref{shuaijian5},
\begin{equation}\label{haoleia}
\begin{split}
|D^{\tilde{\beta}}\py\ome|
\leq& C(1+\|\ome\|_{H^{s, \gamma}}) (1+y)^{-\frac{\si+\ga}{2}-1-\tilde{\beta}_2}\\
\leq& C_{s,\ga,\si,\de}(1+
\|\px^{s} U\|_{L^2(\T)}+\|\ome\|_{H^{s, \gamma}_g})(1+y)^{-\frac{\si+\ga}{2}-1-\tilde{\beta}_2}.
\end{split}
\end{equation}
Therefore,
\begin{equation*}
\begin{split}
&\left|\int_{\Om}(1+y)^{2\ga+2\al_2}D^{\al}\ome
D^{\beta}\ome^{n-2} D^{\tilde{\beta}}\py\ome
D^{\al-\beta-\tilde{\beta}}\py\ome dxdy \right| \\
\leq & C_{s,\ga,\si,\de}\|(1+y)^{\ga+\al_2}D^{\al}\ome\|_{L^2}
\|(1+y)^{\ga+\al_2-\beta_2-\tilde{\beta}_2+1}D^{\al-\beta-\tilde{\beta}}\py\ome\|_{L^2}\\
&\cdot\|(1+y)^{-\frac{\si+\ga}{2}-2+\beta_2}D^{\beta}\ome^{n-2}\|_{L^\infty}
(1+\|\px^{s} U\|_{L^2(\T)}+\|\ome\|_{H^{s, \gamma}_g})\\
\leq &C_{s,\ga,\si,\de} \|(1+y)^{-\frac{\si+\ga}{2}-2+\beta_2}D^{\beta}\ome^{n-2}\|_{L^\infty}
(1+\|\px^{s} U\|_{L^2(\T)}+\|\ome\|_{H^{s, \gamma}_g})^3.
\end{split}
\end{equation*}
On the other hand, similar to \eqref{K41}-\eqref{K46}, we obtain
\begin{equation}\label{K410}
D^{\beta}\ome^{n-2} = \sum_{r=1}^{|\beta|} (n-2)^{\underline{r}} \, \ome^{n-2-r} \sum_{\substack{\beta^1 + \cdots + \beta^r = \beta \\ |\beta^i| \ge 1}} \frac{\beta!}{r!} \prod_{i=1}^r \frac{D^{\beta^i} \ome}{\beta^i!}.
\end{equation}
And
\begin{equation}\label{K430}
\begin{split}
&\|(1+y)^{-\frac{\si+\ga}{2}-2+\beta_2}\ome^{n-2-r}\prod_{i=1}^r D^{\beta^i} \ome\|_{L^{\ty}}\\
\leq &C_{s,\ga,\si,\de}(1+\|\ome\|_{H^{s, \gamma}})^r
\|(1+y)^{d_r}\|_{L^{\ty}}\\
\leq &C_{s,\ga,\si,\de}
(1+\|\px^{s} U\|_{L^2(\T)}+\|\ome\|_{H^{s, \gamma}_g})^r
\|(1+y)^{d_r}\|_{L^{\ty}},
\end{split}
\end{equation}
where
\begin{equation}\label{K440}
\begin{split}
d_r=&-\frac{\si+\ga}{2}-2-\si(n-2-r)-\si r_0-\frac{(s-2)\si(r-r_0)-2\ga(r-r_0)}{s-4}\\
&+\frac{(|\beta|-r_0)(\si-\ga)}{s-4},
\end{split}
\end{equation}
and $r_0$ denote the number of $\beta^i$ such that $|\beta^i|=1$.

Case 1: if $r=r_0$, then it must be that $r=r_0=|\beta|$. In this time, by $\si\leq \ga+\frac{4}{3}$, we have
\begin{equation}\label{K450}
\begin{split}
d_r=-\frac{\si+\ga}{2}-2-\si(n-2)=-\si(n-1)+\frac{\si-\ga}{2}-2<0.
\end{split}
\end{equation}

Case 2: if $r-r_0\geq1$, by $n>1$ and $\si\leq \ga+\frac{4}{3}$, then we have
\begin{equation}\label{K460}
\begin{split}
d_r&=-\frac{\si+\ga}{2}-2-\si(n-2)
+\frac{(|\beta|-r)(\si-\ga)}{s-4}
-\frac{(\si-\ga)(r-r_0)}{s-4}\\
&\leq-\frac{\si+\ga}{2}-2-\si(n-2)
+\frac{(|\beta|-r)(\si-\ga)}{s-4}
-\frac{\si-\ga}{s-4}\\
&\leq-\frac{\si+\ga}{2}-2-\si(n-2)
+\frac{(s-3)(\si-\ga)}{s-4}
-\frac{\si-\ga}{s-4}\\
&=\frac{3}{2}(\si-\ga-\frac{4}{3})-\si(n-1)\\
&\leq0.
\end{split}
\end{equation}
Then, by combining \eqref{K410}-\eqref{K460}, we have for $1<|\beta|\leq s-2$, there holds
\begin{equation*}
\begin{split}
&\left|\int_{\Om}(1+y)^{2\ga+2\al_2}D^{\al}\ome
D^{\beta}\ome^{n-2} D^{\tilde{\beta}}\py\ome
D^{\al-\beta-\tilde{\beta}}\py\ome dxdy \right| \\
\leq &C_{s,\ga,\si,\de}
(1+\|\px^{s} U\|_{L^2(\T)}+\|\ome\|_{H^{s, \gamma}_g})^{s+1}.
\end{split}
\end{equation*}
Finally, for $|\beta|=s-1$ or $s$, one may refer to the estimate of $I_4$; for simplicity, we omit it.
In summary, one can deduce that
\begin{equation}\label{xinhaolei2}
\begin{split}
|I_{53}|
\leq C_{s,\ga,\si,\de,n}
(1+\|\px^{s} U\|_{L^2(\T)}+\|\ome\|_{H^{s, \gamma}_g})^{s+1}.
\end{split}
\end{equation}
Combined with \eqref{I51}, \eqref{I52} and \eqref{xinhaolei2}, \eqref{K5deguji} follows.

\hfill $\square$

\noindent\textit{Proof of \eqref{K140}.}
We consider the following two cases.

Case 1. $|\al|\leq s-1$. 
Applying Lemma~\ref{trace} with
$f=\omega_\epsilon^{n-1}D^\alpha\omega_\epsilon\,\partial_yD^\alpha\omega_\epsilon$
yields
\begin{equation}\label{K141}
\begin{split}
|I^4_1|
&=n\Big|\int_{\T}(\ome^{n-1} D^{\al}\ome
\py D^{\al}\ome)|_{y=0}dx\Big|\\
&\leq
C_n \Big\{\int_0^1\int_{\T} |\ome^{n-1} D^{\al}\ome
\py D^{\al}\ome| dxdy
+\int_0^1\int_{\T} |\py (\ome^{n-1} D^{\al}\ome
\py D^{\al}\ome)| dxdy\Big\}\\
&\leq
C_n \Big\{\int_0^1\int_{\T} |\ome^{n-1} D^{\al}\ome
\py D^{\al}\ome| dxdy
+\int_0^1\int_{\T} |\py\ome \ome^{n-2} D^{\al}\ome
\py D^{\al}\ome| dxdy\\
&\quad\quad\quad +\int_0^1\int_{\T} | \ome^{n-1} \py D^{\al}\ome
\py D^{\al}\ome| dxdy
+\int_0^1\int_{\T} |\ome^{n-1} D^{\al}\ome
\py^2 D^{\al}\ome)| dxdy
\Big\}\\
&:=I^4_{11}+I^4_{12}+I^4_{13}+I^4_{14}.
\end{split}
\end{equation}
Using $\ome\in H^{s,\gamma}_{\si,\de}$ and $n>1$ yields
\begin{equation}\label{K142}
\begin{split}
&|\ome^{n-1}|\leq C_\de(1+y)^{-\si(n-1)} \leq C_\de,\\
&|\py\ome \ome^{n-2}|\leq C_\de(1+y)^{-\si-1-\si(n-2)}=C_\de(1+y)^{-1-\si(n-1)}\leq C_\de.
\end{split}
\end{equation}
By \eqref{K142}, one can deduce
\begin{equation}\label{I41113}
\begin{split}
|I^4_{11}|+|I^4_{12}|+|I^4_{13}|
\leq
C_{\de,n}\|\py D^{\al}\ome\|_{L^2}(\| D^{\al}\ome\|_{L^2}+\|\py D^{\al}\ome\|_{L^2})
\leq
C_{\de,n}\|\ome\|_{\Hg}^2.
\end{split}
\end{equation}
For $I^4_{14}$, applying the H\"{o}lder inequality and \eqref{K142} yields
\begin{equation}\label{I414}
\begin{split}
|I^4_{14}|
\leq
\frac{n}{30}\|(1+y)^{\ga+\al_2}\ome^{\frac{n-1}{2}}\py^2 D^{\al}\ome\|_{L^2}^2
+C_{\delta,n}\|\ome\|_{\Hg}^2.
\end{split}
\end{equation}
Therefore, when $|\al|\leq s-1$, one obtains
\begin{equation}\label{I41}
\begin{split}
\sum\limits_{|\al|\leq s-1}|I^4_{1}|
\leq \frac{n}{30}
\sum_{\substack{|\theta|=s\\ \theta_1\leq s-1}}\|(1+y)^{\gamma+\theta_2}\ome^{\frac{n-1}{2}}\py D^\theta\ome\|_{L^2}^2
+C_{s,\gamma,\sigma,n,\delta}\|\ome\|_{\Hg}^2.
\end{split}
\end{equation}

Case 2. $|\al|=s$. In this case, the order of differentiation of the boundary term $\py D^{\al}\ome|_{y=0}$ is too high, and therefore one cannot apply the trace estimate \eqref{jiguji} to bound the boundary integral $I^4_1$. In order to make use of \eqref{jiguji}, it is necessary to reduce the order of differentiation of the singular term $\py D^{\al}\ome|_{y=0}$. To this end, we need to derive precise boundary conditions, as shown below:

\begin{Lemma}\label{bianjiejiangjie}
At $y=0$,
\begin{equation}\label{yijie}
n\ome^{n-1}\py\ome|_{y=0}=\px \pe.
\end{equation}
For any $0\le k_1\le s-2k_2$ and $1\le k_2\le \frac{s}{2}$, the following identity holds
\begin{equation}\label{gaojie}
\begin{split}
&\px^{k_1}(\ome^{n-1}\py^{2k_2+1}\ome)|_{y=0}\\
=
&\frac{1}{n^{k_2+1}}\ome^{k_2(1-n)}|_{y=0}(\pt-\epsilon^2\px^2)^{k_2}\px^{k_1+1} \pe\\
&+\sum_{l=0}^{k_2}\epsilon^{2l} \sum_{j=1}^{k_1+2k_2+1}\sum_{\rho\in A_{k_{1},k_{2}}^j}
Q_{k_{1},k_{2},\rho}^l(\ome; \mathcal{Q}_\epsilon)
\prod_{i=1}^{j}D^{\rho^i}\ome\Big|_{y=0}.
\end{split}
\end{equation}
Here
$$\mathcal{Q}_\epsilon:=\bigl\{(\pt-\epsilon^2\px^2)^a\px^{b+1}\pe:
a,b\in\mathbb N,\ 2a+b+1\leq k_1+2k_2\bigr\}$$
and
$$A_{k_1,k_2}^j:=\{\rho:=(\rho^1,\rho^2,\cdots,\rho^j)\in \mathbb{N}^{2j};1\leq|\rho^i|=\rho^i_1+\rho^i_2\leq k_1+2k_2;
\sum\limits_{i=1}^{j}|\rho^i|\leq k_1+2k_2+1\}.$$
For each $l$ and $\rho$, $Q_{k_1,k_2,\rho}^{\ell}(\ome;\mathcal{Q}_\epsilon)$ is a polynomial in the lower order derivatives of $\pe$ listed in  $\mathcal{Q}_\epsilon$, with coefficients given by finite linear combinations of real powers of $\ome$.  If a monomial in $Q_{k_1,k_2,\rho}^{\ell}$ contains the pressure factors $(\pt-\epsilon^2\px^2)^{a_\nu}\px^{b_\nu+1}\pe$, $1\leq\nu\leq M$, then
\[
\sum_{i=1}^j|\rho^i|+
\sum_{\nu=1}^{M}(2a_\nu+b_\nu+1)\leq k_1+2k_2+1.
\]
In particular, the last term on the right-hand side of \eqref{gaojie} contains no term that is linear in
$(\pt-\epsilon^2\px^2)^{k_2}\px^{k_1+1}\pe$.
\end{Lemma}


The above lemma follows from a slight modification of the arguments in \cite[Lemma~5.9]{NM},  \cite[Lemma~6.4.2]{Said} and \cite[Lemma~4.2]{ZEW}. To keep the presentation concise, we omit the proof.

We now return to the boundary integral $I_1^4$ in Case~2.

$\bullet$ If $\al_2$ is even, we may assume without loss of generality that $\al_2=2 k_2$.  When $\al=(0,s)$, Lemma \ref{bianjiejiangjie} gives
\begin{equation}\label{I14-even-top}
\begin{split}
I^4_1
&=-n\int_{\T}(\ome^{n-1} \py^{s+1}\ome
\py^{s}\ome)|_{y=0}dx\\
&=
-\frac{1}{n^{k_2}}\int_{\T}\py^{s}\ome\ome^{k_2(1-n)}|_{y=0}(\pt-\epsilon^2\px^2)^{k_2}\px \pe dx\\
&\quad\quad-n\sum_{l=0}^{k_2}\epsilon^{2l} \int_{\T}\sum_{j=1}^{s+1}\sum_{\rho\in A_{0,k_{2}}^j}\py^{s}\ome
Q_{0,k_{2},\rho}^l(\ome; \mathcal{Q}_\epsilon)
\prod_{i=1}^{j}D^{\rho^i}\ome\Big|_{y=0} dx \\
&:=J_{1}+J_{2}.
\end{split}
\end{equation}
Using H\"{o}lder's inequality and \eqref{jiguji}, we can deduce that
\begin{equation*}
\begin{split}
|J_{1}|
&\leq
C_{n,s,\delta}\|(\pt-\epsilon^2\px^2)^{k_2}\px \pe\|_{L^2(\T)}
\big(\int_0^1\|\py^{s}\ome\|_{L^2(\T)} dy+\int_0^1\|\py^{s+1}\ome\|_{L^2(\T)} dy\big)\\
&\leq
C_{n,s,\delta}\|(\pt-\epsilon^2\px^2)^{k_2}\px \pe\|_{L^2(\T)}
\Bigg(\|\ome\|_{\Hg}+\|(1+y)^{\gamma+s}\ome^{\frac{n-1}{2}}\py^{s+1}\ome\|_{L^2} \\ 
&\quad\cdot\left(\int_0^1\|(1+y)^{-(\gamma+s)}\ome^{\frac{1-n}{2}}\|_{L^\infty(\T)}^2\,dy\right)^{1/2}
\Bigg).
\end{split}
\end{equation*}
Since $\ome\in H^{s,\gamma}_{\si,\de}$ and $n>1$, we have
\begin{equation}\label{fuzhibiao}
\begin{split}
\int_0^1\|(1+y)^{-(\gamma+s)}\ome^{\frac{1-n}{2}}\|_{L^\infty(\T)}^2\,dy
&\leq C_\delta\int_0^1 (1+y)^{\sigma(n-1)-2(\gamma+s)}\,dy\\
&\leq C_\delta 2^{\max\{\sigma(n-1)-2(\gamma+s),\,0\}}
\leq C_{s,\gamma,\sigma,n,\delta}.
\end{split}
\end{equation}
Then,
\begin{equation}\label{J1}
\begin{split}
|J_{1}|
&\leq \frac{n}{120}\|(1+y)^{\gamma+s}\ome^{\frac{n-1}{2}}\py^{s+1}\ome\|_{L^2}^2+
C_{s,\gamma,\sigma,n,\delta}\big(\|\ome\|_{\Hg}^2+\|(\pt-\epsilon^2\px^2)^{k_2}\px \pe\|_{L^2(\T)}^2\big).
\end{split}
\end{equation}

The index restrictions in Lemma~\ref{bianjiejiangjie} give
$$ 1\leq|\rho^i|\leq s,\qquad \sum_{i=1}^j|\rho^i|\leq s+1,\qquad 1\leq j\leq s+1.$$
Using \eqref{jiguji} and the Leibniz formula yields
\begin{equation}\label{J2jiguji}
\begin{split}
&\quad\int_\T \left|\py^s\ome\prod_{i=1}^jD^{\rho^i}\ome \Big|_{y=0} \right| dx\\
&\leq
\int_0^1\int_\T \left|\py^s\ome\prod_{i=1}^jD^{\rho^i}\ome\right| dxdy
+\int_0^1\int_\T\left|\py^{s+1}\ome\prod_{i=1}^jD^{\rho^i}\ome\right| dx dy\\
&\quad+\sum_{h=1}^j\int_0^1\int_\T\left|\py^s\ome\,D^{\rho^h+e_2}\ome\prod_{i\ne h}D^{\rho^i}\ome\right| dx dy,
\end{split}
\end{equation}
where $e_1=(1,0)$ and $e_2=(0,1)$. For later use, local Sobolev embedding and \eqref{chabuduo} give the following estimates:
\begin{equation}\label{Dbetaw}
\begin{split}
&\sup_{0\leq y\leq1}\|D^\beta\ome\|_{L^\infty(\T)}
\leq C\sum_{a,b=0}^1\left(\int_0^1\|D^{\beta+ae_1+be_2}\ome\|_{L^2(\T)}^2 dy\right)^{1/2}\leq C\|\ome\|_{\Hg},
\qquad |\beta|\leq s-2,\\
&\int_0^1\|D^\beta\ome\|_{L^2(\T)}^2 dy
\leq C_{s,\gamma,\sigma,\delta}\left(\|\ome\|_{\Hg}+\|\px^sU\|_{L^2(\T)}\right)^2,
\qquad \qquad |\beta|\leq s.
\end{split}
\end{equation}
To estimate the first term on the right-hand side of \eqref{J2jiguji}, choose $q$ such that $|\rho^q|=\max\limits_{1\leq i\leq j}|\rho^i|$. The total order constraint, together with $s\geq6$, implies that $|\rho^i|\leq s-2$ for every $i\ne q$.
Applying H\"{o}lder's inequality together with \eqref{Dbetaw} gives
\begin{equation}\label{y1}
\begin{split}
&\int_0^1\int_\T
\left|\py^s\ome\prod_{i=1}^jD^{\rho^i}\ome\right|dx dy\\
\leq&
\left(\int_0^1\|\py^s\ome\|_{L^2(\T)}^2 dy\right)^{1/2}
\left(\int_0^1\|D^{\rho^q}\ome\|_{L^2(\T)}^2 dy\right)^{1/2}\cdot
\prod_{i\ne q}\sup_{0\leq y\leq1}\|D^{\rho^i}\ome\|_{L^\infty(\T)}\\
\leq& C_{s,\gamma,\sigma,\delta}
\|\ome\|_{\Hg}^{j}\left(\|\ome\|_{\Hg}+\|\px^sU\|_{L^2(\T)}\right).
\end{split}
\end{equation}
For the second term on the right-hand side of \eqref{J2jiguji}, it follows from \eqref{Dbetaw} that
\begin{equation}\label{y2}
\begin{split}
&\int_0^1\int_\T\left|\py^{s+1}\ome\prod_{i=1}^jD^{\rho^i}\ome\right| dx dy\\
\leq&
\sup_{0\leq y\leq1}\|(1+y)^{-\gamma-s}\ome^{\frac{1-n}{2}}\|_{L^\infty(\T)}
\|(1+y)^{\gamma+s}\ome^{\frac{n-1}{2}}\py^{s+1}\ome\|_{L^2}
\cdot
\left(\int_0^1\left\|\prod_{i=1}^jD^{\rho^i}\ome\right\|_{L^2(\T)}^2 dy\right)^{1/2}\\
\leq&
C_{s,\gamma,\sigma,n,\delta}\|(1+y)^{\gamma+s}\ome^{\frac{n-1}{2}}\py^{s+1}\ome\|_{L^2}
\left(\int_0^1\|D^{\rho^q}\ome\|_{L^2(\T)}^2 dy\right)^{1/2}
\prod_{i\ne q}\sup_{0\leq y\leq1}\|D^{\rho^i}\ome\|_{L^\infty(\T)}\\
\leq& C_{s,\gamma,\sigma,n,\delta}\|(1+y)^{\gamma+s}\ome^{\frac{n-1}{2}}\py^{s+1}\ome\|_{L^2}
\|\ome\|_{\Hg}^{j-1}\left(\|\ome\|_{\Hg}+\|\px^sU\|_{L^2(\T)}\right),
\end{split}
\end{equation}
where we have used the fact that
\begin{equation*}
\sup_{0\leq y\leq1}\|(1+y)^{-\gamma-s}\ome^{\frac{1-n}{2}}\|_{L^\infty(\T)}
\leq\delta^{-\frac{n-1}{2}} 2^{\max\{\frac{\sigma(n-1)}2-\gamma-s,\,0\}}
\leq C_{s,\gamma,\sigma,n,\delta}.
\end{equation*}
For the third term on the right-hand side of \eqref{J2jiguji}, the Leibniz formula and H\"{o}lder's inequality give
\begin{equation*}
\begin{split}
\int_0^1\int_\T\left|\py^s\ome\,\py\left(\prod_{i=1}^jD^{\rho^i}\ome\right)\right|dxdy
\leq
\|\ome\|_{\Hg}\sum_{h=1}^j\left(\int_0^1\left\|\py D^{\rho^h}\ome\prod_{i\ne h}D^{\rho^i}\ome\right\|_{L^2(\T)}^2 dy\right)^{1/2}.
\end{split}
\end{equation*}
If $|\rho^h|\leq s-1$, all factors in the differentiated product have order at most $s$, and their total order is at most $s+2$.
Since $s\geq6$, at most one factor has order greater than $s-2$. Consequently,
\begin{equation*}
\begin{split}
\left(
\int_0^1\left\|\py D^{\rho^h}\ome\prod_{i\ne h}D^{\rho^i}\ome\right\|_{L^2(\T)}^2\,dy\right)^{1/2}
&\leq
C_{s,\gamma,\sigma,\delta}\|\ome\|_{\Hg}^{j-1}\left(\|\ome\|_{\Hg}+\|\px^sU\|_{L^2(\T)}\right).
\end{split}
\end{equation*}
If $|\rho^h|=s$, then $\sum\limits_{i\ne h}|\rho^i|\leq1$. From \eqref{Dbetaw}, we derive that
\begin{equation*}
\begin{split}
&\left(\int_0^1\left\|\py D^{\rho^h}\ome\prod_{i\ne h}D^{\rho^i}\ome\right\|_{L^2(\T)}^2 dy\right)^{1/2}\\
\leq& \left(\int_0^1\|\py D^{\rho^h}\ome\|_{L^2(\T)}^2 dy\right)^{1/2}
\prod_{i\ne h}\sup_{0\leq y\leq1}\|D^{\rho^i}\ome\|_{L^\infty(\T)}\\
\leq& C\|\ome\|_{\Hg}^{j-1}
\left(\int_0^1\|\py D^{\rho^h}\ome\|_{L^2(\T)}^2 dy\right)^{1/2}.
\end{split}
\end{equation*}
For $\rho_1^h\leq s-1$, using $\sup\limits_{0\leq y\leq1}\|(1+y)^{-\gamma-\rho_2^h}\ome^{\frac{1-n}{2}}\|_{L^\infty(\T)}
\leq C_{s,\gamma,\sigma,n,\delta}$, then
\begin{equation*}
\begin{split}
&\left(\int_0^1\|\py D^{\rho^h}\ome\|_{L^2(\T)}^2 dy\right)^{1/2}
\leq C_{s,\gamma,\sigma,n,\delta}
\|(1+y)^{\gamma+\rho_2^h}\ome^{\frac{n-1}{2}}\py D^{\rho^h}\ome\|_{L^2}.
\end{split}
\end{equation*}
For $\rho^h=(s,0)$, using $\py\px^s\ome=\py\gse+(\py\aep)\px^s(\ue-U)+\aep\px^s\ome$ and \eqref{chabuduo}, we get
\begin{equation*}
\begin{split}
\left(\int_0^1\|\py\px^s\ome\|_{L^2(\T)}^2 dy\right)^{1/2}
\leq C_{s,\gamma,\sigma,n,\delta}
\left(\|(1+y)^\gamma\ome^{\frac{n-1}{2}}\py\gse\|_{L^2}+\|\ome\|_{\Hg}+\|\px^sU\|_{L^2(\T)}\right).
\end{split}
\end{equation*}
Combining these estimates and summing over $h$ gives
\begin{equation}\label{y3}
\begin{split}
&\int_0^1\int_\T\left|\py^s\ome\,\py\left(\prod_{i=1}^jD^{\rho^i}\ome\right)\right| dx dy\\
&\leq C_{s,\gamma,\sigma,n,\delta}\|\ome\|_{\Hg}^{j}
\Bigg[\|\ome\|_{\Hg}+\|\px^sU\|_{L^2(\T)}\\
&\qquad+\Big(\sum_{\substack{|\theta|=s\\ \theta_1\leq s-1}}\|(1+y)^{\gamma+\theta_2}\ome^{\frac{n-1}{2}}\py D^\theta\ome\|_{L^2}^2
+\|(1+y)^\gamma\ome^{\frac{n-1}{2}}\py\gse\|_{L^2}^2\Big)^{1/2}\Bigg].
\end{split}
\end{equation}
Substituting \eqref{y1}-\eqref{y3} into \eqref{J2jiguji}, we obtain
\begin{equation}\label{J2gujie}
\begin{split}
&\quad\int_\T \left| \py^s\ome\prod_{i=1}^jD^{\rho^i}\ome\Big|_{y=0}\right| dx\\
&\leq C_{s,\gamma,\sigma,n,\delta}
\Bigg[\|\ome\|_{\Hg}^{j}\left(\|\ome\|_{\Hg}+\|\px^sU\|_{L^2(\T)}\right)\\
&\quad+\|(1+y)^{\gamma+s}\ome^{\frac{n-1}{2}}\py^{s+1}\ome\|_{L^2}\|\ome\|_{\Hg}^{j-1}
\left(\|\ome\|_{\Hg}+\|\px^sU\|_{L^2(\T)}\right)\\
&\quad+\|\ome\|_{\Hg}^{j}
\Big(\sum_{\substack{|\theta|=s\\ \theta_1\leq s-1}}\|(1+y)^{\gamma+\theta_2}\ome^{\frac{n-1}{2}}\py D^\theta\ome\|_{L^2}^2
+\|(1+y)^\gamma\ome^{\frac{n-1}{2}}\py\gse\|_{L^2}^2\Big)^{1/2}\Bigg]\\
&\leq C_{s,\gamma,\sigma,n,\delta}
\|\ome\|_{\Hg}^{j-1} \left(\|\ome\|_{\Hg}+\|\px^sU\|_{L^2(\T)}\right)\\
&\quad\cdot
\Bigg[\|\ome\|_{\Hg}+
\Big(\sum_{\substack{|\theta|=s\\ \theta_1\leq s-1}}\|(1+y)^{\gamma+\theta_2}
\ome^{\frac{n-1}{2}}\py D^\theta\ome\|_{L^2}^2
+\|(1+y)^\gamma\ome^{\frac{n-1}{2}}\py\gse\|_{L^2}^2\Big)^{1/2}
\Bigg].
\end{split}
\end{equation}
Combining this estimate with \eqref{J2gujie} and the definition of $J_2$ gives
\begin{equation}\label{J2-bound}
\begin{split}
|J_2|
&\leq n\sum_{l=0}^{k_2}\sum_{j=1}^{s+1}\sum_{\rho\in A_{0,k_2}^j}\|Q_{0,k_2,\rho}^l(\ome;\mathcal Q_\epsilon)|_{y=0}\|_{L^\infty(\T)}
\int_\T\left|\py^s\ome\prod_{i=1}^jD^{\rho^i}\ome\right|_{y=0} dx\\
&\leq C_{s,\gamma,\sigma,n,\delta}
\left(1+\sum_{l=0}^{\frac{s}{2}}\|\partial_t^l\px^{s-2l+1}\pe\|_{L^{\ty}(\T)}\right)^s
\left(\sum_{j=1}^{s+1}\|\ome\|_{\Hg}^{j-1}\right)
\left(\|\ome\|_{\Hg}+\|\px^sU\|_{L^2(\T)}\right)\\
&\quad\cdot
\Bigg[\|\ome\|_{\Hg}+
\Big(\sum_{\substack{|\theta|=s\\ \theta_1\leq s-1}}\|(1+y)^{\gamma+\theta_2}
\ome^{\frac{n-1}{2}}\py D^\theta\ome\|_{L^2}^2
+\|(1+y)^\gamma\ome^{\frac{n-1}{2}}\py\gse\|_{L^2}^2\Big)^{1/2}\Bigg]\\
&\leq
C_{s,\gamma,\sigma,n,\delta}
\left(1+\sum_{l=0}^{\frac{s}{2}}\|\partial_t^l\px^{s-2l+1}\pe\|_{L^{\ty}(\T)}\right)^s
\left(1+\|\ome\|_{\Hg}\right)^s
\left(\|\ome\|_{\Hg}+\|\px^sU\|_{L^2(\T)}\right)\\
&\quad\cdot
\Bigg[\|\ome\|_{\Hg}
+\Big(\sum_{\substack{|\theta|=s\\ \theta_1\leq s-1}}\|(1+y)^{\gamma+\theta_2}
\ome^{\frac{n-1}{2}}\py D^\theta\ome\|_{L^2}^2
+
\|(1+y)^\gamma\ome^{\frac{n-1}{2}}\py\gse\|_{L^2}^2\Big)^{1/2}\Bigg]\\
&\leq
\frac{n}{120}\sum_{\substack{|\theta|=s\\ \theta_1\leq s-1}}\|(1+y)^{\gamma+\theta_2}\ome^{\frac{n-1}{2}}\py D^\theta\ome\|_{L^2}^2
+\frac{n}{16}\|(1+y)^\gamma\ome^{\frac{n-1}{2}}\py\gse\|_{L^2}^2\\
&\quad+C_{s,\gamma,\sigma,n,\delta,}
\Bigg[1+\|\ome\|_{\Hg}+\|\px^sU\|_{L^2(\T)}+
\sum_{l=0}^{\frac{s}{2}}\|\partial_t^l\px^{s-2l+1}\pe\|_{L^{\ty}(\T)}\Bigg]^{4(s+1)}.
\end{split}
\end{equation}

When $2\leq\al_1\leq s-2$, it is easy to see that $\al_1$ is even and $\al=(\al_1,\al_2)=(\al_1,2k_2)$. A direct computation yields
\begin{align*}
I^4_{1}=
&-n\int_{\T}(\ome^{n-1} D^{\al}\ome
\px^{\al_1}\py ^{2k_2+1}\ome)|_{y=0}dx\\
=
&-n\int_{\T}D^{\al}\ome\px^{\al_1}(\ome^{n-1} \py ^{2k_2+1}\ome)|_{y=0}dx\\
&
+n\sum_{i=1}^{\al_1}\binom{\al_1}{i}\int_{\T}D^{\al}\ome(\px^{i}\ome^{n-1}) \px^{\al_1-i}\py ^{2k_2+1}\ome)|_{y=0}dx\\
:=&J_3+J_4.
\end{align*}
For $J_3$, Lemma \ref{bianjiejiangjie} implies that
\begin{equation*}
\begin{split}
J_3=
&-n\int_{\T}D^{\al}\ome\px^{\al_1}(\ome^{n-1} \py ^{2k_2+1}\ome)|_{y=0}dx\\
=&
-\frac{1}{n^{k_2}}\int_{\T}D^{\al}\ome\ome^{k_2(1-n)}|_{y=0}(\pt-\epsilon^2\px^2)^{k_2}\px^{\al_1+1} \pe dx\\
&-n\sum_{l=0}^{k_2}\epsilon^{2l} \sum_{j=1}^{\al_1+2k_2+1}\sum_{\rho\in A_{\al_{1},k_{2}}^j}
\int_{\T}D^{\al}\ome|_{y=0}
Q_{\al_{1},k_{2},\rho}^l(\ome; \mathcal{Q}_\epsilon)
\prod_{i=1}^{j}D^{\rho^i}\ome\Big|_{y=0} dx.
\end{split}
\end{equation*}
By estimates similar to those for $J_1$ and $J_2$, we obtain
\begin{equation}\label{J3}
\begin{split}
|J_{3}|
\leq &
\frac{n}{60s}\sum_{\substack{|\theta|=s\\ \theta_1\leq s-1}}\|(1+y)^{\gamma+\theta_2}\ome^{\frac{n-1}{2}}\py D^\theta\ome\|_{L^2}^2
+\frac{n}{8s}\|(1+y)^\gamma\ome^{\frac{n-1}{2}}\py\gse\|_{L^2}^2\\
&\quad+C_{s,\gamma,\sigma,n,\delta}
\Bigg[1+\|\ome\|_{\Hg}+\|\px^sU\|_{L^2(\T)}+
\sum_{l=0}^{\frac{s}{2}}\|\partial_t^l\px^{s-2l+1}\pe\|_{L^{\ty}(\T)}\Bigg]^{4(s+1)}.
\end{split}
\end{equation}
By using Lemma \ref{faa}, we have
\begin{equation}\label{J4-1}
\px^{i}\ome^{n-1}
= \sum_{r=1}^{i} (n-1)^{\underline{r}} \, \ome^{n-1-r} \sum_{\substack{i_1 + \cdots + i_r = i \\ i_j \ge 1}} \frac{i!}{r!} \prod_{j=1}^r \frac{\px^{i_j} \ome}{i_j!}.
\end{equation}
It follows from \eqref{jiguji} that
\begin{equation}\label{J4-trace-decomposition}
\begin{split}
&\int_{\T}\left| D^{\al}\ome(\prod_{j=1}^r \px^{i_j} \ome) \px^{\al_1-i}\py ^{2k_2+1}\ome\Big|_{y=0}\right| dx\\
\leq
& \int_0^1\int_{\T}\Big|D^{\al}\ome(\prod_{j=1}^r \px^{i_j} \ome) \px^{\al_1-i}\py ^{2k_2+1}\ome)\Big| dx dy\\
&
+\int_0^1\int_{\T}\Big|\py D^{\al}\ome(\prod_{j=1}^r \px^{i_j} \ome) \px^{\al_1-i}\py ^{2k_2+1}\ome)\Big| dx dy\\
&
+\int_0^1\int_{\T}\Big|D^{\al}\ome\py(\prod_{j=1}^r \px^{i_j} \ome) \px^{\al_1-i}\py ^{2k_2+1}\ome)\Big| dx dy\\
&
+\int_0^1\int_{\T}\Big|D^{\al}\ome(\prod_{j=1}^r \px^{i_j} \ome) \px^{\al_1-i}\py ^{2k_2+2}\ome)\Big| dx dy\\
:=&J_{41}+J_{42}+J_{43}+J_{44}.
\end{split}
\end{equation}
Using H\"{o}lder's inequality, we can deduce that
\begin{equation}\label{J41}
\begin{split}
J_{41}\leq
& C\|(1+y)^{\gamma+\al_2}D^{\al}\ome\|_{L^2}
\|(1+y)^{\gamma+\al_2+1}\px^{\al_1-i}\py ^{2k_2+1}\ome\|_{L^2}
\|\prod_{j=1}^r \px^{i_j} \ome\|_{L^\ty}\\
\leq
&C_{s,\gamma,\sigma,n,\delta}(1+\|\ome\|_{\Hg})^s.
\end{split}
\end{equation}
Similar to the treatment of \eqref{K43}-\eqref{xinhaolei}, we have
\begin{equation}\label{J42-preliminary}
\begin{split}
J_{42}\leq
& C\|(1+y)^{\gamma+\al_2}\ome^{\frac{n-1}{2}}\py D^{\al}\ome\|_{L^2}
\|(1+y)^{\gamma+\al_2}\px^{\al_1-i}\py ^{2k_2+1}\ome\|_{L^2}\\
&\cdot
\sup\limits_{0\leq y\leq 1}\|(1+y)^{-2(\gamma+\al_2)}\ome^{\frac{1-n}{2}}\prod_{j=1}^r \px^{i_j} \ome\|_{L^\ty(\T)}.
\end{split}
\end{equation}
Let $r_0$ denote the number of $i$ such that $i_j=1$, then by $\ome\in H^{s,\gamma}_{\si,\de}$, \eqref{shuaijian5} and $i_1 + \cdots + i_r = i$, we have
\begin{equation}\label{J42-2}
\begin{split}
\sup_{0\leq y\leq 1}\|(1+y)^{-2(\gamma+\alpha_2)}\ome^{\frac{1-n}{2}}\prod_{j=1}^{r}\px^{i_j}\ome\|_{L^\infty(\T)}
\leq C_{s,\gamma,\sigma,n,\delta}\sup_{0<y<1}(1+y)^{c_r} (1+\|\ome\|_{\Hg})^r,
\end{split}
\end{equation}
where
\begin{equation*}
\begin{split}
c_r=&-2(\gamma+\al_2)-\si\frac{1-n}{2}-\si r_0-\frac{(s-2)\si(r-r_0)-2\ga(r-r_0)}{s-4}\\
&+\frac{(i-r_0)(\si-\ga)}{s-4}.
\end{split}
\end{equation*}
Clearly,
\begin{equation}\label{J42-3}
\sup_{0\leq y\leq1}(1+y)^{c_r}\leq 2^{\max\{c_r,0\}}\leq C_{s,\gamma,\sigma,n}.
\end{equation}
It follows directly from \eqref{J42-preliminary}-\eqref{J42-3} that
\begin{equation}\label{J42}
\begin{split}
J_{42}
\leq
&C_{s,\gamma,\sigma,n,\delta}\|(1+y)^{\gamma+\al_2}\ome^{\frac{n-1}{2}}\py D^{\al}\ome\|_{L^2}
(1+\|\ome\|_{\Hg})^s.
\end{split}
\end{equation}
Similarly,
\begin{equation}\label{J43}
\begin{split}
J_{43}+J_{44}
\leq& C_{s,\gamma,\sigma,n,\delta}\left(1+
\sum_{\substack{|\theta|=s\\ \theta_1\leq s-1}}\|(1+y)^{\gamma+\theta_2}\ome^{\frac{n-1}{2}}\py D^\theta\ome\|_{L^2}\right)
\big(1+\|\ome\|_{\Hg}+\|\px^sU\|_{L^2(\T)}\big)^s.
\end{split}
\end{equation}
Combining $\big|\ome^{n-1-r}|_{y=0}\big|\leq C_\de$, \eqref{J41}, \eqref{J42} and \eqref{J43}, we obtain
\begin{equation}\label{J4}
\begin{split}
|J_{4}|
\leq
&\frac{n}{60s}\sum_{\substack{|\theta|=s\\ \theta_1\leq s-1}}\|(1+y)^{\gamma+\theta_2}\ome^{\frac{n-1}{2}}\py D^\theta\ome\|_{L^2}^2\\
&+C_{s,\gamma,\sigma,n,\delta,}
\Bigg[1+\|\ome\|_{\Hg}+\|\px^sU\|_{L^2(\T)}\Bigg]^{2s}.
\end{split}
\end{equation}
Therefore,
\begin{equation}\label{J3J4}
\begin{split}
\sum_{\substack{|\alpha|=s\\\alpha_1\leq s-1\\\alpha_2\ \mathrm{is\ even}}}|I_1^4|
\leq& 
\frac{n}{30}\sum_{\substack{|\theta|=s\\ \theta_1\leq s-1}}\|(1+y)^{\gamma+\theta_2}\ome^{\frac{n-1}{2}}\py D^\theta\ome\|_{L^2}^2
+\frac{n}{8}\|(1+y)^\gamma\ome^{\frac{n-1}{2}}\py\gse\|_{L^2}^2\\
&+C_{s,\gamma,\sigma,n,\delta}\Bigg[1+\|\ome\|_{\Hg}+\|\px^sU\|_{L^2(\T)}+
\sum_{l=0}^{\frac{s}{2}}\|\partial_t^l\px^{s-2l+1}\pe\|_{L^{\ty}(\T)}\Bigg]^{4(s+1)}.
\end{split}
\end{equation}

$\bullet$ If $\al_2$ is odd, we may assume without loss of generality that $\al_2=2 q+1$.  Since
$s$ is even, it is easy to see that $1\leq \al_1\leq s-1$. Using integration by parts, we obtain
\begin{align}
I^4_1
=&-n\int_{\T}(\ome^{n-1} D^{\al}\ome
\py D^{\al}\ome)|_{y=0}dx\notag \\
=&
n\int_{\T}\px(\ome^{n-1} \px^{\al_1}\py^{\al_2}\ome) \px^{\al_1-1}\py^{\al_2+1}\ome|_{y=0}dx\notag\\
=&
n\int_{\T}\px\big[\big(\px^{\al_1}(\ome^{n-1}\py^{\al_2}\ome)
-\sum_{k=1}^{\al_1}\binom{\al_1}{k}\px^k\ome^{n-1}\px^{\al_1-k}\py^{\al_2}\ome
 \big)\big]\px^{\al_1-1}\py^{\al_2+1}\ome\Big|_{y=0}dx \notag\\
=&
n\int_{\T}[\px^{\al_1+1}(\ome^{n-1}\py^{\al_2}\ome)\px^{\al_1-1}\py^{\al_2+1}\ome]|_{y=0}dx\notag\\
&
-n\sum_{k=1}^{\al_1}\binom{\al_1}{k}\int_{\T} \px^{k+1}\ome^{n-1}\px^{\al_1-k}\py^{\al_2}\ome\px^{\al_1-1}\py^{\al_2+1}\ome|_{y=0}dx\notag\\
&
-n\sum_{k=1}^{\al_1}\binom{\al_1}{k}\int_{\T}
\px^{k}\ome^{n-1}\px^{\al_1-k+1}\py^{\al_2}\ome\px^{\al_1-1}\py^{\al_2+1}\ome|_{y=0}dx\notag\\
:=&J_5+J_6+J_7.
\label{I14-odd-mixed}
\end{align}
For $J_5$, Lemma \ref{bianjiejiangjie} gives
\begin{equation}\label{J5}
\begin{split}
J_5
&=n\int_{\T}[\px^{\al_1+1}(\ome^{n-1}\py^{2q+1}\ome)\px^{\al_1-1}\py^{\al_2+1}\ome]|_{y=0}dx\\
&=
\frac{1}{n^{q}}\int_{\T}\ome^{q(1-n)}(\pt-\epsilon^2\px^2)^{q}\px^{\al_1+2} \pe\px^{\al_1-1}\py^{\al_2+1}\ome]|_{y=0}dx\\
&\quad+n\sum_{l=0}^{q}\epsilon^{2l} \sum_{j=1}^{s+1}\sum_{\rho\in A_{\al_{1}+1,q}^j}
\int_{\T}Q_{\al_{1}+1,q,\rho}^l(\ome; \mathcal{Q}_\epsilon)
\prod_{i=1}^{j}D^{\rho^i}\ome\px^{\al_1-1}\py^{\al_2+1}\ome]\Big|_{y=0}dx.
\end{split}
\end{equation}
By an estimate similar to that for $J_3$, we obtain
\begin{equation}\label{J5-1}
\begin{split}
|J_5|
\leq &
\frac{n}{30s}\sum_{\substack{|\theta|=s\\\theta_1\leq s-1}}\|(1+y)^{\gamma+\theta_2}\ome^{\frac{n-1}{2}}\py D^\theta\ome\|_{L^2}^2
+\frac{n}{12s}\|(1+y)^\gamma\ome^{\frac{n-1}{2}}\py\gse\|_{L^2}^2\\
&\quad+C_{s,\gamma,\sigma,n,\delta}
\Bigg[1+\|\ome\|_{\Hg}+\|\px^sU\|_{L^2(\T)}+
\sum_{l=0}^{\frac{s}{2}}\|\partial_t^l\px^{s-2l+1}\pe\|_{L^{\ty}(\T)}\Bigg]^{4(s+1)}.
\end{split}
\end{equation}
By using Lemma \ref{faa}, we have
\begin{equation}\label{J6-1}
\px^{k+1}\ome^{n-1}
= \sum_{r=1}^{k+1} (n-1)^{\underline{r}} \, \ome^{n-1-r} \sum_{\substack{k_1 + \cdots + k_r = k+1 \\ k_j \ge 1}} \frac{(k+1)!}{r!} \prod_{j=1}^r \frac{\px^{k_j} \ome}{k_j!}.
\end{equation}
It follows from \eqref{jiguji} that
\begin{equation}\label{J66-1}
\begin{split}
&\int_{\T}\left|(\prod_{j=1}^r \px^{k_j} \ome) \px^{\al_1-k}\py ^{\al_2}\ome \px^{\al_1-1}\py ^{\al_2+1}\ome)\Big|_{y=0}\right| dx\\
\leq
& \int_0^1\int_{\T}\Big|(\prod_{j=1}^r \px^{k_j} \ome) \px^{\al_1-k}\py ^{\al_2}\ome \px^{\al_1-1}\py ^{\al_2+1}\ome)\Big| dx dy\\
&
+\int_0^1\int_{\T}\Big|\py(\prod_{j=1}^r \px^{k_j} \ome) \px^{\al_1-k}\py ^{\al_2}\ome \px^{\al_1-1}\py ^{\al_2+1}\ome)\Big| dx dy\\
&
+\int_0^1\int_{\T}\Big|(\prod_{j=1}^r \px^{k_j} \ome) \px^{\al_1-k}\py ^{\al_2+1}\ome \px^{\al_1-1}\py ^{\al_2+1}\ome)\Big| dx dy\\
&
+\int_0^1\int_{\T}\Big|(\prod_{j=1}^r \px^{k_j} \ome) \px^{\al_1-k}\py ^{\al_2}\ome \px^{\al_1-1}\py ^{\al_2+2}\ome)\Big| dx dy\\
:=&J_{61}+J_{62}+J_{63}+J_{64}.
\end{split}
\end{equation}

For $J_{61}$, if $1\leq k\leq\al_1\leq s-3$, it is easy to obtain
\begin{equation*}
\begin{split}
J_{61}
\leq\left\|\prod_{j=1}^r\px^{k_j}\ome\right\|_{L^\infty}
\|\px^{\al_1-k}\py^{\al_2}\ome\|_{L^2}
\|\px^{\al_1-1}\py^{\al_2+1}\ome\|_{L^2}
\leq C_{s,\gamma,\sigma,\delta}\|\ome\|_{\Hg}^{s+3}.
\end{split}
\end{equation*}
If $\al_1=s-1$ and $k=1$, then $\al_2=1$ and $\sum\limits_{j=1}^r k_j=2$. Hence
\begin{equation*}
\begin{split}
J_{61}
\leq\left\|\prod_{j=1}^r\px^{k_j}\ome\right\|_{L^\infty}
\|\px^{s-2}\py\ome\|_{L^2}\|\px^{s-2}\py^2\ome\|_{L^2}
\leq C_\delta\|\ome\|_{\Hg}^2.
\end{split}
\end{equation*}
If $\al_1=s-1$ and $2\leq k\leq s-1$, using \eqref{chabuduo}, we have
\begin{equation*}
\begin{split}
J_{61}
&\leq
\left\|\prod_{j=1}^r\px^{k_j}\ome\right\|_{L^2}
\|\px^{\al_1-k}\py^{\al_2}\ome\|_{L^\infty}
\|\px^{\al_1-1}\py^{\al_2+1}\ome\|_{L^2}\\
&\leq C_{s,\gamma,\sigma,\delta}
\left(
\|\ome\|_{\Hg}
+\|\px^sU\|_{L^2(\T)}
\right)
\|\ome\|_{\Hg}^{s+2}.
\end{split}
\end{equation*}
Since the weighted lower bound gives $\|\ome\|_{\Hg}\geq c_{\gamma,\sigma,\delta}>0$, the preceding estimates imply
\begin{equation}\label{J61}
J_{61}
\leq C_{s,\gamma,\sigma,\delta}\left(\|\ome\|_{\Hg}+\|\px^sU\|_{L^2(\T)}\right)\|\ome\|_{\Hg}^{s+2}.
\end{equation}

For $J_{62}$, if $1\leq k\leq\min\{\al_1,s-4\}$, it is easy to obtain
\begin{equation*}
\begin{split}
J_{62}\leq
\left\|\py\left(\prod_{j=1}^r\px^{k_j}\ome\right)\right\|_{L^\infty}
\|\px^{\al_1-k}\py^{\al_2}\ome\|_{L^2}\|\px^{\al_1-1}\py^{\al_2+1}\ome\|_{L^2}
\leq C_{s,\gamma,\sigma,\delta}\|\ome\|_{\Hg}^{s+3}.
\end{split}
\end{equation*}
If $s-3\leq k\leq\min\{\al_1,s-2\}$, the Leibniz formula gives
\begin{equation*}
\begin{split}
J_{62}&\leq
\sum_{h=1}^r\|\py\px^{k_h}\ome\|_{L^2}
\|\px^{\al_1-k}\py^{\al_2}\ome\|_{L^\infty}\|\px^{\al_1-1}\py^{\al_2+1}\ome\|_{L^2}
\left\|\prod_{\substack{1\leq j\leq r\\j\ne h}}\px^{k_j}\ome\right\|_{L^\infty}\\
&\leq C_{s,\gamma,\sigma,\delta}\|\ome\|_{\Hg}^{s+3}.
\end{split}
\end{equation*}
If $k=\al_1=s-1$ and $r\geq2$, then $\sum_{j=1}^r k_j=s$ and every $k_j\leq s-1$.
After expanding the normal derivative, all differentiated factors have order at most $s$.
Keeping $\px^{s-2}\py^2\ome$ in $L^2$ and placing a highest-order factor in the differentiated product in $L^2$, the remaining factors are controlled in local $L^\infty$. Thus
\begin{equation*}
J_{62}\leq C_{s,\gamma,\sigma,n,\delta}
\left(1+\|\ome\|_{\Hg}+\|\px^sU\|_{L^2(\T)}\right)^{s+2}.
\end{equation*}
If $k=\al_1=s-1$ and $r=1$, then $k_1=s$. In this case, we apply $\gse$ to control the highest-order term.
Specifically, from the definition of $\gse$, it follows that
\begin{equation*}
\py\px^s\ome=\py\gse+\py\aep\,\px^s(\ue-U)+\aep\px^s\ome.
\end{equation*}
Since $\ome\in H^{s,\gamma}_{\sigma,\delta}$ and $n>1$, we have
\begin{equation}\label{cde}
\begin{split}
&\sup_{0\leq y\leq1}
\left\|(1+y)^{-2\gamma-2}\ome^{\frac{1-n}{2}}\py\ome\right\|_{L^\infty(\T)}
\leq C_{n,\delta}\, 2^{\max\{-2\gamma-3+\frac{\sigma(n-3)}2,\,0\}}
\leq C_{\gamma,\sigma,n,\delta},\\
&\|(1+y)\aep\|_{L^\infty}+\|\py\aep\|_{L^\infty}
\leq C_\delta.
\end{split}
\end{equation}
Then
\begin{equation*}
\begin{split}
J_{62}&=
\int_0^1\int_{\T}\left|\py\px^s\ome\,\py\ome\,\px^{s-2}\py^2\ome\right| dx dy\\
&=\int_0^1\int_{\T}\left|\left(\py\gse+\py\aep\,\px^s(\ue-U)+\aep\px^s\ome\right)\py\ome\,\px^{s-2}\py^2\ome\right| dx dy\\
&\leq C_{s,\gamma,\sigma,n,\delta}\|(1+y)^{\gamma+2}\px^{s-2}\py^2\ome\|_{L^2}\\
&\qquad\cdot\left(\|(1+y)^\gamma\ome^{\frac{n-1}{2}}\py\gse\|_{L^2}
+\|\px^s(\ue-U)\|_{L^2}+\|\px^s\ome\|_{L^2}\right)\\
&\leq
C_{s,\gamma,\sigma,n,\delta}\|\ome\|_{\Hg}
\left(\|(1+y)^\gamma\ome^{\frac{n-1}{2}}\py\gse\|_{L^2}+\|\px^sU\|_{L^2(\T)}+\|\ome\|_{\Hg}\right)\\
&\leq
C_{s,\gamma,\sigma,n,\delta}\left(1+\|\px^sU\|_{L^2(\T)}+\|\ome\|_{\Hg}\right)^{s+2}
\left(1+\|(1+y)^\gamma\ome^{\frac{n-1}{2}}\py\gse\|_{L^2}\right).
\end{split}
\end{equation*}
Therefore,
\begin{equation}\label{J62}
\begin{split}
J_{62}
\leq C_{s,\gamma,\sigma,n,\delta}
\left(1+\|\px^sU\|_{L^2(\T)}+\|\ome\|_{\Hg}\right)^{s+2}
\left(1+\|(1+y)^\gamma\ome^{\frac{n-1}{2}}\py\gse\|_{L^2}\right).
\end{split}
\end{equation}
similarly,
\begin{equation}\label{J63}
\begin{split}
J_{63}+J_{64}
\leq &
C_{s,\ga,\si,n,\de}
\big(1+\|\px^sU\|_{L^2(\T)}+\|\ome\|_{\Hg}\big)^{s+2}\\
&\quad\cdot\left(1+\|(1+y)^{\gamma+\al_2+1}\ome^{\frac{n-1}{2}}\px^{\al_1-1}\py^{\al_2+2}\ome\|_{L^2}\right).
\end{split}
\end{equation}
By combining \eqref{J61}, \eqref{J62}, and \eqref{J63}, one obtains
\begin{equation}\label{J6-components}
\begin{split}
|J_{6}|
\leq
&\frac{n}{120} \|(1+y)^{\gamma+\al_2+1}\ome^{\frac{n-1}{2}}\px^{\al_1-1}\py^{\al_2+2}\ome\|_{L^2}^2
+\frac{n}{12s}\|(1+y)^{\gamma}\ome^{\frac{n-1}{2}}\py \gse\|_{L^2}^2\\
&+C_{s,\ga,\si,n,\de}\big(1+\|\px^{s} U\|_{L^2(\T)}+\|\ome\|_{\Hg}\big)^{2(s+2)}.
\end{split}
\end{equation}
For $J_{7}$, an estimate similar to that for $J_{6}$ yields
\begin{equation}\label{J6}
\begin{split}
|J_{7}|
\leq
&\frac{n}{120} \|(1+y)^{\gamma+\al_2+1}\ome^{\frac{n-1}{2}}\px^{\al_1-1}\py^{\al_2+2}\ome\|_{L^2}^2
+\frac{n}{12s}\|(1+y)^{\gamma}\ome^{\frac{n-1}{2}}\py \gse\|_{L^2}^2\\
&+C_{s,\ga,\si,n,\de}\big(1+\|\px^{s} U\|_{L^2(\T)}+\|\ome\|_{\Hg}\big)^{2(s+2)}.
\end{split}
\end{equation}
Therefore,
\begin{equation}\label{J5J7}
\begin{split}
\sum_{\substack{|\al|=s\\ \al_2~\text{is~odd}}}|I^4_{1}|
\leq& 
\frac{n}{30}\sum_{\substack{|\theta|=s\\ \theta_1\leq s-1}}\|(1+y)^{\gamma+\theta_2}\ome^{\frac{n-1}{2}}\py D^\theta\ome\|_{L^2}^2
+\frac{n}{8}\|(1+y)^\gamma\ome^{\frac{n-1}{2}}\py\gse\|_{L^2}^2\\
&+C_{s,\gamma,\sigma,n,\delta}\Bigg[1+\|\ome\|_{\Hg}+\|\px^sU\|_{L^2(\T)}+
\sum_{l=0}^{\frac{s}{2}}\|\partial_t^l\px^{s-2l+1}\pe\|_{L^{\ty}(\T)}\Bigg]^{4(s+1)}.
\end{split}
\end{equation}
Combining  \eqref{I41}, \eqref{J3J4}, \eqref{J5J7}
and summing over $\alpha$, we obtain
\begin{equation*}
\begin{split}
\sum_{\substack{|\alpha|\leq s\\\alpha_1\leq s-1}}|I_1^4|
\leq &\frac{n}{10}\sum_{\substack{|\al|\leq s \\ \al_1\leq s-1}}\|(1+y)^{\ga+\al_2}\ome^{\frac{n-1}{2}}\py D^{\al}\ome\|_{L^2}^2
+\frac{n}{4}\|(1+y)^{\ga}\ome^{\frac{n-1}{2}}\py\gse\|_{L^2}^2\\
&+C_{s,\gamma,\sigma,n,\delta}
 \left[1+\|\omega_\epsilon\|_{H_g^{s,\gamma}}+\|\partial_x^sU\|_{L^2(\mathbb T)}
+\sum_{m=0}^{s/2}\|\partial_t^m\partial_x^{s-2m+1}p_\epsilon\|_{L^\infty(\mathbb T)}\right]^{4(s+1)}.
\end{split}
\end{equation*}
This establishes \eqref{K140}.
\hfill $\square$

\subsection{Estimates Only in Tangential Derivatives}\label{zuoqiexiangguji}
In this subsection, we will deal with the estimates for $\px^{s}\ome$.

Since $\ome=\py u_\epsilon$, the first equation in \eqref{feiniu2-1} can be rewritten as
\begin{equation}\label{feiniu3}
\pt \ue +\ue\px \ue +\ve\py \ue-\epsilon^2\px^2\ue-n\ome^{n-1}\py^2 \ue+\px \pe=0.
\end{equation}
Then by using Bernoulli's law \eqref{BL}, we have
\begin{equation}\label{feiniu4}
\pt (\ue-U) +\ue\px (\ue-U) +\ve\py (\ue-U)-\epsilon^2\px^2(\ue-U)-n\ome^{n-1}\py^2 (\ue-U)+(\ue-U)\px U=0.
\end{equation}


Applying the operator $\px^s$ to \eqref{feiniu4}, it yields
\begin{equation}\label{ptauu}
\begin{split}
&\quad\pt \px^s(\ue-U) +\ue\px^{s+1}(\ue-U) +\ve\py \px^s(\ue-U)-\epsilon^2\px^{s+2}(\ue-U)\\
&\quad-n\ome^{n-1}\py^2 \px^s(\ue-U)
+\px^s \ve\ome\\
=
&-\sum_{0<j\leq s}\binom{s}{j}
\px^{j}\ue\px^{s-j+1}(\ue-U)
-\sum_{0<j<s}\binom{s}{j}
\px^{j}\ve\py \px^{s-j}(\ue-U)\\
&+n\sum_{0<j\leq s}\binom{s}{j}
\px^{j}\ome^{n-1}\py^2 \px^{s-j}(\ue-U)
-\sum_{0\leq j\leq s}\binom{s}{j}
\px^{j}(\ue-U)\px^{s-j+1}U.
\end{split}
\end{equation}
On the other hand, we have
\begin{equation}\label{ptauwodu}
\begin{split}
&\pt \px^s\ome +\ue\px^{s+1}\ome +\ve\py \px^s\ome-\epsilon^2\px^{s+2}\ome
-n\ome^{n-1}\py^2 \px^s\ome
+\px^s\ve\py\ome\\
=
&-\sum_{0<j\leq s}\binom{s}{j}
\px^j\ue\px^{s-j+1}\ome
-\sum_{0<j<s}\binom{s}{j}
\px^j\ve\py \px^{s-j}\ome\\
&+n\sum_{0<j\leq s}\binom{s}{j}
\px^{j}\ome^{n-1}\py^2 \px^{s-j}\ome
+n(n-1)\sum_{0\leq j\leq s}\binom{s}{j}
\px^{j}\ome^{n-2}\px^{s-j}|\py\ome|^2.
\end{split}
\end{equation}
Subtracting $\frac{\py\ome}{\ome}\times\eqref{ptauu}$ from \eqref{ptauwodu} yields
\begin{equation}\label{ptaug}
\begin{split}
&\pt \gse +\ue\px \gse +\ve\py \gse-\epsilon^2\px^{2}\gse-n\ome^{n-1}\py^2 \gse
\\
=&-\sum_{0<j<s}\binom{s}{j}
\px^{j}\ue g^{s-j+1}_\epsilon
-\sum_{0<j<s}\binom{s}{j}
\px^{j}\ve(\py \px^{s-j}\ome-\aep\px^{s-j}\ome)\\
&+n\sum_{0<j\leq s}\binom{s}{j}
\px^{j}\ome^{n-1}(\py^2 \px^{s-j}\ome
-\aep\py \px^{s-j}\ome)\\
&+n(n-1)\sum_{0\leq j\leq s}\binom{s}{j}
\px^{j}\ome^{n-2}\px^{s-j}|\py\ome|^2
+\sum_{0\leq j<s}\binom{s}{j}
\aep\px^{j}(\ue-U)\px^{s-j+1}U\\
&-\px^{s}(\ue-U)\{2n(n-1)\aep\ome^{n-2}\py^2\ome+n\aep\py^2\ome^{n-1}
-n(n-1)\aep^3\ome^{n-1}\}\\
&+2n\ome^{n-1}\py \aep \gse-g^{1}_\epsilon\px^s U
+2\epsilon^2\px \aep\big[ \px^{s+1}(\ue-U)-\frac{\px \ome}{\ome}\px^{s}(\ue-U)\big].
\end{split}
\end{equation}
where we set $g_\epsilon^{1}:=\px\ome-\aep\px (\ue-U)$ and $\aep:=\frac{\py\ome}{\ome}$.
The derivation of \eqref{ptaug} will be provided later. Now, we are going to prove the following estimate for $\gse$.
\begin{Proposition}\label{guji-tangential}
Let $s\geq6$ be an even integer, $n>1$,
$\gamma\geq1$, $\gamma+\frac12<\sigma\leq \gamma+\frac{4}{3}$, $\epsilon\in (0,1]$ and $\de\in(0,1)$ be sufficiently small.
If $(u_\epsilon,v_\epsilon,\omega_\epsilon)$ is a classical solution of \eqref{wodu} in $[0,T]$ and satisfies
\begin{equation*}
\ome\in C([0,T]; H^{s+4,\gamma}_{\si,\de})\cap C^{1}([0,T]; H^{s+2,\gamma}),
\end{equation*}
then there exists a positive constant C, which depends on $n,~s,~\ga,~\si$ and $\de$ such that
\begin{equation}\label{faguji}
\begin{split}
&\frac{1}{2}\frac{d}{dt}\|(1+y)^{\ga}\gse\|_{L^2}^2+\frac{\epsilon^2}{2}\|(1+y)^{\ga}\px\gse\|_{L^2}^2
+\frac{n}{2}\|(1+y)^{\gamma}\ome^{\frac{n-1}{2}}\py \gse\|_{L^2}^2\\
\leq &\frac{n}{4}\|(1+y)^{\gamma}\ome^{\frac{n-1}{2}}\py^2\px^{s-1}\ome \|_{L^2}^2
+C_{s,\ga,\si,n,\de}\big(1+\|\px^{s+1}U\|_{L^2(\T)}+\|\ome\|_{\Hg}+\|\px^{s+1}\pe\|_{L^2(\T)}\big)^{2s}.
\end{split}
\end{equation}
\end{Proposition}

\textbf{Proof.}
Taking $L^2$ inner product of \eqref{ptaug} with $(1+y)^{2\ga}\gse$ yields
\begin{equation}\label{haoweixian}
\begin{split}
&\frac{1}{2}\frac{d}{dt}\|(1+y)^{\ga}\gse\|_{L^2}^2+\epsilon^2\|(1+y)^{\ga}\px\gse\|_{L^2}^2\\
=&n\int_{\Om}(1+y)^{2\ga}\ome^{n-1}\gse\py^2 \gse dxdy-\int_{\Om}(1+y)^{2\ga}\gse(\ue\px \gse +\ve\py \gse) dxdy
\\
&-\sum_{j=1}^{s-1}\binom{s}{j}\int_{\Om}(1+y)^{2\ga}\gse
\px^{j}\ue g^{s-j+1}_\epsilon dxdy\\
&-\sum_{j=1}^{s-1}\binom{s}{j}\int_{\Om}(1+y)^{2\ga}\gse
\px^{j}\ve(\py \px^{s-j}\ome-\aep\px^{s-j}\ome) dxdy\\
&+n\sum_{j=1}^{s}\binom{s}{j}
\int_{\Om}(1+y)^{2\ga}\gse
\px^{j}\ome^{n-1}(\py^2 \px^{s-j}\ome
-\aep\py \px^{s-j}\ome) dxdy\\
&+n(n-1)\sum_{j=0}^{s}\binom{s}{j}
\int_{\Om}(1+y)^{2\ga}\gse
\px^{j}\ome^{n-2}\px^{s-j}|\py\ome|^2 dxdy\\
&+\sum_{j=0}^{s-1}\binom{s}{j}
\int_{\Om}(1+y)^{2\ga}\gse
\aep\px^{j}(\ue-U)\px^{s-j+1}U dxdy\\
&-\int_{\Om}(1+y)^{2\ga}\gse
\px^{s}(\ue-U)\{2n(n-1)\aep\ome^{n-2}\py^2\ome+n\aep\py^2\ome^{n-1}
-n(n-1)\aep^3\ome^{n-1}\}dxdy\\
&+2n\int_{\Om}(1+y)^{2\ga}\ome^{n-1}\py \aep |\gse|^2 dxdy- \int_{\Om}(1+y)^{2\ga}\gse g^1_{\epsilon}\px^{s} U dxdy\\
&+2\epsilon^2\int_{\Om}(1+y)^{2\ga}\gse \px \aep\big[\px^{s+1}(\ue-U)-\frac{\px \ome}{\ome}\px^{s}(\ue-U)\big] dxdy\\
:=&\sum\limits_{i=1}^{11} S_i.
\end{split}
\end{equation}

For $S_1$, by integration by parts, we have
\begin{equation}\label{S1}
\begin{split}
S_1
=&-n\|(1+y)^{\ga}\ome^{\frac{n-1}{2}}\py \gse\|_{L^2}^2
-n(n-1)\int_{\Om}(1+y)^{2\ga}\py\ome\ome^{n-2}\gse\py \gse dxdy\\
&-2n\ga\int_{\Om}(1+y)^{2\ga-1}\ome^{n-1}\gse\py \gse dxdy
-n\int_{\T}(\ome^{n-1} \gse\py \gse )|_{y=0}dx\\
:=&\sum\limits_{i=1}^{4} S_1^i.
\end{split}
\end{equation}
For $S_1^i,~i=1,2$ and $3$, similar to $I_1^i$ in Proposition \ref{guji1}, we can deduce
\begin{equation}\label{S11}
\begin{split}
S_1^1+S_1^2+S_1^3
\leq &-\frac{7n}{8}\|(1+y)^{\ga}\ome^{\frac{n-1}{2}}\py \gse\|_{L^2}^2
+C_{s,\ga,\si,n,\de}\|\ome\|_{\Hg}^2.
\end{split}
\end{equation}
For $S_1^4$, we have
\begin{equation}\label{S14}
\begin{split}
|S_1^4|\leq& \frac{n}{8}\|(1+y)^{\gamma}\ome^{\frac{n-1}{2}}\py \gse\|_{L^2}^2
+\frac{n}{12}\|(1+y)^{\gamma}\ome^{\frac{n-1}{2}}\py^2\px^{s-1}\ome \|_{L^2}^2\\
&
+C_{s,\ga,\si,n,\de}\Big(1+\|\px^{s}U\|_{L^2(\T)}+\|\ome\|_{\Hg}+\|\px^{s+1}\pe\|_{L^2(\T)}\Big)^{2s}.
\end{split}
\end{equation}
The proof of \eqref{S14} will be provided later. Hence, we have the following estimate
\begin{equation}\label{S1dekongzhi}
\begin{split}
S_1 \leq &-\frac{3n}{4}\|(1+y)^{\gamma}\ome^{\frac{n-1}{2}}\py \gse\|_{L^2}^2
+\frac{n}{12}\|(1+y)^{\gamma}\ome^{\frac{n-1}{2}}\py^2\px^{s-1}\ome \|_{L^2}^2\\
&
+C_{s,\ga,\si,n,\de}\Big(1+\|\px^{s}U\|_{L^2(\T)}+\|\ome\|_{\Hg}+\|\px^{s+1}\pe\|_{L^2(\T)}\Big)^{2s}.
\end{split}
\end{equation}

For $S_2$, by integration by parts, we have
\begin{equation}\label{S2dekongzhi}
\begin{split}
|S_2|=&\ga|\int_{\Om}(1+y)^{2\ga-1}\ve |\gse|^2 dxdy|\\
\leq&\ga \|(1+y)^{-1}\ve\|_{L^{\ty}}
\|(1+y)^{\ga}\gse\|_{L^2}^2\\
\leq&
C_{s,\ga,\si,\de}(\|\ome\|_{H^{s, \gamma}_g}+\|\px^{s} U\|_{L^2(\T)})\|\ome\|_{H^{s, \gamma}_g}^2,
\end{split}
\end{equation}
where we use \eqref{vdeguji}.

For $S_3$,
we obtain by using Sobolev embedding inequality, and \eqref{L5}
\begin{equation}\label{aiyouwei}
\begin{split}
&\Big|\int_{\Om}(1+y)^{2\ga}\gse\px^{j}\ue g^{s-j+1}_\epsilon dxdy\Big|\\
\leq&\|(1+y)^{\ga}\gse\|_{L^2}\|\px^{j}\ue\|_{L^\ty}
\|(1+y)^{\ga} g^{s-j+1}_\epsilon\|_{L^2}
\\
\leq&C_{s,\ga,\si,\delta} \|\ome\|_{H^{s, \gamma}_g}
(\|\ome\|_{H^{s, \gamma}_g}+\|\px^s U\|_{L^2(\T)})^2,
\end{split}
\end{equation}
where we use $\|\px^{j}\ue\|_{L^\ty}\leq C_{s,\ga,\si,\delta}(\|\ome\|_{H^{s, \gamma}_g}+\|\px^s U\|_{L^2(\T)})$ for $j\in[0,s-1]$.

For $S_4$, when $j=s-1$, by using H\"{o}lder's inequality and \eqref{cde}, we obtain
\begin{equation*}
\begin{split}
&\Big|\int_{\Om}(1+y)^{2\ga}\gse
\px^{s-1}\ve(\py \px\ome-\aep\px \ome) dxdy\Big|
\\
\leq &C \|(1+y)^{\ga}\gse\|_{L^2}
\|(1+y)^{-1}\px^{s-1}\ve\|_{L^2_x L^{\ty}_y}\\
&~~~\cdot(
\|(1+y)^{\ga+1}\py \px \ome\|_{L^{\ty}_x L^2_y}
+\|(1+y)a_\epsilon\|_{L^{\ty}}
\|(1+y)^{\ga}\px \ome\|_{L^{\ty}_x L^2_y})\\
\leq &C_{s,\ga,\si,\delta}(\|\ome\|_{H^{s, \gamma}_g}+\|\px^s U\|_{L^2(\T)})\|\ome\|_{H^{s, \gamma}_g}^2.
\end{split}
\end{equation*}
When $0<j\leq s-2$, using sobolev inequality and \eqref{cde} gives
\begin{equation*}
\begin{split}
&\Big|
\int_{\Om}(1+y)^{2\ga}\gse
\px^{j}\ve(\py \px^{s-j}\ome-\aep\px^{s-j}\ome) dxdy\Big|
\\
\leq &C \|(1+y)^{\ga}\gse\|_{L^2}
\|(1+y)^{-1}\px^{j}\ve\|_{L^{\ty}}\\
&~~~\cdot(
\|(1+y)^{\ga+1}\py \px^{s-j}\ome\|_{L^2}
+\|(1+y)a_\epsilon\|_{L^{\ty}}
\|(1+y)^{\ga}\px^{s-j}\ome\|_{L^2})\\
\leq &C_{s,\ga,\si,\delta}(\|\ome\|_{H^{s, \gamma}_g}+\|\px^s U\|_{L^2(\T)})\|\ome\|_{H^{s, \gamma}_g}^2.
\end{split}
\end{equation*}
Thus, we have an estimate for $S_4$
\begin{equation}\label{S4dekongzhi}
\begin{split}
|S_4|
\leq C_{s,\ga,\si,\delta}(\|\ome\|_{H^{s, \gamma}_g}+\|\px^s U\|_{L^2(\T)})\|\ome\|_{H^{s, \gamma}_g}^2.
\end{split}
\end{equation}

For $S_5$ and $S_6$, similar to $I_4$ and $I_5$ in Proposition \ref{guji1}, we can deduce
\begin{equation}\label{S5deguji}
\begin{split}
|S_5|
\leq& \frac{n}{12}\|(1+y)^{\ga}\ome^{\frac{n-1}{2}}\py^2\px^{s-1}\ome\|_{L^2}^2
+\frac{n}{8}\|(1+y)^{\ga}\ome^{\frac{n-1}{2}}\py\gse\|_{L^2}^2\\
&+C_{s,\ga,\si,\de,n}
(1+\|\px^{s} U\|_{L^2(\T)}+\|\ome\|_{H^{s, \gamma}_g})^{s}.
\end{split}
\end{equation}
and
\begin{equation}\label{S6deguji}
\begin{split}
|S_6|\leq& \frac{n}{12}\|(1+y)^{\ga}\ome^{\frac{n-1}{2}}\py^2\px^{s-1}\ome\|_{L^2}^2
+\frac{n}{8}\|(1+y)^{\ga}\ome^{\frac{n-1}{2}}\py\gse\|_{L^2}^2\\
&+C_{s,\ga,\si,\de,n}
(1+\|\px^{s} U\|_{L^2(\T)}+\|\ome\|_{H^{s, \gamma}_g})^{s}.
\end{split}
\end{equation}

For $S_7$, we have
\begin{equation}\label{kuaiwanle}
\begin{split}
&\Big|\int_{\Om}(1+y)^{2\ga}\gse\aep\px^{j}(\ue-U)\px^{s-j+1}U dxdy\Big|\\
\leq&\|(1+y)^{\ga}\gse\|_{L^2}
\|(1+y)^{\ga-1}\px^{j}(\ue-U)\|_{L_x^{\ty}L_y^2}
\|(1+y)\aep\|_{L^{\ty}}
\|\px^{s-j+1}U\|_{L^{2}(\T)}\\
\leq&C_{s,\ga,\si,\de} \|\ome\|_{H^{s, \gamma}_g}
(\|\ome\|_{H^{s, \gamma}_g}+\|\px^s U\|_{L^2(\T)})
\|\px^{s+1}U\|_{L^2(\T)}.
\end{split}
\end{equation}
Hence, we can obtain
\begin{equation}\label{S7deguji}
\begin{split}
|S_7|
\leq C_{s,\ga,\si,\de} \|\ome\|_{H^{s, \gamma}_g}
(\|\ome\|_{H^{s, \gamma}_g}+\|\px^s U\|_{L^2(\T)})
\|\px^{s+1}U\|_{L^2(\T)}.
\end{split}
\end{equation}

For $S_8$, since $\ome\in H^{s,\gamma}_{\si,\de}$, we have by \eqref{cde}
\begin{equation*}
\begin{split}
|(1+y)\aep\ome^{n-2}\py^2\ome|=
&|(1+y)\ome^{n-3}\py\ome\py^2\ome|\\
\leq
&C_\de|(1+y)^{1+\si(3-n)-\si-1-\si-2}|\\
\leq&C_\de|(1+y)^{-(n-1)\si-2}|\\
\leq &C_\de.
\end{split}
\end{equation*}
Similarly, we can prove that
\begin{equation*}
\begin{split}
|(1+y)\aep\py^2\ome^{n-1}|\leq C_{n,\de},\quad|(1+y)\aep^3\ome^{n-1}|
\leq C_{\de}.
\end{split}
\end{equation*}
Hence, by using H\"{o}lder's inequality, it gives
\begin{equation}\label{S8deguji}
\begin{split}
|S_8|
\leq C_{s,\ga,\si,\de} \|\ome\|_{H^{s, \gamma}_g}
(\|\ome\|_{H^{s, \gamma}_g}+\|\px^s U\|_{L^2(\T)}).
\end{split}
\end{equation}

For $S_9$, it is easy to get
\begin{equation*}
\begin{split}
|\ome^{n-1}\py \aep|\leq C_\de (1+y)^{\si(1-n)-2}\leq C_\de.
\end{split}
\end{equation*}
Then there holds
\begin{equation}\label{S9deguji}
\begin{split}
|S_9|
\leq C_{s,\ga,\si,\de} \|\ome\|_{H^{s, \gamma}_g}^2.
\end{split}
\end{equation}

For $S_{10}$, it is easy to obtain
\begin{equation}\label{S10-bound}
\begin{split}
|S_{10}|
\leq &C \|(1+y)^{\ga}\gse\|_{L^2}
\|(1+y)^{\ga}g^1_{\epsilon}\|_{L^2}
\|\px^{s}U\|_{L^{\ty}(\T)}\\
\leq &C_{s,\ga,\si,\de} \|\ome\|_{H^{s, \gamma}_g}
(\|\ome\|_{H^{s, \gamma}_g}+\|\px^s U\|_{L^2(\T)})\|\px^{s}U\|_{L^{\ty}(\T)}.
\end{split}
\end{equation}

For $S_{11}$,  since $\ome\in H^{s,\gamma}_{\si,\de}$, we have
$$\|(1+y)\px \aep\|_{L^\ty}\leq C_\de,\quad \|\frac{\px \ome}{\ome}\|_{L^\ty}\leq C_\de.$$
Then,
\begin{equation}\label{S11-bound}
\begin{split}
|S_{11}|
\leq
& C\epsilon^2\|(1+y)^{\ga}\gse\|_{L^2}\|(1+y)\px \aep\|_{L^\ty}
\big(\|(1+y)^{\ga-1}\px^{s+1}(\ue-U)\|_{L^2}\\
&\quad+\|\frac{\px \ome}{\ome}\|_{L^\ty}\|(1+y)^{\ga-1}\px^{s}(\ue-U)\|_{L^2}\big)\\
\leq& C_{s,\ga,\si,\de}\epsilon^2\|\ome\|_{H^{s, \gamma}_g}
\big(\|(1+y)^{\ga-1}\px^{s+1}(\ue-U)\|_{L^2}+\|(1+y)^{\ga-1}\px^{s}(\ue-U)\|_{L^2}\big)\\
\leq& C_{s,\ga,\si,\de}\epsilon^2\|\ome\|_{H^{s, \gamma}_g}
\big(\|\px^{s+1}U\|_{L^2(\T)}+\|(1+y)^{\ga}\px\gse\|_{L^2}+\|(1+y)^{\ga-1}\px^{s}(\ue-U)\|_{L^2}\big)\\
\leq& C_{s,\ga,\si,\de}\epsilon^2\|\ome\|_{H^{s, \gamma}_g}
\big(\|\px^{s+1}U\|_{L^2(\T)}+\|(1+y)^{\ga}\px\gse\|_{L^2}+\|\ome\|_{H^{s, \gamma}_g}+\|\px^s U\|_{L^2(\T)}\big)\\
\leq& \frac{\epsilon^2}{2}\|(1+y)^{\ga}\px\gse\|_{L^2}^2+
C_{s,\ga,\si,\de}\|\ome\|_{H^{s, \gamma}_g}
\big(\|\px^{s+1}U\|_{L^2(\T)}+\|\ome\|_{H^{s, \gamma}_g}\big).
\end{split}
\end{equation}
Substituting all estimates of $S_i$ into \eqref{haoweixian} and summing over $\al$ , we can establish that \eqref{faguji} holds.

\hfill $\square$

\noindent\textit{Proof of \eqref{ptaug}.}
Applying $\py$ to \eqref{wodu}, we obtain
\begin{equation}\label{pywodu}
\begin{split}
&(\pt +\ue\px +\ve\py-\epsilon^2\px^2-n\ome^{n-1}\py^2) \py\ome\\
=&-\ome\px\ome+\px \ue\py\ome+
2n\py\ome^{n-1}\py^2\ome
+n\py^2\ome^{n-1}\py\ome.
\end{split}
\end{equation}
Then by \eqref{wodu} and \eqref{pywodu}, we can compute
\begin{equation}\label{a1}
\begin{split}
&(\pt +\ue\px +\ve\py-\epsilon^2\px^2-n\ome^{n-1}\py^2) \aep\\
=&-\epsilon^2\px^2 \aep-n\ome^{n-1}\py^2 \aep+
\frac{1}{\ome}(\pt +\ue\px +\ve\py)\py\ome
-\frac{\py\ome}{\ome^2}(\pt +\ue\px +\ve\py)\ome\\
=&2\epsilon^2\frac{\px\ome}{\ome}\px \aep-n\ome^{n-1}\py^2 \aep
+n\ome^{n-2}\py^3\ome-\px\ome+\aep\px \ue+2n(n-1)\aep\ome^{n-2}\py^2\ome\\
&+
n\aep\py^2\ome^{n-1}
-n\aep\ome^{n-2}\py^2\ome-n(n-1)\aep^3\ome^{n-1}.
\end{split}
\end{equation}
On the other hand, we can check that
\begin{equation}\label{a2}
\begin{split}
\py^2 \aep=\frac{\py^3\ome}{\ome}-\aep\frac{\py^2\ome}{\ome}-2\aep\py \aep.
\end{split}
\end{equation}
Substituting \eqref{a2} into \eqref{a1}, we get the following equation for $\aep$:
\begin{equation}\label{a3}
\begin{split}
&(\pt +\ue\px +\ve\py-\epsilon^2\px^2-n\ome^{n-1}\py^2) \aep\\
=&2\epsilon^2\frac{\px\ome}{\ome}\px \aep
-g^1_{\epsilon}+\aep\px U+2n(n-1)\aep\ome^{n-2}\py^2\ome+
n\aep\py^2\ome^{n-1}\\
&-n(n-1)\aep^3\ome^{n-1}+2n\aep\py \aep\ome^{n-1},
\end{split}
\end{equation}
where $g_\epsilon^{1}:=\px\ome-\aep\px (\ue-U)$.

Then Combining \eqref{ptauu}, \eqref{ptauwodu} and \eqref{a3} yields
\begin{equation}\label{Ig}
\begin{split}
&(\pt +\ue\px +\ve\py-\epsilon^2\px^2-n\ome^{n-1}\py^2)\gse\\
=
&-\sum_{j=1}^{s-1}\binom{s}{j}
\px^{j}\ue g^{s-j+1}_\epsilon
-\sum_{j=1}^{s-1}\binom{s}{j}
\px^{j}\ve(\py \px^{s-j}\ome-\aep\px^{s-j}\ome) \\
&+n\sum_{j=1}^{s}\binom{s}{j}
\px^{j}\ome^{n-1}(\py^2 \px^{s-j}\ome
-\aep\py \px^{s-j}\ome)
+n(n-1)\sum_{j=0}^{s}\binom{s}{j}
\px^{j}\ome^{n-2}\px^{s-j}|\py\ome|^2 \\
&+\sum_{j=0}^{s-1}\binom{s}{j}
\aep\px^{j}(\ue-U)\px^{s-j+1}U
+2n\ome^{n-1}\py \aep \gse - g^1_{\epsilon}\px^{s} U\\
&-
\px^{s}(\ue-U)\{2n(n-1)\aep\ome^{n-2}\py^2\ome+n\aep\py^2\ome^{n-1}
-n(n-1)\aep^3\ome^{n-1}\}\\
&+2\epsilon^2 \px \aep\big[\px^{s+1}(\ue-U)-\frac{\px \ome}{\ome}\px^{s}(\ue-U)\big].
\end{split}
\end{equation}
Hence \eqref{ptaug} holds.

\hfill $\square$

\noindent\textit{Proof of \eqref{S14}.}
By the boundary conditions $\eqref{feiniu}_4$, we have
\begin{equation}\label{gabianjie}
\begin{split}
&(\ome^{n-1} \gse\py \gse )\big|_{y=0}\\
=&
\{\ome^{n-1}(\px^{s}\ome+\aep\px^{s}U)
(\px^{s}\py\ome-\aep\px^{s}\ome+\py \aep\px^{s}U)\}\Big|_{y=0}\\
=&\{\ome^{n-1}\px^{s}\ome\px^{s}\py\ome
+\aep\ome^{n-1}\px^{s}U\px^{s}\py\ome
-\aep\ome^{n-1}(\px^{s}\ome)^2\\
&-\aep^2\ome^{n-1}\px^{s}U\px^{s}\ome
+\py \aep\ome^{n-1}\px^{s}\ome\px^{s}U
+\aep\py \aep\ome^{n-1}(\px^{s}U)^2
\}\Big|_{y=0}.
\end{split}
\end{equation}
Recall that
\begin{equation*}
S_1^4=-n\int_{\T}(\ome^{n-1}\gse\py\gse)|_{y=0} dx.
\end{equation*}
First, since $\ome\in H^{s,\gamma}_{\si,\de}$, we have
\begin{equation*}
|(\aep,\py \aep, \ome^{n-1})|_{y=0}\leq C_{n,\de},
\end{equation*}
and
\begin{equation*}
\begin{split}
\sup_{0\leq y\leq1}\|(1+y)^{-\gamma}\ome^{\frac{1-n}{2}}\|_{L^\infty(\T)}
\leq C_\delta\sup_{0\leq y\leq1}(1+y)^{\frac{\sigma(n-1)}2-\gamma}
=C_\delta 2^{\max\{\frac{\sigma(n-1)}2-\gamma,\,0\}}
\leq C_{\gamma,\sigma,n,\delta}.
\end{split}
\end{equation*}
Combining the above two estimates with H\"{o}lder's inequality and \eqref{L1.5} yields
\begin{align}
&n\Big|\int_{\T}\{-\aep\ome^{n-1}(\px^{s}\ome)^2
-\aep^2\ome^{n-1}\px^{s}U\px^{s}\ome
+\py \aep\ome^{n-1}\px^{s}\ome\px^{s}U
+\aep\py \aep\ome^{n-1}(\px^{s}U)^2
\}|_{y=0}dx\Big|\notag\\
\leq & C_{s,\ga,\si,n,\de} \left(\|\px^{s}\ome|_{y=0}\|^2_{L^2(\T)} + \|\px^{s}U\|^2_{L^2(\T)}\right)\notag\\
\leq &C_{s,\ga,\si,n,\de}\left(\|\px^s\ome\|_{L^2}^2+\int_0^1\|\px^s\ome\|_{L^2(\T)}\|\py\px^s\ome\|_{L^2(\T)}\,dy
+\|\px^sU\|_{L^2(\T)}^2\right)\notag\\
\leq &C_{s,\ga,\si,n,\de}\left(\|\px^s\ome\|_{L^2}^2+\sup_{0\leq y\leq1}\|(1+y)^{-\gamma}\ome^{\frac{1-n}{2}}\|_{L^\infty(\T)}
\|\px^s\ome\|_{L^2}\|(1+y)^\gamma\ome^{\frac{n-1}{2}}\py g_\epsilon^s\|_{L^2}\right.\notag\\
&\qquad\qquad\quad\left.+\|\px^sU\|_{L^2(\T)}^2\right)\notag\\
\leq &C_{s,\ga,\si,n,\de}\left( (\|\ome\|_{\Hg}+\|\px^{s}U\|_{L^2(\T)})^2
+\|(1+y)^{\gamma}\ome^{\frac{n-1}{2}}\py\gse\|_{L^2}(\|\ome\|_{\Hg}+\|\px^{s}U\|_{L^2(\T)})\right)\notag\\
\leq &
\frac{n}{8}\|(1+y)^{\gamma}\ome^{\frac{n-1}{2}}\py\gse\|_{L^2}^2+
C_{s,\ga,\si,n,\de}(1+\|\ome\|_{\Hg}+\|\px^{s}U\|_{L^2(\T)})^2.
\label{jiandanbufen}
\end{align}
Next, by the boundary conditions $\eqref{wodu}_3$, we have
\begin{equation}\label{nandian}
\begin{split}
\ome^{n-1}\px^{s}\py\ome|_{y=0}=
\frac{1}{n}\px^{s+1} \pe- \sum_{j=1}^s\binom{s}{j}
\px^{j}\ome^{n-1}\px^{s-j}\py\ome|_{y=0}.
\end{split}
\end{equation}
Hence, putting \eqref{nandian} into \eqref{gabianjie}, we have
\begin{equation}\label{sufu}
\begin{split}
&n\left|\int_{\T}\left(\ome^{n-1}\px^s\ome\,\px^s\py\ome+\aep\ome^{n-1}\px^sU\,\px^s\py\ome\right)\Big|_{y=0} dx\right|\\
&\leq\left|\int_{\T}(\px^s\ome+\aep\px^sU)|_{y=0}\,\px^{s+1}\pe dx\right|\\
&\quad+n\sum_{j=1}^s\binom{s}{j}
\left|\int_{\T}\left\{(\px^s\ome+\aep\px^sU)\px^j(\ome^{n-1})\px^{s-j}\py\ome\right\}\Big|_{y=0} dx\right|.
\end{split}
\end{equation}
Similar to the estimate in \eqref{jiandanbufen}, it follows that
\begin{equation}\label{sufu1}
\begin{split}
&\left|\int_\T(\px^s\ome+\aep\px^sU)\px^{s+1}\pe\big|_{y=0} dx\right|
=\left|\int_\T\gse\big|_{y=0}\px^{s+1}\pe dx\right|\\
\leq&\|\px^{s+1}\pe\|_{L^2(\T)}
\left(\int_0^1\|\gse\|_{L^2(\T)} dy
+\int_0^1\|\py\gse\|_{L^2(\T)} dy\right)\\
\leq& C_{\gamma,\sigma,n,\delta}\|\px^{s+1}\pe\|_{L^2(\T)}
\left(\|\ome\|_{\Hg}
+\|(1+y)^\gamma\ome^{\frac{n-1}{2}}\py\gse\|_{L^2}\right)\\
\leq&\frac{n}{24}\|(1+y)^\gamma\ome^{\frac{n-1}{2}}\py\gse\|_2^2
+C_{\gamma,\sigma,n,\delta}
\left(\|\ome\|_{\Hg}+\|\px^{s+1}\pe\|_{L^2(\T)}\right)^2.
\end{split}
\end{equation}
By using Lemma \ref{faa}, we have
\begin{equation}\label{sufu1-1}
\px^{j}\ome^{n-1}
= \sum_{r=1}^{j} (n-1)^{\underline{r}} \, \ome^{n-1-r} \sum_{\substack{j_1 + \cdots + j_r = j \\ j_i \ge 1}} \frac{j!}{r!} \prod_{i=1}^r \frac{\px^{j_i} \ome}{j_i!}.
\end{equation}
The same trace and product estimates as in \eqref{J2-bound}-\eqref{J4}, together with \eqref{jiguji}, give
\begin{equation}\label{sufu2}
\begin{split}
&n\sum_{j=1}^s\binom{s}{j}
\left|\int_{\T}\left\{(\px^s\ome+\aep\px^sU)\px^j(\ome^{n-1})\px^{s-j}\py\ome\right\}\Big|_{y=0} dx\right|\\
&\leq
\frac{n}{24}\|(1+y)^\gamma\ome^{\frac{n-1}{2}}\py\gse\|_{L^2}^2
+\frac{n}{24}\|(1+y)^\gamma\ome^{\frac{n-1}{2}}\py^2\px^{s-1}\ome\|_{L^2}^2\\
&\quad+C_{s,\gamma,\sigma,n,\delta}
\left(1+\|\px^sU\|_{L^2(\T)}+\|\ome\|_{\Hg}
\right)^{2s}.
\end{split}
\end{equation}
Combining \eqref{jiandanbufen} and \eqref{sufu}-\eqref{sufu2}, we obtain
\begin{equation*}
\begin{split}
|S_1^4|
&\leq
\frac{n}{8}\|(1+y)^\gamma\ome^{\frac{n-1}{2}}\py\gse\|_{L^2}^2
+\frac{n}{12}\|(1+y)^\gamma\ome^{\frac{n-1}{2}}\py^2\px^{s-1}\ome\|_{L^2}^2\\
&\quad+C_{s,\gamma,\sigma,n,\delta}
\left(1+\|\px^sU\|_{L^2(\T)}+\|\ome\|_{\Hg}+\|\px^{s+1}\pe\|_{L^2(\T)}\right)^{2s}.
\end{split}
\end{equation*}
This proves \eqref{S14}.

\hfill$\square$

\subsection{Weighted estimation of $\ome$ and lower-order terms
}\label{jianchi}
In this subsection, we will derive the weighted $H^s$ estimate for $\ome$ and the weighted $L^\ty$ estimate for lower-order derivatives.
\begin{Proposition}
\label{guji-combined}
Under the assumptions of Proposition \ref{guji1}, the following estimate holds:
\begin{equation}\label{omeguji}
\begin{split}
\|\ome\|_{H^{s, \gamma}_g}^2
\leq& \big(\|\om_0\|_{H^{s, \gamma}_g}^2+\int_0^ t \phi(\tau) d\tau\big)
\Big\{1-(2s+1)C_{s,\gamma,\sigma,n,\delta}\big(\|\om_0\|_{H^{s, \gamma}_g}^2+\int_0^ t \phi(\tau) d\tau\big)^{2s+1}t\Big\}^{-\frac{1}{2s+1}},
\end{split}
\end{equation}
where
\begin{equation}\label{phit}
\begin{split}
\phi(t)=C_{s,\gamma,\sigma,n,\delta}\big((1+\|\px^{s+1} U\|_{L^2(\T)})
+\sum_{l=0}^{\frac{s}{2}}\|\partial_t^l\px^{s-2l+1}\pe\|_{L^{\ty}(\T)}\big)^{4(s+1)}.
\end{split}\end{equation}
\end{Proposition}
\noindent{\bf Proof.} From Proposition \ref{guji1}, Proposition \ref{guji-tangential} and the definition of $\|\ome\|_{H^{s, \gamma}_g}$, it follows that
\begin{equation}\label{combined-energy-ode}
\begin{split}
\frac{d}{dt}\|\ome\|_{\Hg}^2
\leq &
C_{s,\gamma,\sigma,n,\delta}
\Big((1+\|\px^{s+1} U\|_{L^2(\T)}+\|\ome\|_{\Hg}
+\sum_{l=0}^{\frac{s}{2}}\|\partial_t^l\px^{s-2l+1}\pe\|_{L^{\ty}(\T)}\Big)^{4(s+1)}\\
\leq &
C_{s,\gamma,\sigma,n,\delta}\|\ome\|_{\Hg}^{4(s+1)}+\phi(t).
\end{split}
\end{equation}
Consequently, the comparison principle of ordinary differential equations gives that
\begin{equation}\label{energy-ode-solution}
\begin{split}
\|\ome\|_{H^{s, \gamma}_g}^2
\leq& \big(\|\om_0\|_{H^{s, \gamma}_g}^2+\int_0^ t \phi(\tau) d\tau\big)
\Big\{1-(2s+1)C_{s,\gamma,\sigma,n,\delta}\big(\|\om_0\|_{H^{s, \gamma}_g}^2+\int_0^ t \phi(\tau) d\tau\big)^{2s+1}t\Big\}^{-\frac{1}{2s+1}}
\end{split}
\end{equation}
when
\begin{equation}\label{energy-ode-condition}
\begin{split}
1-(2s+1)C_{s,\gamma,\sigma,n,\delta}\big(\|\om_0\|_{H^{s, \gamma}_g}^2+\int_0^ t \phi(\tau) d\tau\big)^{2s+1}t\geq 0.
\end{split}
\end{equation} Choosing $t>0$ small enough ensures this condition holds. This completes the proof of Proposition \ref{guji-combined}.

\hfill $\square$

Next, we will derive the $L^\ty$ estimate for $\sum\limits_{|\alpha| \leq 2}\left|(1+y)^{\sigma+\alpha_2} D^\alpha \ome\right|^2$ and the lower bound estimate for $(1+y)^{\si}\ome$.

\begin{Proposition}\label{xianyanguji-linfty}
Under the same regularity and positivity assumptions as in Proposition \ref{guji1},
assume in addition that $n>1$ and $0<\epsilon\leq1$.  Define
\begin{equation*}
 I(t):=\sum_{|\alpha| \leq 2}\left|(1+y)^{\sigma+\alpha_2}D^\alpha\ome(t)\right|^2,
 \qquad
 A_\epsilon(t):=\sup_{0\leq\tau\leq t}\|\ome(\tau)\|_{\Hg},
\end{equation*}
and
\begin{equation*}
 G_\epsilon(t):=A_\epsilon(t)
 +\sup_{0\leq\tau\leq t}\|\px^sU(\tau)\|_{L^2(\T)},
 \qquad
 \Lambda_\epsilon(t):=C_{s,\gamma,\sigma,n,\delta}\bigl(1+G_\epsilon(t)\bigr).
\end{equation*}
Then the boundary-growth estimate gives
\begin{equation}\label{It-growth}
 \|I(t)\|_{L^\ty(\Om)}
 \leq\left\{\|I(0)\|_{L^\ty(\Om)}
 +C_{s,\gamma,\sigma,n,\delta}\bigl(1+A_\epsilon(t)\bigr)A_\epsilon^2(t)t\right\}
 e^{\Lambda_\epsilon(t)t}.
\end{equation}
Moreover, for all sufficiently small $t>0$,
\begin{equation}\label{wxiajie}
\begin{split}
\min_\Om(1+y)^{\si}\ome(t)
\geq{}&\left\{1-\Lambda_\epsilon(t)t
e^{\Lambda_\epsilon(t)t}\right\}\\
&\quad\times\left\{
\min_\Om(1+y)^{\si}\om_0
-C_{s,\gamma,\sigma,n,\delta}
\bigl(1+A_\epsilon(t)\bigr)A_\epsilon(t)t\right\}.
\end{split}
\end{equation}
\end{Proposition}
\noindent{\bf Proof.} For $|\al|\leq 2$, we let $$B_{\al}:=(1+y)^{\sigma+\al_2} D^\alpha \ome,$$ then
$I(t)=\sum\limits_{|\alpha| \leq 2}|B_{\al}|^2$ and $(1+y)^{\si}\ome=B_{(0,0)}$. By a direct computation, we deduce that $B_{\al}$ satisfies
\begin{equation}\label{bbll}
\begin{split}
(\pt +\ue\px +\ve\py-\epsilon^2\px^2-n\ome^{n-1}\py^2) B_{\al}
=
D_{\al}B_{\al}+E_{\al}\py B_{\al}+F_{\al},
\end{split}
\end{equation}
where
\begin{equation*}
\begin{split}
D_{\al}=(\si+\al_2)\frac{\ve}{1+y}
+n(\si+\al_2)(\si+\al_2+1)\frac{\ome^{n-1}}{(1+y)^2},
\quad
E_{\al}=-2n(\si+\al_2)\frac{\ome^{n-1}}{1+y},
\end{split}
\end{equation*}
and
\begin{equation*}
\begin{split}
F_{\al}=
&-(1+y)^{\si+\al_2}\sum_{0<\beta\leq\al}\binom{\al}{\beta}
\left(D^{\beta}\ue\px D^{\al-\beta}\ome+D^{\beta}\ve\py D^{\al-\beta}\ome\right)\\
&+n(1+y)^{\si+\al_2}\sum_{0<\beta\leq\al}\binom{\al}{\beta}
D^{\beta}\ome^{n-1}\py^2 D^{\al-\beta}\ome\\
&+n(n-1)(1+y)^{\si+\al_2}\sum_{0\leq\beta\leq\al}\binom{\al}{\beta}
D^{\beta}\ome^{n-2}D^{\al-\beta}|\py\ome|^2.
\end{split}
\end{equation*}
After multiplying equation \eqref{bbll} by $B_{\al}$  and summing over $\al$ with $|\al|\leq 2$, the evolution equation for $I$ is derived as follows
\begin{equation}\label{iiii}
\begin{split}
&(\pt +\ue\px +\ve\py-\epsilon^2\px^2-n\ome^{n-1}\py^2) I\\
=&-2\sum_{|\alpha| \leq 2}\Big(\epsilon^2|\px B_{\al}|^2+n\ome^{n-1}|\py B_{\al}|^2\Big)
+2 \sum_{|\alpha| \leq 2}\Big(D_{\al}|B_{\al}|^2+E_{\al}B_{\al}\py B_{\al}+F_{\al}B_{\al}\Big)
.
\end{split}
\end{equation}
Using $\ome\in H^{s,\gamma}_{\si,\de}$ and $\|\frac{\ve}{1+y}\|_{L^\ty}\leq C_{s,\gamma,\sigma,n,\delta}(\|\ome\|_{\Hg}+\|\px^s U\|_{L^2(\T)})$, we obtain
\begin{equation}\label{DE}
\begin{split}
|D_{\al}|\leq C_{s,\gamma,\sigma,n,\delta} (1+\|\px^s U\|_{L^2(\T)}+\|\ome\|_{\Hg}),\quad
|E_{\al}|\leq C_{\sigma,n}\frac{\ome^{n-1}}{1+y}.
\end{split}
\end{equation}
From \eqref{DE}, it follows that
\begin{equation}\label{i11}
\begin{split}
&\sum_{|\alpha| \leq 2}\Big(D_{\al}|B_{\al}|^2+E_{\al}B_{\al}\py B_{\al}\Big)\\
\leq
&C_{s,\gamma,\sigma,n,\delta} (1+\|\px^s U\|_{L^2(\T)}+\|\ome\|_{\Hg})\sum_{|\alpha| \leq 2}|B_{\al}|^2
+C_{\sigma,n}\sum_{|\alpha| \leq 2}B_{\al} \ome^{\frac{n-1}{2}}\py B_{\al}
\|(1+y)^{-1}\ome^{\frac{n-1}{2}}\|_{L^\ty}\\
\leq
&\frac{n}{4}\sum_{|\alpha| \leq 2}|\ome^{\frac{n-1}{2}}\py B_{\al}|^2+C_{s,\gamma,\sigma,n,\delta} (1+\|\px^s U\|_{L^2(\T)}+\|\ome\|_{\Hg})\sum_{|\alpha| \leq 2}|B_{\al}|^2.
\end{split}
\end{equation}
On the other hand, for $D^{\beta}\ue$, if $\beta_2>0$, then by \eqref{chabuduo}
\begin{equation*}
\begin{split}
|(1+y)^{\beta_2}D^{\beta}\ue|=&|(1+y)^{\beta_2}D^{\beta-e_2}\ome|\\
\leq &|(1+y)^{\ga+\beta_2-1}D^{\beta-e_2}\ome|
\leq C_{s,\gamma,\sigma,n,\delta} (\|\px^s U\|_{L^2(\T)}+\|\ome\|_{\Hg}).
\end{split}
\end{equation*}
If $|\beta_2|=0$, then by \eqref{L5}
\begin{equation*}
\begin{split}
|D^{\beta}\ue|=|\py^{-1}D^{\beta}\ome|
\leq C_{s,\gamma,\sigma,n,\delta} (\|\px^s U\|_{L^2(\T)}+\|\ome\|_{\Hg}).
\end{split}
\end{equation*}
Combining the above two cases, we have
\begin{equation}\label{xixi}
\begin{split}
|(1+y)^{\beta_2}D^{\beta}\ue|\leq C_{s,\gamma,\sigma,n,\delta} (\|\px^s U\|_{L^2(\T)}+\|\ome\|_{\Hg}).
\end{split}
\end{equation}
For $D^{\beta}\ve$, similarly, we have
\begin{equation}\label{xixiv}
\begin{split}
|(1+y)^{\beta_2-1}D^{\beta}\ve|
\leq C_{s,\gamma,\sigma,n,\delta} (1+\|\px^s U\|_{L^2(\T)}+\|\ome\|_{\Hg}).
\end{split}
\end{equation}
Hence
\begin{equation}\label{i111}
\begin{split}
&\sum_{|\alpha| \leq 2}(1+y)^{\si+\al_2}\Big(D^{\beta}\ue\px D^{\al-\beta}\ome+D^{\beta}\ve\py D^{\al-\beta}\ome\Big) B_{\al}\\
\leq
&C_{s,\gamma,\sigma,n,\delta}(1+\|\px^s U\|_{L^2(\T)}+\|\ome\|_{\Hg})\sum_{|\alpha| \leq 2}
(|B_{\al-\beta+e_1}|+|B_{\al-\beta+e_2}|)|B_{\al}|\\
\leq
&C_{s,\gamma,\sigma,n,\delta}(1+\|\px^s U\|_{L^2(\T)}+\|\ome\|_{\Hg})\sum_{|\alpha| \leq 2}|B_{\al}|^2.
\end{split}
\end{equation}
Similarly, from $|\al|\leq 2$ and $0\leq \beta\leq \al$, we obtain
\begin{equation}\label{i112}
\begin{split}
\sum_{|\alpha| \leq 2}(1+y)^{\si+\al_2}D^{\beta}\ome^{n-2}D^{\al-\beta}|\py\ome|^2 B_{\al}
\leq
C_{s,\gamma,\sigma,n,\delta}(1+\|\px^s U\|_{L^2(\T)}+\|\ome\|_{\Hg})\sum_{|\alpha| \leq 2}|B_{\al}|^2.
\end{split}
\end{equation}
When $|\al|=2$ and $|\beta|=1$, the term $\py^2 D^{\al-\beta}\ome\notin I(t)$, which will cause difficulties. Using $\ome\in H^{s,\gamma}_{\si,\de}$ and $n>1$, we obtain
\begin{equation}\label{i113}
\begin{split}
&\sum_{|\alpha| = 2}(1+y)^{\si+\al_2}D^{\beta}\ome^{n-1}\py^2 D^{\al-\beta}\ome B_{\al}\\
=
&(n-1)\sum_{|\alpha| = 2}(1+y)^{\si+\al_2}\ome^{n-2}D^{\beta}\ome\py D^{\al-\beta+e_2}\ome B_{\al}\\
\leq
&C_n\sum_{|\alpha| = 2}|(1+y)^{\si+\al_2-\beta_2}\ome^{\frac{n-1}{2}}\py D^{\al-\beta+e_2}\ome|
|(1+y)^{\beta_2}\ome^{\frac{n-3}{2}}D^{\beta}\ome||B_{\al}|\\
\leq
&\frac{n}{4}\sum_{|\alpha| \leq 2}|\ome^{\frac{n-1}{2}}\py B_{\al}|^2+C_{s,\gamma,\sigma,n,\delta} |B_{\al}|^2.
\end{split}
\end{equation}
Except for this case, we easily obtain
\begin{equation}\label{i114}
\begin{split}
\sum_{|\alpha|\leq 2}(1+y)^{\si+\al_2}D^{\beta}\ome^{n-1}\py^2 D^{\al-\beta}\ome B_{\al}
\leq
C_{s,\gamma,\sigma,n,\delta}(1+\|\px^s U\|_{L^2(\T)}+\|\ome\|_{\Hg})\sum_{|\alpha| \leq 2}|B_{\al}|^2.
\end{split}
\end{equation}
Combining \eqref{iiii}, \eqref{i11} and \eqref{i111}-\eqref{i114}, we obtain
\begin{equation}\label{iiie}
\begin{split}
(\pt +\ue\px +\ve\py-\epsilon^2\px^2-n\ome^{n-1}\py^2) I
\leq
C_{s,\gamma,\sigma,n,\delta}(1+\|\px^s U\|_{L^2(\T)}+\|\ome\|_{\Hg})I.
\end{split}
\end{equation}

By the definition of $G_\epsilon(t)$,
$$C_{s,\gamma,\sigma,n,\delta}(1+\|\px^s U\|_{L^2(\T)}+\|\ome\|_{\Hg})\leq \Lambda_\epsilon(t).$$
Using the classical maximum principle for parabolic equations (see Lemma \ref{jida}) to $I(t)$ yields
\begin{equation}\label{jidazhi}
\begin{split}
\|I(t)\|_{L^{\ty}}
\leq \max\left\{
e^{\Lambda_\epsilon(t)t}\|I(0)\|_{L^{\ty}},~
\max\limits_{\tau\in[0,t]}\{e^{\Lambda_\epsilon(t)(t-\tau)}\|I(\tau)|_{y=0}\|_{L^{\ty}(\T)}\}
\right\}.
\end{split}
\end{equation}

Thus it remains to estimate the boundary value in \eqref{jidazhi}. Using the equation for $D^{\al}\ome$, we have
\begin{equation}\label{DW}
\begin{split}
\pt D^{\al}\ome|_{y=0}=\epsilon^2\px^2D^{\al}\ome|_{y=0} +n\ome^{n-1}\py^2 D^{\al}\ome|_{y=0}+Q_\al^1|_{y=0}+Q_\al^2|_{y=0}+Q_\al^3|_{y=0},
\end{split}
\end{equation}
for any $|\al|\leq2$, and where
\begin{equation}\label{Ea123}
\begin{split}
Q_\al^1&=-\sum_{0<\beta\leq\al}\binom{\al}{\beta}
\left(D^{\beta}\ue\px D^{\al-\beta}\ome +D^{\beta}\ve\py D^{\al-\beta}\ome\right),\\
Q_\al^2&=n\sum_{0<\beta\leq\al}\binom{\al}{\beta}
D^{\beta}\ome ^{n-1}\py^2 D^{\al-\beta}\ome,\\
Q_\al^3&=n(n-1)\sum_{0\leq\beta\leq\al}\binom{\al}{\beta}
D^{\beta}\ome^{n-2}D^{\al-\beta}|\py\ome|^2.
\end{split}
\end{equation}
From Lemma \ref{chazhi} and $\ue|_{y=0}=\ve|_{y=0}=0$, it follows that
\begin{equation}\label{Qa1}
\begin{split}
\|Q_\al^1|_{y=0}\|_{L^\ty(\T)}\leq C_{s,\gamma,\sigma,n,\delta}\|\ome\|_{\Hg}^2.
\end{split}
\end{equation}
Similarly, for $s\geq 6$, it holds that
\begin{equation}\label{Qa23}
\begin{split}
&\|Q_\al^2|_{y=0}\|_{L^\ty(\T)}\leq C_{s,\gamma,\sigma,n,\delta}\|\ome\|_{\Hg}^2,\\
&\|Q_\al^3|_{y=0}\|_{L^\ty(\T)}\leq C_{s,\gamma,\sigma,n,\delta}\|\ome\|_{\Hg}^2.
\end{split}
\end{equation}
In addition, for $\epsilon\in[0,1]$, we get
\begin{equation}\label{Qaa}
\begin{split}
\|\epsilon^2\px^2D^{\al}\ome|_{y=0} +n\ome^{n-1}\py^2 D^{\al}\ome|_{y=0}\|_{L^\ty(\T)}\leq C_{s,\gamma,\sigma,n,\delta}(\|\ome\|_{\Hg}+\|\ome\|_{\Hg}^2).
\end{split}
\end{equation}
Therefore, combining \eqref{DW} and \eqref{Qa1}-\eqref{Qaa}, for $s\geq 6$, we obtain
\begin{equation}\label{ptI}
\begin{split}
\|\pt I|_{y=0}\|_{L^\ty(\T)}\leq &C_{s,\gamma,\sigma,n,\delta}(1+\|\ome\|_{\Hg})\|\ome\|_{\Hg}^2,
\end{split}
\end{equation}
which implies by direct integration
\begin{equation}\label{bianjie22}
\begin{split}
\|I(t)|_{y=0}\|_{L^{\ty}(\T)}
\leq
&\|I(0)|_{y=0}\|_{L^{\ty}(\T)}+\int_{0}^{t}\|\pt I(\tau)|_{y=0}\|_{L^{\ty}(\T)}d\tau\\
\leq
&\|I(0)\|_{L^{\ty}}+C_{s,\gamma,\sigma,n,\delta}(1+A_\epsilon(t))A_\epsilon^2(t)t.
\end{split}
\end{equation}
which shows \eqref{It-growth}.

Next, we will give a lower bound estimate for $(1+y)^\sigma \ome$. Clearly,
\begin{equation*}
B_{(0,0)}:=(1+y)^\sigma\ome.
\end{equation*}
Taking $\alpha=(0,0)$ in \eqref{bbll} and moving the first-order term $E_{(0,0)}\py B_{(0,0)}$ to the left-hand side, we obtain
\begin{equation}\label{eqB0}
\begin{split}
&\left\{\pt+\ue\px+
\left(\ve+2n\sigma\frac{\ome^{n-1}}{1+y}\right)\py-\epsilon^2\px^2-n\ome^{n-1}\py^2\right\}B_{(0,0)}\\
&\qquad=D_{(0,0)}B_{(0,0)}+n(n-1)(1+y)^\sigma\ome^{n-2}|\py\ome|^2,
\end{split}
\end{equation}
where
\begin{equation*}
D_{(0,0)}=\sigma\frac{\ve}{1+y}+n\sigma(\sigma+1)\frac{\ome^{n-1}}{(1+y)^2}.
\end{equation*}
Since $n>1$, the last term on the right-hand side of \eqref{eqB0} is nonnegative.  Furthermore, the weighted Sobolev estimates used in \eqref{DE}, together with the definition of
$G_\epsilon(t)$, imply
\begin{equation*}
\left\|D_{(0,0)}\right\|_{L^\ty([0,t]\times\Om)}
+\left\|\frac{1}{1+y}\left(\ve+2n\sigma\frac{\ome^{n-1}}{1+y}\right)\right\|_{L^\ty([0,t]\times\Om)}
\leq\Lambda_\epsilon(t).
\end{equation*}
Consequently, the proof of the minimum principle in Lemma \ref{jixiao}
yields
\begin{equation}\label{minB0}
\min_\Om B_{(0,0)}(t)\geq
\left\{1-\Lambda_\epsilon(t)t e^{\Lambda_\epsilon(t)t}\right\}\kappa_\epsilon(t),
\end{equation}
where
\begin{equation*}
\kappa_\epsilon(t):=\min\left\{
\min_\Om(1+y)^\sigma\om_0,
\min_{[0,t]\times\T}\ome\big|_{y=0}\right\}.
\end{equation*}

It remains to estimate the boundary value in this expression.  Because $\ue|_{y=0}=\ve|_{y=0}=0$, the vorticity equation gives
\begin{equation}\label{ptomega}
\pt\ome\big|_{y=0}
=\left\{\epsilon^2\px^2\ome+n\ome^{n-1}\py^2\ome
+n(n-1)\ome^{n-2}|\py\ome|^2\right\}\Big|_{y=0}.
\end{equation}
The trace estimate in Lemma \ref{chazhi}, the pointwise bounds defining $H^{s,\gamma}_{\sigma,\delta}$, and $0<\epsilon\leq1$ show, for
$s\geq 6$, that
\begin{equation}\label{ptomelwq}
\left\|\pt\ome\big|_{y=0}\right\|_{L^\ty(\T)}
\leq C_{s,\gamma,\sigma,n,\delta}
\bigl(1+A_\epsilon(t)\bigr)A_\epsilon(t).
\end{equation}
Indeed, the first two terms in \eqref{ptomega} are the case $\alpha=0$ of \eqref{Qaa}; the remaining quadratic term is estimated in exactly the same way as $Q_{(0,0)}^3$ in \eqref{Qa23}.  Integrating \eqref{ptomelwq} in time gives
\begin{equation}\label{minomega}
\begin{split}
\min_{[0,t]\times\T}\ome\big|_{y=0}
&\geq\min_\T\om_0\big|_{y=0}-C_{s,\gamma,\sigma,n,\delta}
\bigl(1+A_\epsilon(t)\bigr)A_\epsilon(t)t\\
&\geq\min_\Om(1+y)^\sigma\om_0-C_{s,\gamma,\sigma,n,\delta}
\bigl(1+A_\epsilon(t)\bigr)A_\epsilon(t)t.
\end{split}
\end{equation}
For $t$ sufficiently small, the last member is nonnegative.  Substituting \eqref{minomega} into \eqref{minB0} proves \eqref{wxiajie}.  In particular, if $\min\limits_\Om(1+y)^\sigma\om_0\geq2\delta$, then the right-hand side of \eqref{wxiajie} is at least $\delta$ after decreasing the lifespan if necessary; hence the weighted monotonicity is propagated.

\hfill $\square$
\section{Local-in-time existence and uniqueness}\label{jixian}

In this section we complete the construction of a solution to the velocity problem \eqref{feiniu2}.
The argument follows the compactness and nonlinear cancellation method of \cite[Section 6]{NM}.
Two points require a modification in the present power-law setting.
First, the normal diffusion is nonlinear:
$$n\ome^{n-1}\py^2\ome+n(n-1)\ome^{n-2}|\py\ome|^2=\py^2(\ome^n).$$
Second, in the uniqueness argument the difference of the nonlinear fluxes is linearized by an averaged coefficient.

\subsection{Uniform Bounds and lifespan for $\ome$}

Lemma \ref{lemmalocal} guarantees the existence of a solution $\ome$ to the regularized Prandtl equations \eqref{wodu} on the time interval $[0,T_{\varepsilon}]$,  whose length may depend on $\varepsilon$. On the other hand, Proposition \ref{guji-combined}-\ref{xianyanguji-linfty} provides estimates for \(\omega_\varepsilon\) that are uniform with respect to $\varepsilon$. By applying a standard continuity argument, the solution can therefore be extended to a time interval independent of $\varepsilon$. Consequently, we obtain the following result
\begin{Proposition}
\label{lem:regularized-solvability}
Under the assumptions of Lemma~\ref{lemmalocal}, there exists a uniform lifespan
$T:=T\left(s,\gamma,\sigma,n,\delta,\|\omega_0\|_{H_g^{s,\gamma}},U\right)>0$
independent of $\varepsilon$.
The regularized vorticity system \eqref{wodu} with initial data $\omega_0^\varepsilon$ has a solution
$\omega_\varepsilon\in
C\left([0,T];H_{\sigma,\delta}^{s,\gamma}\right)\cap C^1\left([0,T];H^{s-2,\gamma}\right)$
satisfying the following uniform (in $\varepsilon$) estimates for all $0<\varepsilon\leq\varepsilon_0$ and $t\in[0,T]$:
\begin{equation}\label{yizhiguji1}
\begin{split}
&\|\ome\|_{H_g^{s,\gamma}}\leq 4\|\omega_0\|_{H_g^{s,\gamma}},\\
&\left\|\sum_{|\alpha| \leq 2}\left|(1+y)^{\sigma+\alpha_2}D^\alpha\ome(t)\right|^2\right\|_{L^\ty(\Om)}\leq \frac{1}{\delta^2},\\
&\min_{\Om}(1+y)^\sigma \ome\geq \delta.
\end{split}
\end{equation}
\end{Proposition}

\noindent\textbf{Proof.}
We now choose a lifespan independent of $\epsilon$. By the definition of $\phi(t)$ in \eqref{phit} and \eqref{BL}, we obtain
\begin{equation*}
\phi(t)\leq C_{s,\gamma,\sigma,n,\delta}\big(1+\sum_{l=0}^{\frac{s}{2}+1}\|\partial_t^l U\|_{H^{s-2l+3}(\T)}\big)^{8(s+1)}
\leq C_{s,\gamma,\sigma,n,\delta} M_1,
\end{equation*}
where $M_1:=\sup\limits_{t\geq0}\big(1+\sum\limits_{l=0}^{\frac{s}{2}+1}\|\partial_t^l U\|_{H^{s-2l+3}(\T)}\big)^{8(s+1)}.$ Thus, setting
\begin{equation*}
T_1:=\min\left\{\frac{\|\om_0\|_{H^{s, \gamma}_g}^2}{C_{s,\gamma,\sigma,n,\delta} M_1},
\frac{1}{4^{s+1}(2s+1)C_{s,\gamma,\sigma,n,\delta}\|\om_0\|_{H_g^{s,\gamma}}^{4s+2}}
\right\},
\end{equation*}
\eqref{omeguji} implies that estimate $\eqref{yizhiguji1}_1$ holds for any $t\in[0,T_1]$.

Clearly,
\begin{equation*}
A_\epsilon(t)\leq 4\|\omega_0\|_{H_g^{s,\gamma}}
,~
G_\epsilon(t)\leq 4\|\omega_0\|_{H_g^{s,\gamma}}+M_1:=M_2
~\text{and}~
\Lambda_\epsilon(t)\leq C_{s,\gamma,\sigma,n,\delta}\bigl(1+M_2\bigr).
\end{equation*}
Taking \begin{equation*}
T_2:=\min\left\{T_1,\frac{1}{256\delta^2C_{s,\gamma,\sigma,n,\delta} (1+\|\omega_0\|_{H_g^{s,\gamma}})\|\omega_0\|_{H_g^{s,\gamma}}^2},
\frac{\ln 2}{C_{s,\gamma,\sigma,n,\delta}\bigl(1+M_2\bigr)}
\right\},
\end{equation*}
 and combining the initial data assumption
 \begin{equation*}
\sum_{|\alpha| \leq 2}\left|(1+y)^{\sigma+\alpha_2}D^\alpha\om_0\right|^2\leq \frac{1}{4\de^2},
\end{equation*}
with \eqref{It-growth}, we conclude that $\eqref{yizhiguji1}_2$ holds for any $t\in[0,T_2]$.

Choosing
 \begin{equation*}
T_3:=\min\left\{T_1,\frac{1}{6C_{s,\gamma,\sigma,n,\delta}(1+M_2)},
\frac{\ln 2}{C_{s,\gamma,\sigma,n,\delta}\bigl(1+M_2\bigr)},
\frac{\delta}{64C_{s,\gamma,\sigma,n,\delta}(1+\|\omega_0\|_{H_g^{s,\gamma}})\|\omega_0\|_{H_g^{s,\gamma}}}
\right\},
\end{equation*}
we deduce from \eqref{wxiajie} that $\eqref{yizhiguji1}_3$ holds for any $t\in[0,T_3]$.\\
In summary, if
$T=\min\{T_1,T_2,T_3\},$
then the uniform estimates \eqref{yizhiguji1} hold for every \(t\in[0,T]\).

\hfill$\square$

\subsection{Compactness and passage to the limit}

The almost equivalence relation \eqref{chabuduo} and uniform weighted $H^s$ estimate \eqref{yizhiguji1} gives
\begin{equation}\label{supome}
\sup_{0\leq t\leq T}
\bigl(\|\omega_\epsilon(t)\|_{H^{s,\gamma}}
+\|u_\epsilon(t)-U(t)\|_{H^{s,\gamma-1}}\bigr)
\leq C_{s,\ga,\si,\de}\bigl(4\|\omega_0\|_{H_g^{s,\gamma}}
+\sup_{0\leq t\leq T}\|\px^s U\|_{L^2(\T)}\bigr)< +\ty.
\end{equation}
By \eqref{feiniu2-1}, \eqref{wodu} and \eqref{supome}, we have
\begin{equation}\label{shijiadao}
\|\pt\omega_\epsilon\|_{L^\ty(0,T;H^{s-2,\gamma})}
+\|\pt(u_\epsilon-U)\|_{L^\ty(0,T;H^{s-2,\gamma-1})}\leq C.
\end{equation}
By the Aubin-Lions lemma and the compact embedding
$H^s\hookrightarrow\hookrightarrow H^{s'}_{loc}$,
there exist a sequence $\epsilon_k\to0^+$  and functions $\omega$ and $u$ such that, for any $s'<s$, the following convergence properties hold:
\begin{equation}\label{shoulian}
\begin{split}
\omega_{\epsilon_k}&\stackrel{*}{\rightharpoonup}\omega
\quad\text{in }L^\ty([0,T];H^{s,\gamma}),\\
\omega_{\epsilon_k}&\rightarrow\omega
\quad\text{in }C([0,T];H^{s'}_{loc}),\\
u_{\epsilon_k}-U&\stackrel{*}{\rightharpoonup}u-U
\quad\text{in }L^\ty([0,T];H^{s,\gamma-1}),\\
u_{\epsilon_k}&\rightarrow u
\quad\text{in }C([0,T];H^{s'}_{loc}),
\end{split}
\end{equation}
where
\begin{equation*}
\begin{split}
&\omega\in L^\ty([0,T];H^{s,\gamma})\cap\bigcap_{s'<s}C([0,T];H^{s'}_{loc}),\\
&u-U\in L^\ty([0,T];H^{s,\gamma-1})\cap\bigcap_{s'<s}C([0,T];H^{s'}_{loc}).
\end{split}
\end{equation*}
Taking $s'>2$, the local uniform convergence of $\partial_xu_{\varepsilon_k}$ also implies the pointwise convergence of \begin{equation*}
v_{\epsilon_k}\rightarrow
v:=-\int_0^y\px u(t,x,z)\,dz, ~~\text{as}~\epsilon_k\to0^+.
\end{equation*}
On the other hand, from \eqref{shoulian},
\begin{equation*}
\omega_{\varepsilon_k}^n\rightarrow \omega^n ~\text{in }~
C\left([0,T];H^{s'}_{loc}\right)~\text{and }~
\partial_y^2\left(\omega_{\varepsilon_k}^n\right)
\rightarrow\partial_y^2(\omega^n)
~\text{ in }~
C\left([0,T];H^{s'-2}_{loc}\right).
\end{equation*}
Passing to the limit in \eqref{wodu} yields
\begin{equation*}
\pt\omega+u\px\omega+v\py\omega-\py^2(\omega^n)=0.
\end{equation*}
The trace convergence and \eqref{BL} give
\begin{equation*}
\py(\omega^n)|_{y=0}=n\omega^{n-1}\py\omega|_{y=0}=\px p.
\end{equation*}
The initial condition follows from time-strong local convergence. The bounds
\eqref{yizhiguji1} pass to the
limit, and lower semicontinuity gives
$\omega\in L^\ty([0,T];H^{s,\gamma}_{\sigma,\delta})$.

To recover the matching condition \eqref{feiniu2}, by Lebesgue's dominated convergence theorem, we get
\begin{equation*}
\int_0^\infty\omega \,dy=\lim_{\epsilon_k\to0^+}\int_0^\infty\omega_{\epsilon_k}\,dy=U(t,x)=U.
\end{equation*}
This proves existence.

\subsection{An $L^2$ comparison estimate and uniqueness}

Let $(u_i,v_i,\omega_i)$, $i=1,2$, be two solutions with the same data. Set
\begin{equation*}
\bar u=u_1-u_2,\quad \bar v=v_1-v_2,\quad
\bar\omega=\omega_1-\omega_2,\quad
a_2=\frac{\py\omega_2}{\omega_2}.
\end{equation*}
Linearize the diffusion exactly by
\begin{equation}\label{zuocha}
\omega_1^n-\omega_2^n=\mu\bar\omega,\qquad
\mu=n\int_0^1(\omega_2+\theta\bar\omega)^{n-1}\,d\theta>0.
\end{equation}
With $\mathcal{L}_1=\pt+u_1\px+v_1\py$, subtraction gives
\begin{equation}\label{L1u}
\begin{split}
\mathcal{L}_1\bar u-\py(\mu\bar\omega)&=-\bar u\px u_2-\bar v\omega_2,\\
\mathcal{L}_1\bar\omega-\py^2(\mu\bar\omega)
&=-\bar u\px\omega_2-\bar v\py\omega_2.
\end{split}
\end{equation}
Introduce, as in \cite[Section 6.2]{NM},
\begin{equation}\label{barg}
\bar g=\bar\omega-a_2\bar u
=\omega_2\py\left(\frac{\bar u}{\omega_2}\right).
\end{equation}
Subtracting $a_2$ times the first equation in \eqref{L1u} from
the second cancels both $\bar v$ terms and yields
\begin{equation}\label{L1g}
\mathcal{L}_1\bar g-(\py^2-a_2\py)(\mu\bar\omega)=\mathcal R_2\bar u,\qquad
\mathcal R_2=-\px\omega_2+a_2\px u_2-\mathcal{L}_1a_2.
\end{equation}

\begin{Lemma}
\label{lem:difference-hardy1}
There exists a constant $C_{\sigma,\delta}$, depending only on $\sigma$ and $\delta$, such that
\begin{equation}\label{difference-hardy-estimate}
\left\|\frac{\bar u}{1+y}\right\|_{L^2}
+\|\bar\omega\|_{L^2}\leq C_{\sigma,\delta}\|\bar g\|_{L^2}.
\end{equation}
\end{Lemma}

\noindent\textbf{Proof.}
Since $\bar u|_{y=0}=0$, \eqref{barg} gives
\begin{equation*}
\frac{\bar u(y)}{\omega_2(y)}
=\int_0^y\frac{\bar g(z)}{\omega_2(z)}\,dz.
\end{equation*}
The bound
$\delta(1+y)^{-\sigma}\leq\omega_2\leq
\delta^{-1}(1+y)^{-\sigma}$ and weighted Hardy's inequality yield the first
term. Since $|(1+y)a_2|\leq C_\de$, the identity
$\bar\omega=\bar g+a_2\bar u$ gives the second. \hfill$\square$

The $H^{s,\gamma}_{\sigma,\delta}$ bounds imply
\begin{equation}\label{weiyixishu}
\begin{split}
&|(1+y)a_2|+|(1+y)^2\py a_2|+|(1+y)\px a_2|
+|(1+y)\mathcal R_2|\leq C_{s,\gamma,\sigma,n,\delta},\\
&0<\mu\leq C_{s,\gamma,\sigma,n,\delta}(1+y)^{-\sigma(n-1)},\qquad
|\py\mu|\leq C_{s,\gamma,\sigma,n,\delta}\frac{\mu}{1+y}.
\end{split}
\end{equation}
Differentiating \eqref{zuocha} proves the second line;
the lower bound for $\omega_i$ controls every reciprocal factor.

Similarly to \cite[Proposition 6.4]{NM}, we obtain the following conclusion:
\begin{Proposition}
\label{prop:L2-comparison-power}
For $0\leq t\leq T$,
\begin{equation}\label{L2barg}
\|\bar g(t)\|_{L^2}^2
+\int_0^t\!\int_\Om\mu|\py\bar g|^2\,dx\,dy\,d\tau
\leq \|\bar g(0)\|_{L^2}^2
+C_*\int_0^t\|\bar g(\tau)\|_{L^2}^2\,d\tau,
\end{equation}
where $C_*>0$ depends only on $s,\gamma,\sigma,n,\delta,U$ and $\|(\omega_1,\omega_2)\|_{L^\infty(0,T;H^{s,\gamma})}$.
\end{Proposition}

\noindent\textbf{Proof.}
Since $\bar\omega=\bar g+a_2\bar u$,
\begin{equation}\label{pymu}
\py(\mu\bar\omega)
=\mu\py\bar g+(\py\mu+\mu a_2)\bar g
+\{\py(\mu a_2)+\mu a_2^2\}\bar u.
\end{equation}
It follows readily from the boundary conditions that
\begin{equation}\label{difference-flux-boundary}
\py(\mu\bar\omega)|_{y=0}
=\py(\omega_1^n-\omega_2^n)|_{y=0}=0.
\end{equation}
Multiplying \eqref{L1g} by $\bar g$ and integrating over $\Omega$, we observe that the transport term gives no boundary contribution, owing to $\px u_1+\py v_1=0$ and $v_1|_{y=0}=0$ . Integration by parts, together with \eqref{difference-flux-boundary}, yields the dissipation term
$$\int_\Om\mu|\py\bar g|^2\ dxdy.$$
Furthermore, by \eqref{weiyixishu}, \eqref{pymu}, Cauchy inequality, and Lemma \ref{lem:difference-hardy1},  the remaining terms satisfy
\begin{equation*}
\begin{split}
&\left|\int_\Om\left[\py\bar g\{(\py\mu+\mu a_2)\bar g
+(\py(\mu a_2)+\mu a_2^2)\bar u\}
+a_2\bar g\py(\mu\bar\omega)+\mathcal R_2\bar u\bar g\right]\right|\\
\leq&\frac12\int_\Om\mu|\py\bar g|^2
+C_{s,\gamma,\sigma,n,\delta}\|\bar g\|_{L^2}^2.
\end{split}
\end{equation*}
Absorb the dissipation term and integrate in time. On the unbounded strip, first insert a cutoff $\chi(y/R)$; its errors vanish as $R\to\infty$ by dominated convergence, as in \cite[(6.18)-(6.21)]{NM}. This proves \eqref{L2barg}. \hfill$\square$

In view of $\bar g(0)=0$, \eqref{L2barg} and Gronwall's inequality imply that $\bar g\equiv0$.
Equation \eqref{barg} then implies $\bar u=q(t,x)\omega_2$.
Since $\bar u|_{y=0}=0$ and $\omega_2|_{y=0}>0$, it follows that $q=0$ , which implies $u_1=u_2$.
The divergence condition and $v_i|_{y=0}=0$ give $v_1=v_2$, hence $\omega_1=\omega_2$.

Finally, \eqref{shijiadao} implies weak continuity in the top
weighted space and strong continuity in lower local Sobolev spaces. Together
with the energy inequality at $t=0$, this proves the time continuity stated in
Theorem \ref{zhuding}.


\subsection*{Acknowledgements}
Zhonger Wu is supported by National Natural Science Foundation of China Grant 12601431 and STU Scientific Research Initiation Grant NTF25028T.

\subsection*{Competing interests}

This work does not have any conflicts of interest.


\begin{thebibliography}{99}

\bibitem{RA} R. Alexandre, Y. G. Wang, C. J. Xu and T. Yang, Well-posedness of the Prandtl equation in Sobolev spaces,
{\it J. Amer. Math. Soc.}, {\bf 28} (2015), 745-784.

\bibitem{JA} J. Arbel, Fa\`a di Bruno's note on eponymous formula, trilingual version, arXiv:1612.05393.



\bibitem{DC} D. Chen, Y. Wang and Z. Zhang, Well-posedness of the linearized Prandtl equation around a non-monotonic shear flow,
 {\it Ann. Inst. H. Poincar$\mathrm{\acute{e}}$ Anal. Non Lin$\mathrm{\acute{e}}$aire}, {\bf 35} (2018), 1119-1142.


\bibitem{HD} H. Dietert and D. G$\mathrm{\acute{e}}$rard-Varet, Well-posedness of the Prandtl system without any structural assumption,
 {\it Ann. PDE}, {\bf 5} (2019), no. 1, Paper No. 8, 51 pp.

\bibitem{DG2010} D. G$\mathrm{\acute{e}}$rard-Varet and E. Dormy, On the ill-posedness of the Prandtl equation,
{\it J. Amer. Math. Soc.}, {\bf 23} (2010), 591-609.

\bibitem{DG2015} D. G$\mathrm{\acute{e}}$rard-Varet and N. Masmoudi, Well-posedness for the Prandtl system without analyticity or monotonicity,
{\it Ann. Sci. $\mathrm{\acute{E}}$cole Norm. Sup. (4)}, {\bf 48} (2015), 1273-1325.

\bibitem{MI} M. Ignatova and V. Vicol, Almost global existence for the Prandtl boundary layer system,
 {\it Arch. Ration. Mech. Anal.}, {\bf 220} (2016),  809-848.


\bibitem{WXLMY} W. X. Li, N. Masmoudi  and T. Yang, Well-posedness in Gevrey function space for 3D Prandtl system without structural assumption, {\it  Comm. Pure Appl. Math.}, {\bf 75} (2022), 1755-1797.



\bibitem{WXLY2} W. X. Li and T. Yang, Well-posedness in Gevrey function spaces for the Prandtl system with non-degenerate critical points,
 {\it J. Eur. Math. Soc. (JEMS)}, {\bf 22} (2020), 717-775.

\bibitem{WXLY} W. X. Li and T. Yang, Well-posedness of the MHD boundary layer system in Gevrey function space without structural assumption, {\it SIAM J. Math. Anal.}, {\bf 53} (2021), 3236-3264.



\bibitem{CJLARMA} C. J. Liu, Y. G. Wang and T. Yang, On the ill-posedness of the Prandtl equations in three-dimensional space, {\it Arch. Ration. Mech. Anal.}, {\bf 220} (2016), 83-108.

\bibitem{CJLADV} C. J. Liu, Y. G. Wang and T. Yang, A well-posedness theory for the Prandtl equations in three space variables, {\it Adv. Math.}, {\bf 308} (2017), 1074-1126.


\bibitem{CJLCPAM} C. J. Liu, F. Xie and T. Yang, MHD boundary layers theory in Sobolev spaces without monotonicity I: Well-posedness theory,
 {\it Comm. Pure Appl. Math.}, {\bf 72} (2019), 63-121.


\bibitem{MCL} M. C. Lombardo, M. Cannone and M. Sammartino, Well-posedness of the boundary layer system,
{\it SIAM J. Math. Anal.}, {\bf 35} (2003), 987-1004.


\bibitem{NM} N. Masmoudi and T. K. Wong, Local-in-time existence and uniqueness of solutions to the Prandtl system by energy methods, {\it Comm. Pure Appl. Math.}, {\bf 68} (2015), 1683-1741.

\bibitem{OAO} O. A. Oleinik and V. N. Samokhin, {\it Mathematical Models in Boundary Layer Theory}, Appl. Math. Math. Comput. {\bf 15}, Chapman \& Hall/CRC, Boca Raton, FL, 1999.

\bibitem{MP} M. Paicu and P. Zhang,  Global existence and the decay of solutions to the Prandtl system with small analytic data,
 {\it Arch. Ration. Mech. Anal.}, {\bf 241} (2021), 403-446.


\bibitem{Said} E. M. Said,  Sur des propri\'{e}t\'{e}s de certains mod\`{e}les en m\'{e}canique des fluides pr\`{e}s d'une surface, {\it Normandie Universit\'{e}}, 2022.

\bibitem{MS} M. Sammartino and R. E. Caflisch, Zero viscosity limit for analytic solutions, of the Navier-Stokes equation on a half-space. I. Existence for Euler and Prandtl system, {\it Comm. Math. Phys.}, {\bf 192} (1998), 433-461.

\bibitem{ZT} Z. Tan and Z. E. Wu, Global small solutions of MHD boundary layer system in Gevrey function space, {\it J. Differential Equations}, {\bf 366} (2023), 444-517.

\bibitem{ZT2} Z. Tan, Z. E. Wu and M. X. Zhang, Gevrey well-posedness of the thermal boundary layer equation. {\it Commun. Math. Sci.}, {\bf 23} (2025), 1319-1355.

\bibitem{CW} C. Wang, Y. X. Wang and P. Zhang, On the global small solution of 2-D Prandtl system with initial data in the optimal Gevrey class, {\it Adv. Math.}, {\bf 440} (2024), 109517.

\bibitem{YGW1} Y.G. Wang and S.Y. Zhu,
Blowup of solutions to the thermal boundary layer problem in two-dimensional incompressible heat conducting flow,
{\it Commun. Pure Appl. Anal.}, {\bf 19} (2020), 3233-3244.

\bibitem{YGW2} Y.G. Wang and S.Y. Zhu,
Well-posedness of thermal boundary layer equation in two-dimensional incompressible heat conducting flow with analytic datum,
{\it Math. Methods Appl. Sci.}, {\bf 43} (2020), 4683-4716.

\bibitem{ZEW} Z. E. Wu, A well-posedness theory for the MHD boundary layer equations in three space variables, {\it J. Differential Equations}, {\bf 447} (2025), Paper No. 113650, 37 pp.

\bibitem{WZE} Z. E. Wu and Z. Tan, Local well-posedness of the boundary layer for a pseudo-plastic fluid by energy methods, {\it arXiv:2608.05693}, 2026.

\bibitem{ZPX} Z. P. Xin and L. Q. Zhang, On the global existence of solutions to the Prandtl's system, {\it Adv. Math.}, {\bf 181} (2004), 88-133.

\bibitem{CJX} C. J. Xu and X. Zhang, Long time well-posedness of the Prandtl equations in Sobolev space, {\it J. Differential Equations} {\bf 263} (2017), 8749-8803.


\bibitem{JWZ} J. W. Zhang, Solutions for the problem of boundary layer formation in a pseudo-plastic fluid: the case of gradual acceleration, {\it J. Math. Anal. Appl.}, {\bf 328} (2007), 220-244.

\bibitem{PZ} P. Zhang and Z. F. Zhang, Long time well-posedness of Prandtl system with small and analytic initial data, {\it J. Funct. Anal.}, {\bf 270} (2016), 2591-2615.

\bibitem{ZW} M. X. Zhang and Z. E. Wu, Linear ill-posedness of three-dimensional MHD boundary layer equations,
{\it arXiv:2609.14532} (2026).


\end{thebibliography}
\end{document}